\documentclass[12pt,oneside]{amsart}
\usepackage{stmaryrd}
\usepackage{hyperref}
\usepackage{amsrefs}
\usepackage{bm}
\usepackage{color}
\newtheorem{theorem}{Theorem}[section]
\newtheorem{lemma}[theorem]{Lemma}

\theoremstyle{definition}
\newtheorem{definition}[theorem]{Definition}

\numberwithin{equation}{section}

\newcommand {\mat} [1] {\left[\begin{array}{#1}}
\newcommand {\rix}     {\end{array}\right]}

\def \per{{\rm per}}

\def\N{{\mathbb{N}}}
\def\R{{\mathbb{R}}}
\def\Z{{\mathbb{Z}}}

\newcommand{\sfX}{{\mathsf X}}

\newcommand{\sfH}{{\mathsf H}}
\newcommand{\sfV}{{\mathsf V}}

\newcommand{\sfF}{{\mathsf F}}
\newcommand{\sfd}{{\mathsf d}}
\newcommand{\sfp}{{\mathsf p}}

\begin{document}
 
\title{A Hamilton-Jacobi PDE
associated with hydrodynamic fluctuations from a nonlinear diffusion equation}

\def \non{{\nonumber}}

\author{Jin Feng, Toshio Mikami, Johannes Zimmer}
\address{
\newline
Mathematics Department \\
University of Kansas \\
Lawrence, KS 66045, USA.
\newline
Department of Mathematics \\
 Tsuda University\\
 2-1-1 Tsuda-machi \\
  Kodaira, Tokyo 187-8577, Japan.
  \newline
Technical University of Munich \\ Department of Mathematics \\ 85748 Garching, Germany.
}                                              
 
%\subjclass[2010]{Primary 49LXX}

\date{\today} 
\keywords{Hamilton-Jacobi PDEs in the space of probability measures, Porous medium equations with control, Multi-scale convergence of Hamiltonians in the space of probability measures, Large deviations.}  
 \thanks{Jin Feng's work was supported in part by US NSF Grant No. DMS-1440140 while he was in residence at MSRI - Berkeley, California, during fall 2018; and in part by a Simons Visiting Professorship to Mathematisches Forschungsinstitut Oberwolfach, Germany and to Technical University of Eindhoven, the Netherlands in November 2017,  and in part by LABEX MILYON (ANR-10-LABX-0070) of Universit{\'e} de Lyon, within the program ``Investissements d'Avenir" (ANR-11-IDEX-0007) operated by French National Research Agency (ANR). He thanks Albert Fathi for helpful discussions on weak KAM theory on numerous occasions. Toshio
 Mikami's work was supported by JSPS KAKENHI Grant Numbers JP26400136 and JP16H03948 from Japan. Johannes Zimmer's work was
 supported through a Royal Society Wolfson Research Merit Award.}

\begin{abstract}
  We study a class of Hamilton-Jacobi partial differential equations in the space of probability measures. In the first part of
  this paper, we prove comparison principles (implying uniqueness) for this class. In the second part, we establish the existence
  of a solution and give a representation using a family of partial differential equations with control. A large part of our
  analysis exploits special structures of the Hamiltonian, which might look mysterious at first sight. However, we show that this
  Hamiltonian structure arises naturally as limit of Hamiltonians of microscopical models. Indeed, in the third part of this
  paper, we informally derive the Hamiltonian studied before, in a context of fluctuation theory on the hydrodynamic scale.  The
  analysis is carried out for a specific model of stochastic interacting particles in gas kinetics, namely a version of the
  Carleman model. We use a two-scale averaging method on Hamiltonians defined in the space of probability measures to derive the
  limiting Hamiltonian.

\end{abstract}
\maketitle

\section{Introduction}
\label{sec:Introduction}
\subsection{An overview}
\label{sec:An-overview}

We develop a variational approach to derive macroscopic hydrodynamic equations from particle models. Within this broad context,
this article studies a new class of Hamilton-Jacobi equations in the space of probability measures. Specifically, we study the
convergence of Hamiltonians and establish well-posedness for a limiting Hamilton-Jacobi equation using a model problem in
statistical physics. However, our main interest is to develop a new scale-bridging methodology, rather than to advance the
understanding of the specific model.

The theory of hydrodynamic limits can be divided into deterministic and stochastic theories. The goal of the deterministic
approach is to derive continuum-level conservation laws as scaling limit of particle motions on the microscopic level, which
satisfy Hamiltonian ordinary differential equations.  At a key step, this requires the use of ergodic theory for Hamiltonian
dynamical systems. Such theory, at a level to be applied successfully, is not readily available for a wide range of problems.
Consequently, the program of rigorously deriving hydrodynamic limits from deterministic models remains a huge
challenge~\cite[Chapter 1]{Spohn1991a}.  (See, however, the work of Dolgopyat and Liverani~\cite{Dolgopyat2011a} on weakly
interacting geodesic flows on manifolds of negative curvature, which made a progress in this direction.) This topic has a long
history; indeed, the passage from atomistic to continuum models is mentioned by Hilbert in his explanation of his sixth problem
(see~\cite{Gorban2018a} for a recent review). Stochastic hydrodynamics, on the other hand, relies on probabilistic interacting
particle models, and the program has been more successful. Conceptually, one often thinks of these stochastic models as
regularizations of underlying deterministic models. Usually, the interpretation is that, at appropriate intermediate scale, a
class of particles develop contact with a fast oscillating environment, which is modeled by stochastic noises.  With randomness
injected in the particle motions, we can invoke probabilistic ergodic theorems. Probabilistic ergodic theory is much tamer than
its counterpart for deterministic Hamiltonian dynamics. Hence the program can be carried out with rigor for a wide variety of
problems. In many cases in stochastic hydrodynamics, convergence towards a macroscopic equation can be seen as a sophisticated
version of the law of large numbers. Large deviation theory describes fluctuations around this macroscopic equation as limit, and
offers finer information in form of a variational structure through a so-called rate function. In this context, the rate function
automatically describes the hydrodynamic (macroscopic) limit equation as minimizer. A Hamiltonian formulation of large deviation
theory for Markov processes has been developed by Feng and Kurtz~\cite{FK06}. We will follow the method in Section~\ref{DerH} of
this paper. The literature on the large deviation approach to stochastic hydrodynamics is so extensive that we do not try to
review it (e.g., Spohn~\cite{Spohn1991a}, Kipnis and Landim~\cite{Kipnis1999a}). Instead, we only mention the seminal work of Guo,
Papanicolaou and Varadhan~\cite{GPV88}, which introduced a major novel technique, known as block-averaging/replacement-lemma, to
handle multi-scale convergence. With refinements and variants, this block-averaging method remains the standard of the subject to
the present day.

Despite the power and beauty of block-averaging and the replacement lemma, this method relies critically on probabilistic ergodic
theories, and is thus largely restricted to stochastic models. In Section~\ref{DerH}, we introduce a functional-analytic approach
for scale-bridging which is different to the one of Guo, Papanicolaou and Varadhan~\cite{GPV88}.  We will still consider a
stochastic model in this paper, but our method is developed with a view to be applicable in the deterministic program as well. Our
approach takes inspiration from the Aubry-Mather theory for deterministic Hamiltonian dynamical systems, and, more directly, from
another deeply linked topic known as the weak KAM (Kolmogorov-Arnold-Moser) theory (e.g., Fathi~\cite{FaBook}). We will derive and
study an infinite particle version of a specific weak KAM problem.  By infinite particle version, we mean a Hamiltonian describing
the motion of infinitely many particles. Hence it is defined in space of probability measures.  To avoid raising false hopes, we
stress again that the model we use still incorporates randomness. However, in principle, the Hamilton-Jacobi part of our program
is not fundamentally tied to probabilistic ergodic theory. We choose a stochastic model problem to test ideas, as the infinite
particle version of weak KAM theory in this case becomes simple and directly solvable. At the present time, a general infinite
particle version of weak KAM theory does not exist. This is in sharp contrast with the very well-developed theories for finite
particle versions of deterministic Hamiltonian dynamics defined in finite-dimensional compact domains. Therefore, it is useful for
us to focus on a more modest goal in this article --- we are content with developing a method of studying problems where other
(probabilistic) approaches may apply in principle, though we are not aware of such results for Carleman-type models. To put things
in perspective, we hope the study of this paper will reveal the relevance and importance of studying Hamilton-Jacobi partial
differential equations in the space of probability measures, and open up many possibilities on this type of problems in the
future.

In light of the preceding discussion, we study the convergence of Hamiltonians arising from the large deviation setting in
Section~\ref{DerH} using a Hamilton-Jacobi theory developed by Feng and Kurtz~\cite{FK06}. This approach is different from the
usual approach to large deviations, which can be viewed as a Lagrangian technique.

To understand the convergence of Hamilton-Jacobi equations, following a topological compactness-uniqueness strategy, we need to resolve two issues: one is the multi-scale convergence of particle-level Hamiltonians to the continuum-level Hamiltonian; the
other is the uniqueness for a class of abstract Hamilton-Jacobi equations which includes the limiting continuum Hamiltonian. We
settle the first issue in Section~\ref{DerH} in a semi-rigorous manner, and the second issue on Sections~\ref{Sec:CMP}
and~\ref{Sec:Nisio} rigorously.

We first sketch how the second problem --- existence and uniqueness for a class of Hamilton-Jacobi equations --- is addressed in
this paper; this is a question of independent interest, and corresponds to a hard issue for the traditional Lagrangian approach to
hydrodynamic limits and large deviations, namely to match large deviation upper and lower bounds. Essentially, we may not know how regular the paths needs to be in order to approximate the Lagrangian action accurately for all paths. In the Hamiltonian setting
advocated here, this problem can be solved rigorously for the model problem considered in this article.  Indeed, we develop a
method to establish a strong form of uniqueness (the comparison principle) for the macroscopic Hamilton-Jacobi equation.  The
analysis uses techniques from viscosity solution in space of probability measures developed by Feng and Kurtz~\cite{FK06} and Feng and Katsoulakis~\cite{FK09}.  This is a relatively new topic and is a step forward compared to earlier studies initiated by
Crandall and Lions~\cites{CL85,CL86,CL86b,CL90,CL91,CL94,CL94b} on Hamilton-Jacobi equations in infinite dimensions, focusing on Hilbert spaces. Our Hamiltonian has a structure which is closer in spirit to the one studied by Crandall and Lions~\cite{CL91} (but with a nonlinear operator and other subtle differences), to Example~9.35 in Chapter~9 and Section~13.3.3 in Chapter~13 of~\cite{FK06} and to Example~3 of~\cite{FK09}, or Feng and Swiech~\cite{FS13}. It is different in structure than those studied by Gangbo, Nguyen and Tudorascu~\cite{GNT08} and Gangbo and Tudorascu~\cite{GT10} in Wasserstein space; or to those studied with metric analysis techniques by Giga, Hamamuki and Nakayasu~\cite{GHN15}, Ambrosio and Feng~\cite{AF14}, Gangbo and Swiech~\cite{GS15}. The difference is that we have to deal with an unbounded drift term given by a nonlinear operator which we explain next. In~\eqref{H} below, we first informally define the Hamiltonian as $H=H(\rho,\varphi)$ for a probability measure $\rho$ and a smooth test function $\varphi$. This definition contains a nonlinear term $\partial_{xx}^2 \log \rho$.  \emph{A  priori}, $\rho$ is just a probability measure, thus even the definition of this expression is a problem. If the probability measure $\rho$ is zero on a set of positive Lebesgue measure, $\log \rho$ cannot even be defined in a distributional sense. In
addition, our notion of a solution is more than the $B$-continuity studied in~\cite{CL91}. Namely, for the large deviation theory,
we need the solution to be continuous in the metric topology on the space (and it in fact is). This can be established through an
\emph{a posteriori} estimate technique which we introduce in Lemma~\ref{visExt}.  Regarding the possible singularity of
$\partial_{xx}^2 \log \rho$, in rigorous treatment, we will use non-smooth test functions $\varphi$ in the Hamiltonian to
compensate for the possible loss of distributional derivatives of the term $\log \rho$. This is essentially a renormalization
idea, in the sense that we rewrite the equation in appropriately chosen new coordinates to ``tame'' the singularities.
Section~\ref{Sec:CMP} explores a hidden controlled gradient flow structure in the Hamiltonian. Using a theorem by Feng and
Katsoulakis~\cite{FK09} and the regularization technique in Lemma~\ref{visExt}, we establish a comparison principle for the
Hamilton-Jacobi equation with the Hamiltonian given by~\eqref{H}. For the existence of a solution, we argue in the Lagrangian
picture. Here, the problem translates into a nonlinear parabolic problem, which again is quite singular. We establish in
Section~\ref{Sec:Nisio} an existence theory using the theory of optimal control and Nisio semigroups, and by deriving some
non-trivial estimates.

We now describe the approach for dealing with the first problem, the multi-scale convergence. This is discussed at the end of the
paper (Section~\ref{DerH}) and involves some semi-rigorous arguments. As this part involves a broad spectrum of techniques coming
from different areas of mathematics, a rigorous justification and the description of details are long; 
we postpone them to future studies. The stochastic model we use here is known as the stochastic Carleman model studied by Caprino, De Masi, Presutti and
Pulvirenti~\cite{CDPP89}. This is a fictitious system of interacting stochastic particles describing a two-velocity gas, and leads to the
Carleman equation as the kinetic description. At a different (coarser) level, a hydrodynamic limit theorem has been derived by
Kurtz~\cite{Kurtz73} and then by McKean~\cite{McKean75}. It yields the nonlinear diffusion equation studied in
Section~\ref{Sec:Nisio}. Ideas of  Lions and Toscani~\cite{LT97} to study this equation in terms of the density and the flux
turn out to be useful. Here, following the Hamilton-Jacobi method of~\cite{FK06}, we study large deviations of the stochastic
Carleman equations. We give three heuristic derivations to identify the limit Hamiltonian, which is the one given in~\eqref{H} and
studied in the earlier parts of the paper. We now sketch these three limit identifications. The first is based on a formal weak KAM theory in an infinite-dimensional setting.  The second approach is based on a finite-dimensional weak KAM theory, and the key is reduction due to propagation of chaos. The third derivation uses semiclassical approximations. We remark that our overall aim is to provide new functional-analytic methods for deriving the limiting continuum level Hamiltonian in this hydrodynamic large deviation setting. We turn the issue of multi-scale into one of studying a small-cell Hamiltonian averaging problem in the space of probability measures. In the present case, this can be solved at least formally by the weak KAM theory for Hamiltonian dynamical systems, which is a deterministic method.

Our program combines tools from a variety of sources, notably viscosity solutions in the space of probability measures, optimal
transport, parabolic estimates, optimal control, Markov processes, and Hamiltonian dynamics. We use weak KAM type arguments to
replace stronger versions of ergodic theory in the derivation of limiting Hamiltonian.

\subsection{The setting}
\label{sec:setting}

Let $\mathcal O$ be the one dimensional circle, i.e. the unit interval $[0,1]$ with periodic boundary by identifying $0$ and $1$
as one point.  We denote by $\mathcal P(\mathcal O)$ the space of probability measures on $\mathcal O$ and formally define a
Hamiltonian function on $\mathcal P(\mathcal O) \times C^\infty(\mathcal O)$:
\begin{align}\label{H}
H(\rho,\varphi)  :=   \langle \varphi, \frac12 \partial_{xx}^2 \log \rho \rangle + \frac12 \int_{\mathcal O} |\partial_x \varphi|^2 dx, 
 \quad \forall \rho  \in \mathcal P(\mathcal O), \varphi \in C^\infty(\mathcal O).
\end{align}
We use the word formal because even for probability measures admitting a Lebesgue density
$\rho(dx)=\rho(x) dx \in \mathcal P(\mathcal O)$, as long as $\rho(x)=0$ on a positive Lebesgue measure set of $\mathcal O$, we have $-\log \rho(x) = +\infty$ on this set. In such cases, $\partial_{xx}^2 \log \rho$ cannot be defined as a distribution. Therefore, we will explore special choices of the test functions $\varphi$ which are $\rho$ dependent and possibly non-smooth to compensate for loss of the distribution derivative on the $\log \rho$ term.

We will introduce a number of notations and definitions in Section~\ref{Intro2}. In particular, we denote $\sfX:=\mathcal P(\mathcal O)$ and define a homogeneous negative order Sobolev space $H_{-1}(\mathcal O)$ according to \eqref{defH-1}.  In Section~\ref{Intro2}, we will show that $\sfX$ can be identified as a closed subset of this $H_{-1}(\mathcal O)$.  Hence it is a metric space as well.  With the formal Hamiltonian function~\eqref{H}, we can now proceed to the second step to introduce a
formally defined operator
\begin{align*}
  H f(\rho) := H\big(\rho, \frac{\delta f}{\delta \rho}\big), \quad \forall \rho \in \sfX,
\end{align*}
where the test functions $f$ are only chosen to be very smooth, 
\begin{align}
  \label{D}
  D&:= \big\{ f(\rho)  = \psi(\langle \rho, \varphi_1 \rangle, \ldots, \langle \varphi_k, \rho \rangle ) : \psi \in C^2(\R^k), \\
   &  \qquad \qquad \varphi_i \in C^\infty(\mathcal O), i=1,\ldots, k; k=1,2, \ldots \big\}. \nonumber
\end{align}

In the first part of this paper (Section~\ref{Sec:CMP}), we prove a comparison principle (Theorem~\ref{CMP}) for a Hamilton-Jacobi equation in the space of probability measures. This equation is formally written as
\begin{align}
  \label{HJBmacro}
  f - \alpha H f = h.
\end{align}
In this equation, the function $h$ and the constant $\alpha >0$ are given, and $f$ is a solution. However, making sense of~\eqref{HJBmacro} rigorous is very subtle. Motivated by \emph{a priori} estimates, we make sense of the operator $H$ by introducing two more operators $H_0$ and $H_1$, and interpret equation~\eqref{HJBmacro} as two families of inequalities
\eqref{H0eqn} and~\eqref{H1eqn}, which define sub- and super- viscosity solutions (Definition~\ref{viscDef}). The comparison principle in Theorem~\ref{CMP} compares the sub- and super- solutions of these two (in-)equations.  This result implies in particular that there is at most one function $f$ which is both a sub- as well as a super- solution.

In the second part of the paper (Section~\ref{Sec:Nisio}), we construct solutions by studying the Lagrangian dynamics associated with the Hamiltonian $H$ in~\eqref{H}.  A Legendre dual transform of the formal Hamiltonian gives a Lagrangian function
\begin{align*}
L(\rho, \partial_t \rho):= \sup_{\varphi \in C^\infty(\mathcal O)} \big( \langle \partial_t \rho, \varphi \rangle - H(\rho, \varphi) \big) = \frac12 \Vert \partial_t \rho - \frac12 \partial^2_x \log \rho \Vert_{-1}^2
\end{align*}
(the norm is defined in~\eqref{defH-1norm} below). We define an action on $\mathcal P(\mathcal O)$-valued curves by
\begin{align}\label{ACT}
A_T[\rho(\cdot)]:= \int_0^T L(\rho(t), \partial_t\rho(t)) dt.
\end{align}
One can consider variational problems with this action defined in the space of curves $\rho(\cdot)$, or equivalently, consider a nonlinear partial differential equation with control,
\begin{align}
  \label{CPDE0}
\partial_t \rho(t,x) =\frac12 \partial_{xx}^2 \log \rho(t,x) 
      + \partial_x \eta(t,x), \quad t >0,  x \in \mathcal O,
\end{align}
with $\rho(r,\cdot)$ being the state variable, $\eta(t,\cdot)$ (or equivalently $\partial_t \rho$) being a control, and $A_T$ being a running cost.  We take the control interpretation next and defined a class of admissible control as those satisfying
\begin{align}
  \label{CPDE0fi}
  \qquad \int_0^T \int_{\mathcal O} |\eta(s,x)|^2 dx ds <\infty.
\end{align}
We also define a value function for the above optimal control problem,
\begin{align}
  \label{Resolve}
  R_\alpha h(\rho_0) &:= \limsup_{t \to \infty} 
  \sup \Big\{ \int_0^t e^{-\alpha^{-1} s} \Big( \frac{h(\rho(s))}{\alpha} 
                          - \frac12 \int_{\mathcal O} |\eta(s,x)|^2 dx \Big)ds :  \\
                     &   \qquad \qquad  (\rho(\cdot), \eta(\cdot)) \text{ satisfies~\eqref{CPDE0} and~\eqref{CPDE0fi} with } 
                         \rho(0) = \rho_0 \Big\}.  \nonumber
\end{align}
Then, assuming $h \in C_b(\sfX)$, we show that
\begin{align}
  \label{fhrep}
  f := R_\alpha h
\end{align}
is both a sub-solution to~\eqref{H0eqn} as well as a super-solution to~\eqref{H1eqn} (see Lemma~\ref{Sec:Exist}).  This gives us an existence result for the Hamilton-Jacobi PDE~\eqref{HJBmacro} in the setting we introduced. Hence, by the comparison principle
earlier proved, it is the only solution.

The formal basis for the existence results above comes from an observation that
\begin{align}
  \label{NisG}
  H f(\rho) = \sup_{\eta \in L^2(\mathcal O)} \Big\{ \langle \frac{\delta f}{\delta \rho}, 
  \frac12 \partial_{xx}^2 \log \rho +\partial_x \eta \rangle 
  - \frac12 \int_{\mathcal O} |\eta(x)|^2 dx  \Big\}
\end{align}
We emphasize again that $\log \rho$ may not be a distribution, hence the above variational representation is not rigorous. However, it suggests at least formally that $H$ is a Nisio semigroup generator associated with the family of nonlinear diffusion equations with control~\eqref{CPDE0}.  We also comment that the value function
$R_\alpha h\colon \sfX \mapsto \bar{\R}: = \R \cup \{ \pm \infty\}$ introduced before is well defined for all
$h\colon \sfX \mapsto \bar{\R}$ satisfying
\begin{align}
  \label{hclass}
  \int_0^t e^{-\alpha^{-1} s} h(\rho(s)) ds <+\infty, \quad \forall (\rho(\cdot),\eta(\cdot)) 
  \text{ satisfies~\eqref{CPDE0} and~\eqref{CPDE0fi}} , \quad \forall t>0.  
\end{align}
This includes in particular the class of measurable $h\colon \sfX \mapsto \bar{\R}$ which are bounded from above $\sup_\sfX h<+\infty$.  Additionally, the precise meaning of control equation~\eqref{CPDE0} is given in Definition~\ref{DefCPDE}. We establish existence and some regularities of solutions in Lemmas~\ref{apriori} and~\ref{CPDEsta}. Finally, in Section~\ref{Sec:Nisio} we use that the partial differential equation~\eqref{CPDE0} can also be written as a system in a density-flux $(\rho,j)$-variables
\begin{align} \label{CSYS}
\begin{cases}
\partial_t \rho + \partial_x j =0, \\
2 \rho(j+\eta) + \partial_x \rho=0. 
\end{cases}
\end{align}
This ``change-of-coordinate" turns out to be very useful when we justify the derivation of the Hamiltonian $H$ from microscopic models in the last part of the paper.

The third part of this paper, Section~\ref{DerH}, is, unlike the other parts of the paper, non-rigorous. The purpose of this section is to place results of first two parts of this paper in context of a bigger program, by explaining significance of studying the equation~\eqref{HJBmacro}. Specifically, in Section~\ref{DerH}, we will informally derive the Hamiltonian $H$ given in~\eqref{H} in a context of Hamiltonian convergence using generalized multi-scale averaging techniques (for operators on
functions in the space of probability measures). Our starting point is a stochastic model of microscopically defined particle system in gas kinetics. By a two-scale hydrodynamic rescaling and by taking the number of particles to infinity, the (random) empirical measure of particle number density $\rho_\epsilon$ follows an asymptotic expression
\begin{align}\label{LDPform}
 P ( \rho_\epsilon(\cdot) \in d \rho(\cdot) ) \sim Z_\epsilon^{-1} e^{- \epsilon^{-1} A_T[\rho(\cdot)]} P_0(d \rho(\cdot)),
\end{align}
where the $A_T$ is precisely the action given by~\eqref{ACT} and the $P_0$ is some ambient background reference measure. We justify the above probabilistic limit theorem (known as large deviations) through a Hamilton-Jacobi approach.  For a full exposition on this approach in a rigorous and general context, see Feng and Kurtz~\cite{FK06}. In this general theory, $H$ is
derived from a sequence of Hamiltonians which describe Markov processes given by stochastic interacting particle systems. The rigorous application of the general theory developed in~\cite{FK06} requires to establish a comparison principle for the limiting Hamilton-Jacobi equation given by the $H$. In addition, if we have an optimal control representation of $H$ (such as the
identity~\eqref{NisG}), then we can explicitly identify the right hand side of~\eqref{LDPform} using the action. This is the reason we studied these problems in Sections~\ref{Sec:CMP} (comparison principle) and~\ref{Sec:Nisio} (optimal control problem) in this paper.

To summarize, we derive the Hamiltonian convergence in Section~\ref{DerH} in a non-rigorous manner; we rigorously prove the comparison principle in Section~\ref{Sec:CMP}; and rigorously construct the solution and related the solution with an optimal control problem in Section~\ref{Sec:Nisio}.

 \subsection{Notations and definitions}
\label{Intro2}
Let $\mathcal P(A)$ denote the collection of all probability measures on a set $A$. On a metric space $(E, r)$, we use $B(E)$ to
denote the collection of bounded function on $E$. Further, $C_b(E)$ denotes bounded continuous functions, $UC(E)$ denotes
uniformly continuous functions, $UC_b(E):= UC(E) \cap B(E)$. Finally, $LSC(E)$ (respectively $USC(E)$) denotes
lower-semicontinuous (respectively upper-semicontinuous) functions, which are possibly unbounded.  For a function $h \in UC(E)$,
we denote $\omega_h$ the (minimal) modulus of continuity of $h$ with respect to the metric $r$ on $E$:
 \begin{align*}
   \omega_h(t):= \sup_{r(x,y)\leq t} |h(x) - h(y)|.
\end{align*}

We write $C^\infty(\mathcal O)$ for the collection of infinitely differentiable functions on $\mathcal O$.  For a Schwartz
distribution $m \in \mathcal D^\prime(\mathcal O)$, we define
\begin{align}
  \label{defH-1norm}
  \Vert m \Vert_{-1}:= \sup \Big( \langle m, \varphi \rangle : \varphi \in C^\infty(\mathcal O), 
  \int_{\mathcal O} |\partial_x \varphi|^2 dx \leq 1 \Big).
\end{align}
We denote the homogeneous Sobolev space of negative order
\begin{align}
  \label{defH-1}
  H_{-1}(\mathcal O):= \big\{ m \in \mathcal D^\prime(\mathcal O) : \Vert m \Vert_{-1}<\infty \big\}.
\end{align}
The associated norm has the property that
\begin{align*}
  \Vert m \Vert_{-1}=+\infty, \quad  \forall m \in \mathcal D^\prime(\mathcal O) \text{ such that }\langle m, 1 \rangle \neq 0.
\end{align*}
Hence $H_{-1}(\mathcal O)$ is a subset of distributions that annihilates constants, $\langle m, 1\rangle =0$. In fact, the
following representation holds: for every $m \in H_{-1}(\mathcal O)$, we have
\begin{align*}
  m = \partial_x \eta, \quad \exists \eta \in L^2(\mathcal O).
\end{align*}

Regarding the one dimensional torus $\mathcal O:= \R/ \Z$ as a quotient metric space, we consider a metric $r$ defined by
\begin{align}
  \label{defr}
  r(x,y):=\inf_{k \in \Z}  |x-y-k|.
\end{align}
Let $\rho, \gamma \in \mathcal P(\mathcal O)$. We write
\begin{align}
  \label{defPi}
  \Pi (\rho,\gamma):= \big\{ \bm{\nu} \in \mathcal P(\mathcal O \times \mathcal O) 
  \text{ satisfying } \bm{\nu}(dx, \mathcal O) =\rho(dx),
  \bm{\nu}(\mathcal O, dy) =\gamma(dy) \big\}.
\end{align}
For $p \in (1, \infty)$, let $W_p$ be the Wasserstein order $p$-metric on $\mathcal P(\mathcal O)$:
\begin{align*}
  W_p^p (\rho, \gamma) :=\inf \left\{ \int_{\mathcal O \times \mathcal O} r^p(x,y) \bm{\nu}(dx, dy) 
  : \bm{\nu} \in \Pi(\rho, \gamma) \right\}.
\end{align*}
See Chapter~7 in Ambrosio, Gigli and Sav\'are~\cite{AGS08} or Chapter~7 of Villani~\cite{V03} for properties of this metric.
Next, we claim that
\begin{align}
  \label{PsubHneg}
  \mathcal P(\mathcal O) -1:=  \{ \rho -1 : \rho \in \mathcal P(\mathcal O)\} \subset H_{-1}(\mathcal O).
\end{align}
To see this, we note that on one hand, by the Kantorovich-Rubinstein theorem (e.g., Theorem~1.14 of~\cite{V03}), for every
$\rho, \gamma \in \mathcal P(\mathcal O)$, we have
\begin{align}
  \label{W1Hn1}
  W_1(\rho,\gamma)& = \sup \Big\{ \langle \varphi, \rho - \gamma \rangle : \varphi \in C^\infty(\mathcal O) \text{ satisfying }  
                    \Vert \varphi\Vert_{\rm Lip} \leq 1\Big\} \\
                  &\leq  \sup \Big\{ \langle \varphi, \rho -\gamma \rangle : \varphi \in C^\infty(\mathcal O), 
                    \int_{\mathcal O} |\nabla \varphi|^2 dx \leq 1 \Big\} \nonumber \\
                  & = \Vert \rho - \gamma\Vert_{-1}. \nonumber
\end{align}
On the other hand, by an adaptation of Lemma~4.1 of Mischler and Mouhot~\cite{MM13} to the torus case (see Lemma~\ref{equiDis} in
Appendix~\ref{sec:Append-Some-prop}), there exists a universal constant $C>0$ such that
\begin{align*}
\Vert \rho - \gamma \Vert_{-1} \leq C \sqrt{W_1(\rho,\gamma)}.
\end{align*}
Therefore the topology induced by the metric $\Vert \cdot \Vert_{-1}$ is identical to the usual topology of weak convergence of
probability measure on $\mathcal P(\mathcal O)$.  Since any sequence of elements in $\sfX:=\mathcal P(\mathcal O)$ is tight, we
conclude that $(\sfX, \sfd)$ is a compact metric space with the $\sfd(\rho,\gamma) :=\Vert \rho -\gamma\Vert_{-1}$. In particular,
this argument establishes that~\eqref{PsubHneg} holds.

We define a free energy functional $S\colon \mathcal P(\mathcal O) \mapsto [0,+\infty]$, 
\begin{align}
  \label{freeE}
  S(\rho):= 
  \begin{cases}
    \int_{\mathcal O} \rho(x)  \log \rho(x) dx & \text{ if }  \rho(dx) = \rho(x) dx  \\
    +\infty & \text{ otherwise.}  
\end{cases}
\end{align}
We use convention that $0\log 0 :=0$.  Since this is the relative entropy between $\rho$ and the uniform probability measure $1$
on $\mathcal O$, we have $S(\rho) \geq S(1)=0$.  By the variational representation
\begin{align*}
  S(\rho) = \sup_{\varphi \in C^\infty(\mathcal O)} \big( \langle \varphi, \rho \rangle - \log \int_{\mathcal O} e^{\varphi} dx \big),
\end{align*}
we have $S \in LSC\big(\mathcal P(\mathcal O)\big)$.

We make two formal observations. For the Hamiltonian $H$ in~\eqref{H}, we have
\begin{align*}
  \big( H ( \epsilon S) \big) (\rho) = H\big(\rho, \epsilon \frac{\delta S}{\delta \rho}\big) 
  = -\frac{\epsilon(1-\epsilon)}{2} \int_{\mathcal O} |\partial_x \log \rho|^2 dx \leq 0, \quad \epsilon \in [0,1].
\end{align*}
We introduce an analog of Fisher information in this context, extending the usual definition in optimal mass transport theory, by defining
\begin{align}
  \label{defI}
  I(\rho) &:= \begin{cases}
     \sup_{\xi \in C^\infty(\mathcal O)} \{ 2 \langle \partial_x \xi, \log \rho\rangle 
      - \int_{\mathcal O}|\xi|^2 dx  \} & \text{ if } \rho(dx) = \rho(x) dx, \log \rho \in L^1(\mathcal O) \\
    +\infty & \text{ otherwise}  
    \end{cases} \\
    & =  \begin{cases}
    \int_{\mathcal O} |\partial_x \log \rho|^2 dx & \text{ if } \rho(dx) = \rho(x) dx, 
     \log \rho \in L^1(\mathcal O), \partial_x \log \rho \in L^2(\mathcal O) \\
    +\infty & \text{ otherwise.}
  \end{cases}  \nonumber
\end{align}
We claim that $I \in LSC(\mathcal P(\mathcal O); \bar{\R})$. This claim can be verified by the following observations. Let
$\rho_n$ be a sequence such that $\sup_n I(\rho_n) <\infty$.  First, by one-dimensional Sobolev inequalities,
$\sup_n \Vert \log \rho_n \Vert_{L^\infty} <\infty$ and $\rho_n \in C(\mathcal O)$. 
In fact, $\{\rho_n\}_n$ has a uniform modulus of continuity, and is hence relatively compact in $C(\mathcal O)$. This implies relative compactness of the $\{ \log \rho_n \}_n$
in $C(\mathcal O)$. Secondly, by variational formula, for all $\rho$ such that $\log \rho$ is bounded:
\begin{align*}
  \int_{\mathcal O} |\partial_x \log \rho|^2 dx= \sup\{ 2 \langle \partial_x \xi, \log \rho\rangle - \int_{\mathcal O}|\xi|^2 dx 
  : \forall \xi \in C^\infty(\mathcal O)\}.
\end{align*}
We note that $\varphi \mapsto H(\rho, \varphi)$ is convex in the usual sense.

Next, extending the existing theory of viscosity solution for Hamilton-Jacobi equations in abstract metric spaces, we define a
notion of solutions that will be used in this paper.
 \begin{definition}
   \label{viscDef}
   Let $(E,r)$ be an arbitrary compact metric space. A function $\overline{f} \colon E \mapsto \R$ is a sub-solution
   to~\eqref{HJBmacro} if for every $f_0 \in D(H)$, \emph{there exists} $x_0 \in E$ such that
   \begin{align*}
     \big( \overline{f} - f_0 \big)(x_0) = \sup_E \big(\overline{f} - f_0\big),
   \end{align*}
  we have  
   \begin{align*}
     \alpha^{-1}\big( \overline{f}(x_0) - h(x_0) \big) \leq H f_0(x_0).
   \end{align*}
   Similarly, a function $\underline{f} \colon E \mapsto \R$ is a super-solution to~\eqref{HJBmacro} if for every $f_1 \in D(H)$,
   \emph{there exists} $y_0 \in E$ such that
   \begin{align*}
     \big( f_1 - \underline{f} \big)(y_0) = \sup_E \big(f_1 - \underline{f}\big),
   \end{align*}
   we have
   \begin{align*}
     \alpha^{-1}\big(\underline{f}(y_0) - h(y_0)\big) \geq H f_1(y_0).
   \end{align*}
 \end{definition}
 Note this definition differs from the usual one; indeed, if in Definition~\ref{viscDef} ``there exist'' is replaced by ``for
 every'', then $\bar{f}$ and $\underline{f}$ are strong sub- and super- solution. A function $f$ is a (strong) solution if it is
 both a (strong) sub-solution and a (strong) super-solution.

\subsection{Towards a well-posedness theory: two more Hamiltonians}
\label{sec:Two-more-Hamilt}

We now establish some useful properties of the Hamiltonian $H$ in~\eqref{H}. In particular, we now give a heuristic argument why a
comparison principle can be expected formally for~\eqref{HJBmacro}. Since we will use non-smooth test functions which do not fall
inside the domain of very smooth functions $D(H)$ given by~\eqref{D}, it is not at all trivial to make this result rigorous; to
this behalf we introduce two more Hamiltonians related to the one in~\eqref{H}, as motivated through the formal calculations we
now give. Throughout the paper, we denote $\sfd(\rho,\gamma) := \Vert \rho - \gamma \Vert_{-1}$.

Let $k \in \R_+$, at least formally for $\rho(dx) = \rho(x) dx$ and $\gamma(dx)= \gamma(x) dx$, we have
\begin{align}
  \label{eq:H-formal-sup}
  H \big( \frac{k}{2} \sfd^2(\cdot, \gamma)\big)  (\rho) 
   &= - \frac{k}{2} \int_{\mathcal O} (\rho - \gamma) \log \rho dx 
                                                             + \frac{k^2}{2} \Vert \rho - \gamma\Vert_{-1}^2, \\
                                                         & \leq  \frac{k}{2} \big( S(\gamma) - S(\rho)\big)
                                                                  + \frac{k^2}{2} \Vert \rho - \gamma\Vert_{-1}^2. \nonumber
\end{align}
The last inequality follows since by Jensen's inequality
\begin{align*}
  \int_{\mathcal O} \gamma \log \rho dx = \int_{\mathcal O} \gamma \log  \frac{\rho}{\gamma} dx 
  + \int_{\mathcal O} \gamma \log \gamma dx \leq 
  \log \int_{\mathcal O} \rho dx + S(\gamma).
 \end{align*}
Similarly, we have 
\begin{align}
  \label{eq:H-formal-sub}
  H \big(- \frac{k}{2} \sfd^2(\rho, \cdot)\big)  (\gamma) 
  &=  \frac{k}{2} \int_{\mathcal O} (\gamma - \rho) \log \gamma dx + \frac{k^2}{2} \Vert \rho - \gamma\Vert_{-1}^2, \\
  & \geq  \frac{k}{2} \big( S(\gamma) - S(\rho)\big)+ \frac{k^2}{2} \Vert \rho - \gamma\Vert_{-1}^2. \nonumber
\end{align}
In particular,
\begin{align}
  \label{Hcmp}
  H \big( \frac{k}{2}\sfd^2(\cdot, \gamma)\big)  (\rho) - H \big(- \frac{k}{2} \sfd^2(\rho, \cdot)\big)  (\gamma) \leq 0.
\end{align}
Experts in viscosity solution theory may immediately recognize that, at a formal level, this inequality implies the comparison
principle of~\eqref{HJBmacro} (see for instance, Theorems~3 and~5 of Feng and Katsoulakis~\cite{FK09}). To make this rigorous, we
need to face the possibility of cancellations of the kind $\infty - \infty$ when dealing with $S(\rho) - S(\gamma)$, or more
generally, $\infty - \infty -\infty +\infty$ when dealing with the left hand of~\eqref{Hcmp}.

To establish this result rigorously, we introduce two more Hamiltonian operators and formulate Theorem~\ref{wellposed}, which
establishes not only the comparison principle, but also the existence of super- and sub-solutions for the Hamiltonians we now
introduce. Let us mention that, although every operator in this paper is single-valued, for notational convenience, we may still
identify an operator with its graph.

We now define two operators $H_0 \subset C(\sfX) \times M(\sfX; \bar{\R})$ and $H_1 \subset C(\sfX) \times M(\sfX; \bar{\R})$.
Let
\begin{align}
  \label{DH0def}
  D(H_0):= \{ f_0 : f_0(\rho):= \frac{k}{2} \sfd^2(\rho, \gamma) : S(\gamma) <\infty, k \in \R_+ \},
\end{align}
and 
\begin{align}
  \label{H0def}
  H_0 f_0(\rho):=  \frac{k}{2} \big( S(\gamma) - S(\rho)\big)+ \frac{k^2}{2} \Vert \rho - \gamma\Vert_{-1}^2.
\end{align}
This definition is motivated by the formal calculation\eqref{eq:H-formal-sup}. Analogously, motivated by~\eqref{eq:H-formal-sub},
let
 \begin{align}
   \label{DH1def}
   D(H_1):= \{ f_1 : f_1(\gamma):= - \frac{k}{2} \sfd^2(\rho, \gamma) : S(\rho) <\infty, k \in \R_+ \},
\end{align}
and 
\begin{align}
  \label{H1def}
  H_1 f_1(\gamma):=  \frac{k}{2} \big( S(\gamma) - S(\rho)\big)+ \frac{k^2}{2} \Vert \rho - \gamma\Vert_{-1}^2.
\end{align}
  
Instead of working with~\eqref{HJBmacro} with $H$ given in~\eqref{H}, we consider the equation for $H_0$ and seek sub-solutions,
and analogously super-solutions for $H_1$. We establish existence and show that these solutions coincide for a common right-hand
side $h$. Namely, let $h_0, h_1 \in UC_b(\sfX)$ and $\alpha >0$.  We consider a viscosity sub-solution $\overline{f}$ for
\begin{align}\label{H0eqn}
(I- \alpha H_0 ) \overline{f} \leq h_0,
\end{align}
 and a viscosity super-solution $\underline{f}$ for
 \begin{align}\label{H1eqn}
(I-\alpha H_1) \underline{f} \geq h_1.
\end{align}
We will prove the following well-posedness result in Sections~\ref{Sec:CMP} and~\ref{Sec:Nisio}.
\begin{theorem}
  \label{wellposed}
  Let $h \in C_b(\sfX)$ and $\alpha >0$. We consider viscosity sub-solution to~\eqref{H0eqn} and super-solution to~\eqref{H1eqn}
  in the case where $h_0= h_1=h$.  There exists a unique $f \in C_b(\sfX)$ such that it is both a sub-solution to~\eqref{H0eqn} as
  well as a super-solution to~\eqref{H1eqn}.  Moreover, such solution is given by
  \begin{align*}
    f:= R_\alpha h,
  \end{align*}
  where the $R_\alpha$ is given by~\eqref{Resolve}.
  \end{theorem}

\section{The comparison principle}
\label{Sec:CMP}

In this section, we establish the following comparison principle.
\begin{theorem}
  \label{CMP}
  Let $h_0, h_1 \in UC_b(\sfX)$ and $\alpha >0$.  Suppose that 
  $\overline{f} \in USC(\sfX) \cap B(\sfX)$ respectively
  $\underline{f} \in LSC(\sfX) \cap B(\sfX)$ is a sub-solution to~\eqref{H0eqn} respectively a super-solution
  to~\eqref{H1eqn}. Then
  \begin{align*}
    \sup_{\sfX} \big( \overline{f} - \underline{f} \big) \leq \sup_\sfX \big( h_0 - h_1 \big).
  \end{align*}
\end{theorem}
We divide the proof into two parts.

\subsection{A set of extended Hamiltonians and a comparison principle}
\label{sec:set-extend-Hamilt}

We define a new set of operators ${\bar H}_0$ and ${\bar H}_1$ which extend $H_0$ and $H_1$ by allowing a wider class of test
functions. The test functions for these operators are generally discontinuous and can take the values $\pm\infty$.  The operators
${\bar H}_0$ and ${\bar H}_1$ satisfy a structural assumption (Condition~1) used in Feng and Katsoulakis~\cite{FK09}, hence the
comparison principle for the associated Hamilton-Jacobi equations follows from~\cite[Theorem~3]{FK09} (the same technique for
establishing comparison principle has also been presented in Chapter~9.4 of Feng and Kurtz~\cite{FK06} (Condition~9.26 and
Theorem~9.28)).  In the next Section~\ref{sec:Visc-extensions-fro}, we will link viscosity solutions for $\bar{H}_0$ and
$\bar{H}_1$ with those for $H_0$ and $H_1$.

Let $\epsilon \in (0,1)$, $\delta \in [0,1]$ and $\gamma \text{ be such that } S(\gamma)<\infty$. We define
\begin{align}
  \label{f0ext}
  f_0(\rho):=(1-\delta)  \frac{\sfd^2(\rho, \gamma)}{2\epsilon} +   \delta \frac{S(\rho)}{2},
\end{align}
and  
\begin{align}
  \label{barH0}
  {\bar H}_0 f_0(\rho) 
  &:= (1-\delta) H_0  \frac{\sfd^2(\cdot, \gamma)}{2 \epsilon}(\rho)  -  \frac{\delta}{8}I (\rho) \\
                       &= (1-\delta)   \Big( \frac{1}{2\epsilon} \big( S(\gamma ) - S(\rho )\big) 
                           + \frac{\sfd^2(\rho,\gamma)}{2\epsilon^2}  \Big)
                           -\frac{\delta}{8} I( \rho). \nonumber
\end{align}
This definition of ${\bar H}_0$ is motivated by convexity considerations: formally, 
\begin{align*}
  H f_0 \leq (1-\delta) H \frac{\sfd^2(\cdot,\rho)}{2\epsilon}+ \delta H \frac{S}{2} = \bar{H}_0 f_0.
\end{align*}

Similarly, let $\epsilon \in (0,1)$, $\delta \in [0,1]$ and $\rho \text{ be such that } S(\rho)<\infty$. We define
\begin{align}
  \label{f1ext}
  f_1(\gamma):=  (1+\delta) \frac{- \sfd^2(\rho, \gamma)}{2\epsilon} - \delta \frac{S(\gamma)}{2},
\end{align}
and 
\begin{align}
  \label{barH1}
  {\bar H}_1 f_1(\gamma)& := (1+\delta)  H_1 \big(- \frac{\sfd^2(\rho, \cdot)}{2 \epsilon}\big) (\gamma)   
                               + \frac{\delta}{8}  I(\gamma). \\
                        &=(1+\delta)  \Big( \frac{1}{2\epsilon} \big(S(\gamma ) -S(\rho ) \big) 
                            + \frac{ \sfd^2(\rho,\gamma)}{2\epsilon^2}  \Big) + \frac{\delta}{8} I(\gamma). \nonumber
\end{align}
The definition of ${\bar H}_1$ is motivated in a similar way to that of ${\bar H}_0$, but is a bit more involved. We first observe
that
\begin{align*}
  - \frac{\sfd^2(\rho, \gamma)}{2\epsilon} = \frac{1}{1+\delta} f_1(\gamma) + \frac{\delta}{1+\delta} \frac{S(\gamma)}{2}.
\end{align*}
By  convexity considerations applied formally to $H$, we have 
\begin{align*}
  \Big(H   \frac{-\sfd^2(\rho,\cdot)}{2\epsilon}  \Big) (\gamma) \leq \frac{1}{1+\delta}  
  H f_1(\gamma) + \frac{\delta}{1+\delta} H\frac{S}{2}(\gamma).
\end{align*}
That is, $\bar{H}_1$ is defined such that
\begin{align*}
  \bar{H}_1 f_1 \leq H f_1.
\end{align*}

We now give an auxiliary statement to establish the existence of strong viscosity solutions.
\begin{lemma}
  \label{bfHcmp}
  Take $f_0,f_1$ in~\eqref{f0ext} and~\eqref{f1ext}, then $ {\bar H}_0 f_0 \in USC(\sfX, \bar{\R})$ and
  ${\bar H}_1 f_1 \in LSC(\sfX, \bar{\R})$. Moreover, for $\rho, \gamma \in \sfX$ such that $S(\rho)+S(\gamma)<\infty$, we have
  \begin{align}
    \label{barHineq}
    \frac{1}{1-\delta} {\bar H}_0 f_0(\rho) - \frac{1}{1+\delta} {\bar H}_1 f_1(\gamma) 
    \leq - \frac{1}{8} \big( \frac{\delta}{1-\delta} I(\rho) + \frac{\delta}{1+\delta} I (\gamma) \big) \leq 0.
  \end{align}
\end{lemma}

\begin{proof}
  The semi-continuity properties follow from the lower semi-continuity of $\rho \mapsto S(\rho)$ and $\rho \mapsto I(\rho)$ in
  $(\sfX, \sfd)$.  The estimate~\eqref{barHineq} follows from direct verification.
\end{proof}

Below, we establish the first comparison result in this paper for \emph{strong} viscosity solutions (see the note
following~\ref{viscDef}).

\begin{lemma}
  \label{cmpbfH}
  Suppose that $\overline{f} \in USC(\sfX, \R) \cap B(\sfX)$ and $\underline{f} \in LSC(\sfX, \R) \cap B(\sfX)$ are respectively
  viscosity strong sub-solution and strong super-solution to
  \begin{align}
    \overline f - \alpha {\bar H}_0 \overline{f} &\leq   h_0, \label{bfH0eqn} \\
    \underline f - \alpha {\bar H}_1 \underline{f} & \geq   h_1.\label{bfH1eqn}
  \end{align}
  Then
  \begin{align*}
    \sup_{\rho \in \sfX : S(\rho) <\infty} \big( \overline{f} (\rho) - \underline{f}(\rho) \big) \leq \sup_{\rho \in \sfX} 
    \big( h_0(\rho)  - h_1(\rho) \big) .
  \end{align*}
  Moreover, if $\overline{f}, \underline{f} \in C_b(\sfX)$, then
  \begin{align*}
    \sup_\sfX \big( \overline{f}   - \underline{f}  \big) \leq \sup_\sfX \big( h_0  - h_1\big).
  \end{align*}
\end{lemma}

\begin{proof}
  The estimates in Lemma~\ref{bfHcmp} imply that Theorem~3 in Feng and Katsoulakis~\cite{FK09} applies, hence the conclusions
  follow.
\end{proof}

\subsection{Viscosity extensions from $H_0$ and $H_1$ to $\bar{H}_0$ and $\bar{H}_1$ and the comparison theorem}
\label{sec:Visc-extensions-fro}

Throughout this section, we assume that the functions $\overline{f} \in USC(\sfX) \cap B(\sfX)$ respectively
$\underline{f} \in LSC(\sfX) \cap B(\sfX)$ are a viscosity sub-solution to~\eqref{H0eqn} respectively a super-solution
to~\eqref{H1eqn}. The following regularizations $\overline{f}_t$ and $\underline{f}_t$ and Lemma~\ref{YosC} are analogues of
Lemma~13.34 of Feng and Kurtz~\cite{FK06}.

For each $t \in (0,1)$, we define
\begin{align}
  \label{fReg}
  \overline{f}_t(\rho) &:=  \sup_{\gamma \in \sfX} \big( \overline{f}(\gamma) - \frac{\sfd^2(\rho,\gamma)}{2t} \big), \\
  \underline{f}_t(\gamma)&:= \inf_{\rho \in \sfX}  \big( \underline{f}(\rho) + \frac{\sfd^2(\rho,\gamma)}{2t} \big).
\end{align}
It follows that 
\begin{align*}
  \overline{f} \leq \overline{f}_t , \quad \underline{f}_t \leq \underline{f}, \quad \forall t \in (0,1).
\end{align*}

\begin{lemma}
  \label{YosC}
  $\overline{f}_t, \underline{f}_t \in {\rm Lip}(\sfX)$.
\end{lemma}

Next we establish a few \emph{a priori} estimates. To get some intuition of what we will derive, we now explain the heuristic
ideas. Let $f_0$ be as in~\eqref{f0ext} taking the form
\begin{align}
  \label{f0ext2}
  f_0(\rho) := (1-\delta) \frac{\sfd^2(\rho,\hat{\gamma})}{2\epsilon}  + \delta \frac{S(\rho)}{2},
\end{align}
where the $\hat{\gamma} \in \sfX$ is such that $S(\hat{\gamma})<\infty$. Then formally
\begin{align}\label{delf0}
\frac{\delta f_0}{\delta \rho} = (1-\delta)  \frac{(-\partial_{xx}^2)^{-1}(\rho - \hat{\gamma})}{\epsilon} 
  + \delta \frac{\log \rho}{2}.
\end{align}
If $\log \rho \in L^\infty(\mathcal O)$, then the expression above is an element in $L^2(\mathcal O)$.  Let
$\rho_0, \gamma_0 \in \sfX$ be such that $S(\rho_0) <\infty$. We assume that
\begin{align}
  \label{cmpfun}
  f_0(\rho)-   f_0(\rho_0) \geq  - \frac{1}{t} \Big(\frac{\sfd^2(\rho,\gamma_0)}{2} 
  - \frac{\sfd^2(\rho_0,\gamma_0)}{2} \Big), \quad \forall \rho \in \sfX .
\end{align}
Then, by taking directional derivatives along paths $t \mapsto \rho:=\rho(t)$ with $\rho(0) =\rho_0$, the above will imply the following comparison of Hamiltonians
\begin{align*}
  H \Big(\gamma_0,\frac{\delta}{\delta \gamma_0}\frac{ \sfd^2(\rho_0,\cdot)}{2t}  \Big) 
  \leq H_0 \big( \frac{\sfd^2(\rho_0,\cdot)}{2t}  \big)(\gamma_0) 
  \leq {\bar H}_0 f_0 (\rho_0).
\end{align*}

We now rigorously justify these formal comparisons. We divide the justification into three steps. First, we make sense of the
following statement in a rigorous way:
\begin{align*}
  \langle \frac{\delta f_0}{\delta \rho_0}, \frac12 \partial_{xx}^2 \log \rho_0 \rangle
  \geq - \frac{1}{2t} \langle \frac{\delta \sfd^2(\cdot,\gamma_0) }{\delta \rho_0}, \frac12 \partial_{xx}^2 \log \rho_0 \rangle.
\end{align*}

\begin{lemma}
  \label{Df0est1} 
  Let $f_0$ be given by~\eqref{f0ext2} with $S(\hat{\gamma})<\infty$.  Let $\rho_0,\gamma_0 \in \sfX$ be such that
  $S(\rho_0)<\infty$ and that~\eqref{cmpfun} holds. Then
  \begin{align}
    \label{SleqSI}
    \frac{1}{2t} \Big( S(\rho_0) - S(\gamma_0)\Big) 
    \leq \Big( (1-\delta) \frac{S( \hat{\gamma}) - S(\rho_0)}{2 \epsilon}
    -\frac{ \delta}{4} I(\rho_0)\Big).
  \end{align}
  Note that if we assume $S(\gamma_0)<\infty$, then this estimate immediately implies an \emph{a posteriori} estimate
  \begin{align*}
    I(\rho_0)<\infty.
  \end{align*}
\end{lemma}

\begin{proof}
  We claim that there exists a curve $\rho \in C([0,\infty); \sfX)$ such that the following partial differential equation
  \begin{align*}
    \partial_t \rho = \frac12 \partial_{xx}^2 \log \rho, \quad \rho(0) = \rho_0.
  \end{align*}
  is satisfied in the sense of Definition~\ref{DefCPDE} below. Moreover,
  \begin{align}
    \label{eq:curve1}
    S(\rho(s)) - S(\rho_0) \leq -\frac12 \int_0^s  I(\rho(r)) dr, \quad \forall s>0,
  \end{align}
  and 
  \begin{align}
    \label{eq:curve2}
    \frac12 \sfd^2(\rho(s), \hat{\gamma}) - \frac12 \sfd^2(\rho_0,\hat{\gamma}) 
     \leq  \int_0^s \frac12 \Big( S(\hat{\gamma}) - S(\rho(r))  \Big) dr,
    \quad \forall \hat{\gamma} \text{ with }  S(\hat{\gamma}) <\infty.
\end{align}

A rigorous justification of the claim above can be found in Lemma~\ref{apriori}. Results of this type are well-known for
$\partial_t \rho = \partial_{xx} \Phi(\rho)$ with a regular $\Phi$ (e.g., Theorem~5.5 of Vazquez~\cite{Vaz07}). However, in our
case, this existing theory does not directly apply, as our $\Phi(r)= \frac12 \log r$ is singular at $r=0$. Although main ideas for
establishing these estimates remain the same, additional subtleties need to be taken care of.  Of course, the proofs in
Section~\ref{DiffC} are independent of the results in this section on comparison principle, hence we can invoke these results here
without creating circular arguments.
 
In summary, by~\eqref{eq:curve1} and~\eqref{eq:curve2}, the curve satisfies
\begin{align*}
  f_0(\rho(s)) - f_0(\rho_0)   \leq   
      \int_0^s \Big( (1-\delta) \frac{S( \hat{\gamma}) - S(\rho(r))}{2 \epsilon}
                                      -\frac{ \delta}{4} I(\rho(r)) \Big) dr.
\end{align*}
We also have (note that the following inequality holds trivially if $S(\gamma_0)=+\infty$)
\begin{align*}
- \frac12 \sfd^2(\rho(s), \gamma_0) + \frac12 \sfd^2(\rho_0,\gamma_0)  \geq   \int_0^s \frac12 \Big(  S(\rho(r)) - S(\gamma_0) \Big) dr.
\end{align*}
We plug the two lines above into~\eqref{cmpfun} and note that $S, I$ are both lower semicontinuous. The inequality~\eqref{SleqSI}
follows.
\end{proof}

\begin{lemma}
  \label{Df0est2} 
  Let the $f_0, \rho_0,\gamma_0, \hat{\gamma}$ be as in the previous Lemma~\ref{Df0est1} with the additional assumption that
  $S(\gamma_0)<\infty)$ (hence $I(\rho_0)<\infty$ by the previous lemma). Then the term $\frac{\delta f_0}{\delta \rho_0}$
  in~\eqref{delf0} is well defined and
  \begin{align*}
    \frac{\delta f_0}{\delta \rho_0} = - \frac{1}{t}   (-\partial^2_{xx})^{-1}(\rho_0 - \gamma_0), 
  \end{align*}
  which implies
  \begin{align*}
    \int_{\mathcal O} |\partial_x \frac{\delta f_0}{\delta \rho_0}|^2 d x
      = \frac{\sfd^2(\rho_0, \gamma_0)}{t^2} .
  \end{align*}
\end{lemma}

\begin{proof}
Let $\gamma \in C^\infty(\mathcal O) \cap \sfX$ with $\inf_{\mathcal O} \gamma >0$. We define
\begin{align*}
  \rho(s):=  \rho_0 + sj, \text{ with } j :=(\gamma -\rho_0), \forall s \in [0,1].
\end{align*}
From $I(\rho_0)<\infty$, we have $\rho_0, j \in C(\mathcal O)$ and $\inf_{\mathcal O} \rho_0 >0$. Therefore,
$\rho(s) \in C(\mathcal O) \cap \sfX$ and $\inf_{\mathcal O} \rho(s)>0$ for all $s \in [0,1]$. With these regularities,
$(-\partial_{xx}^2)^{-1}(\rho(s) - \gamma) \in C(\mathcal O)$. Hence, if we define
\begin{align*}
\frac{\delta f_0}{\delta \rho(s)}:= (1-\delta) \frac{1}{\epsilon} (-\partial^2)^{-1} (\rho(s) - \hat{\gamma}) + \delta \frac12 \log \rho,
\end{align*}
then this expression is well defined and
\begin{align*}
  \frac{\delta f_0}{\delta \rho(r)} \in C(\mathcal O) \text{ and } 
  r \mapsto \langle \frac{\delta f_0}{\delta \rho(r)}, j \rangle \in C([0,1]).
\end{align*}
Therefore
\begin{align*}
  f_0(\rho(s)) - f_0(\rho_0) = \int_0^s \langle \frac{\delta f_0}{\delta \rho(r)}, j \rangle dr
\end{align*}
and 
\begin{align*}
  \frac12 \sfd^2(\rho(s), \gamma_0) - \frac12 \sfd^2(\rho_0,\gamma_0) 
  = \int_0^s \langle (-\partial_{xx}^2)^{-1}(\rho(r) -\gamma_0) , j \rangle dr.
\end{align*}

In view of~\eqref{cmpfun} and the regularities $\rho(r), j \in C(\mathcal O)$, we have
\begin{align*}
  \langle \frac{\delta f_0}{\delta \rho_0}, j \rangle \geq  \langle -\frac{1}{t} 
  (-\partial_{xx}^2)^{-1}(\rho(r) -\gamma_0) , j \rangle.
\end{align*}
As $j$ is arbitrary, the claim follows.  
\end{proof}
 
\begin{lemma}
  \label{Df0est3}
 Let $f_0, \rho_0,\gamma_0, \hat{\gamma}$ be as in Lemma~\ref{Df0est1}, with $S(\gamma_0)<\infty$. We assume that~\eqref{cmpfun}
 holds. Then
 \begin{align*}
   H_0 \big( \frac{\sfd^2(\rho_0,\cdot)}{2t}  \big)(\gamma_0) \leq   {\bar H}_0 f_0(\rho_0).
 \end{align*}
\end{lemma}

\begin{proof}
We have shown in Lemma~\ref{Df0est1} that $I(\rho_0)<\infty$. Note that by definition
\begin{align*}
  H_0 \frac{ \sfd^2(\rho_0, \cdot)}{2t} (\gamma_0)  &= \frac{1 }{2t} \Big( S( \rho_0) -S(\gamma_0) \Big)
  + \frac{  \sfd^2(\rho_0,\gamma_0)}{2t^2},
  \intertext{and}
  {\bar H}_0 f_0 (\rho_0)&= (1-\delta) \Big( \frac{1}{2 \epsilon}\big( S(\hat{\gamma}) - S(\rho_0)\big) 
  + \frac{\sfd^2(\rho_0, \hat{\gamma})}{2 \epsilon^2}\Big)  -\frac{\delta}{8} I (\rho_0).
\end{align*}
By Lemma~\ref{Df0est2} and then~\eqref{delf0} and the convexity of quadratic functions, we have
\begin{align*}
  \frac{1}{t^2} \sfd^2(\rho_0, \gamma_0) =  
  \int_{\mathcal O} |\partial_x \frac{\delta f_0}{\delta \rho_0}|^2 d x 
  \leq  \Big( (1-\delta) \frac{\sfd^2(\rho_0,\hat{\gamma})}{\epsilon^2} 
  + \delta \frac{1}{4} I(\rho_0) \Big).
\end{align*}
Combined with the estimate~\eqref{SleqSI} in Lemma~\ref{Df0est1}, the conclusion follows.
\end{proof}

We now state the first existence result for viscosity solutions, in a suitably regularized setting. The proof of Theorem~\ref{CMP}
will follow easily from this statement.
\begin{lemma}
  \label{visExt}
  Let us consider $h_0 \in UC_b(\sfX)$ with a nondecreasing modulus of continuity denoted as $\omega_0:=\omega_{h_0}$. Let
  \begin{align*}
    h_{0,t}(\rho):= h_0(\rho)+ \omega_{0} \big( \sqrt{t} C_{\overline{f}} \big), \forall \rho \in \sfX, 
    \text{ with } C_{\overline{f}} := \sqrt{2}\big( \Vert \overline{f} \Vert_\infty\big)^{1/2}.
  \end{align*}
  Then the $\overline{f}_t \in C_b(\sfX)$ is a strong viscosity sub-solution to the Hamilton-Jacobi equation~\eqref{bfH0eqn} with
  $h_0$ being replaced by $h_{0,t}$.

  Similarly, suppose that $h_1 \in UC_b(\sfX)$ with a nondecreasing modulus of continuity $\omega_1:=\omega_{h_1}$. Let
  \begin{align*}
    h_{1,t}(\gamma):= h_1(\gamma)- \omega_{1} \big( \sqrt{t} C_{\underline{f}} \big), \forall \gamma \in \sfX, 
    \text{ with } C_{\underline{f}} := \sqrt{2}\big( \Vert \underline{f} \Vert_\infty\big)^{1/2}.
  \end{align*}
  Then $\underline{f}_t \in C_b(\sfX)$ is a strong viscosity super-solution to the Hamilton-Jacobi equation~\eqref{bfH1eqn} with
  $h_1$ being replaced by $h_{1,t}$.
\end{lemma}

\begin{proof}
We only prove the sub-solution case, the super-solution case is similar. Let $f_0$ be as in~\eqref{f0ext2}. We assume that $\rho_0
\in \sfX$ is such that
\begin{align*}
  ( \overline{f}_t - f_0)(\rho_0)  = \sup_{\sfX} (\overline{f}_t - f_0).
\end{align*}
Then $S(\rho_0)<\infty$. The existence of such $\rho_0$ is guaranteed by the lower-semicontinuity of $f_0$, $f_0 \in
LSC(\sfX;\bar{\R})$ and the compactness of $\sfX$.  We have
\begin{align}
  \label{fdf0}
  \sup_{\gamma \in \sfX} \Big( \overline{f}(\gamma) - \frac{\sfd^2(\rho_0,\gamma)}{2t} \Big) - f_0(\rho_0) 
  =   \sup_{\rho, \gamma \in \sfX} \Big( \big( \overline{f}(\gamma) - \frac{\sfd^2(\rho,\gamma)}{2t} \big) - f_0(\rho) \Big).
\end{align}
Since the $\overline{f}$ is a viscosity sub-solution to~\eqref{H0eqn}, by compactness of $\sfX$, there exists $\gamma_0 \in \sfX$
such that
\begin{align}
  \label{Yosida}
  \Big( \overline{f}(\gamma_0) - \frac{\sfd^2(\rho_0,\gamma_0 )}{2t} \Big)
  =\sup_{\gamma \in \sfX}  \Big( \overline{f}(\gamma) - \frac{\sfd^2(\rho_0,\gamma)}{2t} \Big)
  =  \overline{f}_t(\rho_0)
\end{align}
with
\begin{align}
  \label{fhHd}
  (\overline{f} - h_0)(\gamma_0)  \leq H_0 \frac{\sfd^2(\rho_0, \cdot)}{2t} (\gamma_0).
\end{align}
From the upper boundedness of $h_0 - \overline{f}$, we arrive at the estimate that $S(\gamma_0)<\infty$.  Thus~\eqref{fdf0}
reduces to
\begin{align}\label{coupling}
  \Big(\overline{f}(\gamma_0) - \frac{\sfd^2(\rho_0,\gamma_0)}{2t} \Big) - f_0(\rho_0)
 =\sup_{\gamma, \rho \in \sfX} \Big(  \big( \overline{f}(\gamma) - \frac{\sfd^2(\rho,\gamma)}{2t} \big) - f_0(\rho)\Big).
\end{align}
The above implies~\eqref{cmpfun}, hence we can apply Lemma~\ref{Df0est3} to~\eqref{fhHd}, which results in
\begin{align*}
  (\overline{f} - h_0)(\gamma_0)  \leq {\bar H}_0 f_0 (\rho_0).
\end{align*}

From~\eqref{coupling}, we obtain a rough estimate 
\begin{align*}
  t^{-1} \sfd^2(\rho_0,\gamma_0) =\overline{f}(\gamma_0) - \overline{f}_t(\rho_0) 
  \leq \overline{f}(\gamma_0) - \overline{f}(\rho_0) \leq 2 \Vert \overline{f} \Vert_\infty.
\end{align*}
Denoting a nondecreasing modulus of $h$ by $\omega_h$, then
\begin{align*}
  h(\gamma_0) -  h(\rho_0) \leq \omega_h\big(\sfd(\rho_0, \gamma_0)\big) \leq \omega \big( \sqrt{t} C_{\overline{f}} \big), 
  \text{ with } C_{\overline{f}} := \sqrt{2}\big( \Vert \overline{f} \Vert_\infty\big)^{1/2}.
\end{align*}
We note that, from~\eqref{Yosida}, 
\begin{align*}
  (\overline{f} - h_0)(\gamma_0)  &=  \overline{f}_t (\rho_0) + \frac{\sfd^2(\rho_0,\gamma_0)}{2t} -h(\rho_0)
  + \big( h(\rho_0) - h_0(\gamma_0)\big) \\
  &\geq  \overline{f}_t (\rho_0) -\Big( h(\rho_0) + \omega_h \big(\sqrt{t} C_{ \overline{f}} \big) \Big).
\end{align*}
The claim is established.
\end{proof}

Finally, we are in a position to prove Theorem~\ref{CMP}.
\begin{proof}
By Lemmas~\ref{YosC} and \ref{visExt}, we know that the functions $\overline{f}_t$ and $\underline{f}_t$ satisfy the conditions of
Lemma~\ref{cmpbfH}, for each $t>0$.  Hence by the comparison principle in Lemma~\ref{cmpbfH}, we have
\begin{align*}
  \sup_\sfX\big( \overline{f}_t - \underline{f}_t\big) \leq \sup_\sfX \big(h_{0,t} - h_{1,t}\big) = \sup_\sfX \big( h_0 - h_1 \big) 
  + \omega_0(\sqrt{t} C_{\overline{f}}) + \omega_1(\sqrt{t} C_{\underline f}).
\end{align*}
Since 
\begin{align*}
  \overline{f}(\rho)  - \underline{f}(\rho) \leq \overline{f}_t(\rho)  - \underline{f}_t(\rho),
  \quad \forall t\in (0,1), \rho \in \sfX
\end{align*}
the conclusion of Theorem~\ref{CMP} follows by taking $t \to 0^+$.
\end{proof}

\section{Existence of solutions for the Hamilton-Jacobi equation through optimal control of nonlinear diffusion equations and
  related Nisio semigroups}
  \label{Sec:Nisio}
We recall that $(\sfX, \sfd)$ is a compact metric space, hence $C(\sfX) =C_b(\sfX)= UC(\sfX)$. Theorem~\ref{CMP} establishes that
for each $h \in C_b(\sfX)$ and $\alpha >0$, there exists at most one function $f$ such that it is both a sub-solution
to~\eqref{H0eqn} as well as a super-solution to~\eqref{H1eqn}. In this section, we show that there exists such a
solution. Moreover, this solution is unique and can always be represented as the value function $f= R_\alpha h$ of the family of
nonlinear diffusion equations with control introduced in the introduction (see~\eqref{Resolve} for the definition of the operator
$R_\alpha$):
\begin{align}
  &  \partial_t \rho =\frac12 \partial_{xx}^2 \log \rho + \partial_x \eta,  \label{CPDE}
  \intertext{with}
  &    \int_0^T\int_{\mathcal O} | \eta(r,x)|^2 dx dr < \infty.  \label{CPDEcost}
\end{align}

\subsection{A set of nonlinear diffusion equations with control}
\label{DiffC}
Throughout this section, we always assume that $\eta$ satisfies~\eqref{CPDEcost}. We use the convention $0 \log 0 :=0$.

\begin{definition}
  \label{DefCPDE}
  We say that $(\rho, \eta)$ is a weak solution to~\eqref{CPDE} in the time interval $[0,T]$ if the following holds:
  \begin{enumerate}
  \item $\rho(\cdot) \in C([0,T];\sfX)$.
  \item $\rho(t, dx) = \rho(t,x) dx$ holds for $t>0$, for some measurable function $(t,x) \mapsto \rho(t,x)$.
  \item The following estimates hold:
    \begin{align}
      \label{Sest}
      \int_0^T \int_{\mathcal O} \rho(t,x) \log \rho(t,x) dx dt <\infty,
    \end{align} 
    and 
    \begin{align}
      \label{logrho}
      \int_s^T \int_{\mathcal O} | \log \rho(t,x) | dx dt <\infty, \quad \forall s>0,
    \end{align}
    and
    \begin{align}
      \label{Dlog}
      \int_s^T \int_{\mathcal O} |\partial_x \log \rho(t,x)|^2 dx dt <\infty, \quad \forall s>0.
    \end{align}
  \item For every $\varphi \in C^\infty(\mathcal O)$ and $0<s<t \leq T$, we have
    \begin{align}
      \label{wCPDE}
      \langle \varphi, \rho(t) \rangle - \langle \varphi, \rho(s)\rangle 
      = \int_s^t \big( \langle \frac12 \partial^2_{xx} \varphi, \log \rho(r) \rangle
      - \langle \partial_x \varphi, \eta(r) \rangle \big)dr.
    \end{align}
  \end{enumerate}
\end{definition}
In the above, note that $[0,\infty) \ni r \mapsto r \log r$ is a function bounded from below (with convention $0\log 0 =0$), hence the $\int_{\mathcal O} \rho(x) \log \rho(x) dx \in \R \cup \{+\infty\}$ is well defined.

We now describe the technical difficulties we need to overcome in this section. Let $\Phi(r):= \frac12 \log r$, for
$r>0$. Then~\eqref{CPDE} can be written as
\begin{align*}
\partial_t \rho = \partial_{xx}^2 \Phi(\rho) + \partial_x \eta,
\end{align*}
where 
\begin{align*}
  \eta(t,x) \in L^2\big((0,T); L^2(\mathcal O)\big).
\end{align*}
Equations similar to this type have been studied by V\'azquez~\cite{Vaz07} with the control variable $\partial_x \eta$ denoted using $f$.  However, there it is assumed that $\Phi$ is at least continuous. In contrast, our $\Phi$ has $\Phi(0) =-\infty$ and is thus singular.  In addition, we also need to ensure that solution is non-negative.  In~\cite{Vaz07}, an approach based on the maximum principle is developed to establish positivity of a solution. This works well in the absence of control, $f=0$, or when $f\geq 0$.  However, the positivity of a solution in our case, for this special $f$, seems to be of a different origin: the singularity of $\Phi(0)=-\infty$ plays a key role.  Therefore, we present a detailed justification using energy estimates. A further, but very minor, issue in that~\cite{Vaz07} is focused on Dirichlet or Neumann boundary conditions, whereas we have a periodic boundary. However, the boundary conditions only appear after integration by parts and the argument simplifies for the case of periodic boundary conditions. Hence, we do not provide details for this last issue and only address the first two issues below by studying a sequence of approximate equations.

The main purpose of this subsection is to establish the following existence result.  We recall that the definition of the entropy function $S$ is given in~\eqref{freeE}.

\begin{lemma}
  \label{apriori}
  For every $\eta$ satisfying~\eqref{CPDEcost} and every $\rho(0) =\rho_0 \in \sfX \subset H_{-1}(\mathcal O)$, there exists a
  $\rho(\cdot) \in C([0,T]; \sfX)$ such that $(\rho,\eta)$ solve~\eqref{CPDE}--\eqref{CPDEcost} in the weak sense of
  Definition~\ref{DefCPDE}. This solution is unique.  Indeed, such a pair $(\rho,\eta)$ also satisfies the following properties.
  \begin{enumerate}
  \item For every $\gamma_0 \in \sfX$ such that $S(\gamma_0)<\infty$, and for every $0 \leq s<t <T$, the following variational
    inequalities hold:
    \begin{align}\label{disineq}
      &  \frac12 \Vert \rho(t) -\gamma_0 \Vert_{-1}^2 + \int_s^t  \Big( \frac12 \big( S(\rho(r)) - S(\gamma_0) \big)  \\
      &  \qquad \qquad \qquad \qquad \qquad 
      + \int_{\mathcal O} \eta(r,x)\big( \partial_x (-\partial_{xx}^2)^{-1} (\rho(r)-\gamma_0)(x) \big)dx \Big) dr \nonumber \\
      &  \qquad \qquad \qquad \qquad \qquad  \qquad  \leq \frac12 \Vert \rho(s) -\gamma_0 \Vert_{-1}^2. \nonumber
    \end{align}
  \item It holds that $S(\rho(t)) <\infty$ for every $t >0$ and $\int_0^T S(\rho(r)) dr <\infty$ (this implies in particular that
    $\rho(t,dx) = \rho(t,x) dx$ for $t>0$).
  \item For every $0<s<T<\infty$, it holds that
    \begin{align*}
      \int_s^T \int_{\mathcal O} \big(   - \log \rho\big)^+ dx dr <\infty.
    \end{align*}
  \item For every $0 \leq s \leq t$, allowing the possibility of $S(\rho(0))=+\infty$, the following holds
    \begin{align}
      \label{SIfluc}
      S(\rho(t)) +\int_s^t \int_{\mathcal O} \big( \frac12 |\partial_x \log \rho(r,x)|^2 + \eta(r,x) \partial_x \log
      \rho(r,x)\big) dx dr \leq S(\rho(s)).
    \end{align}
\end{enumerate}
\end{lemma}

We divide the proof into several parts.

\subsubsection{Approximate equations}
\label{sec:Appr-equat}

Let $\eta \in L^2((0,T) \times {\mathcal O})$ and $\rho_0 \in \sfX$.  We extend the definition to 
$L^2(\R \times {\mathcal O})$ by $\eta(t,x):=0$ whenever $t \leq 0$ or $t\geq T$.
Let $J \in C^\infty(\mathcal O)$ be a standard spatial mollifier and $G \in C^\infty_c(\mathcal O)$ a standard time-variable mollifier. We
define mollification of (possibly signed) measures and functions on $\mathcal O$ in the usual sense. Hence $\rho_{\epsilon,0}: =
J_\epsilon * \rho_0 \in C^\infty(\mathcal O)$. We write
\begin{align*}
  \eta_\epsilon(t,x) := (G_\epsilon*_t J_\epsilon *_x\eta)(t,x),  \text{ and } f_\epsilon:= \partial_x \eta_\epsilon.
\end{align*}
We approximate the singular function $\Phi$ by a smooth function $\Phi_\epsilon$ as follows:
\begin{align*}
  \Phi_\epsilon(r):=
  \begin{cases}
    \frac12 \log r  + C_\epsilon, & r \geq \epsilon, \\
    \theta_\epsilon(r)  \in C^2 & 0\leq r \leq \epsilon \\
    \frac{1}{\epsilon} r, & r <0.
  \end{cases}
\end{align*}
Note that $\Phi^\prime(r) =\Phi_\epsilon^\prime(r)$ for $r > \epsilon$.  We choose the constant $C_\epsilon:= -\frac{3}{2}\log \epsilon$ so that $\Phi_\epsilon(\epsilon) = - \log \epsilon >0 = \Phi_\epsilon(0)$. This feature allows us to pick a smooth function $\theta_\epsilon$ with $\theta^\prime_\epsilon >0$ such that
\begin{align*}
  & \theta_\epsilon(0) =0, \quad \theta_\epsilon^\prime(0) =\frac{1}{\epsilon},\quad \theta_\epsilon^{\prime \prime}(0)=0; \\
  & \theta_\epsilon(\epsilon) =\frac12 \log \epsilon + C_\epsilon = -\log \epsilon >0, \quad \theta_\epsilon^\prime(\epsilon)
  =\frac{1}{2\epsilon},
  \quad \theta_\epsilon^{\prime \prime}(\epsilon)=-\frac{1}{2 \epsilon^2},
\end{align*}
so that 
\begin{align*}
\Phi_\epsilon \in C^3(\R),  \quad \Phi^\prime_\epsilon(r)>0, \forall r\in \R,  \quad \Phi_\epsilon(0)=0.
\end{align*}
We denote the primitive $\Theta_\epsilon(t) =\int_0^t \theta_\epsilon(r) dr$, and note that ($\theta_\epsilon^\prime >0$ ensures that $\theta_\epsilon$ is an increasing function)
\begin{align*}
  \sup_{0 <t \leq \epsilon} |\Theta_\epsilon(t)| \leq \epsilon \theta_\epsilon(\epsilon) = - \epsilon \log \epsilon \to 0
  \text{ as } \epsilon \to 0^+.
\end{align*}
This construction ensures that $\Phi_\epsilon^\prime \in C(\R)$.  Now, we consider
\begin{align}
  \label{appCPDE}
  \partial_t \rho_\epsilon = \partial_{xx}^2 \Phi_\epsilon(\rho_\epsilon)
  + \partial_x \eta_\epsilon, \quad \rho_\epsilon(0) = J_\epsilon * \rho_0.
\end{align}

By Theorem~5.7 in~\cite{Vaz07}, there exists a unique weak solution $\rho_\epsilon(\cdot)$ in the sense of Definition 5.4
of~\cite{Vaz07}. Hence for every $\varphi \in C^\infty(\mathcal O)$, it holds that
\begin{align}\label{awCPDE}
 \langle \varphi, \rho_\epsilon(t) \rangle - \langle \varphi, \rho_\epsilon(s)\rangle
   =\int_s^t \big( \langle \frac12 \partial_{xx}^2 \varphi, \Phi_\epsilon(\rho_\epsilon(r)) \rangle
    - \langle \partial_x \varphi, \eta_\epsilon(r) \rangle \big) dr.
\end{align}
In fact, in the regularized situation considered at the moment, standard quasilinear theory applies (e.g., the method of proving Theorem 6.1 in Chapter~V of Lady\v{z}enskaja, Solonnikov and Ural{\textprime}ceva~\cite{LSU68}), hence $\rho_\epsilon \in C^{1,2}((0,T) \times \bar{\mathcal O})$ is a classical solution.  Note that the first part of condition (6.9) in Chapter V of \cite{LSU68} requires $\Phi^\prime$ be uniformly bounded away from zero. The above constructed $\Phi_\epsilon$ does not satisfy this requirement. However, this is not a problem in current context because that $\rho_\epsilon$ is bounded. We explain this in detail: Let $M>0$ be a large parameter, we modify the definition of $\Phi_\epsilon(r)$  into $\Phi_\epsilon^M(r)$  for those $r > M$ and keep 
$\Phi_\epsilon^M(r)= \Phi_\epsilon(r)$ for $r \leq M$. We do such modification so that the $\Phi_\epsilon^M$ satisfies conditions of \cite{LSU68}. 
Then there exists a unique classical $C^{1,2}$-solution $\rho_\epsilon^M$ for
\begin{align*}
 \partial_t \rho_\epsilon^M = \partial_{xx}^2 \Phi_\epsilon^M(\rho_\epsilon^M)
  + \partial_x \eta_\epsilon.
\end{align*}
By the maximum principle,
\begin{align*}
\inf_{\mathcal O} \rho_\epsilon(0) + t \inf_{[0,T] \times \mathcal O} \partial_x \eta_\epsilon
 \leq \rho_\epsilon^M(t,x)  \leq \sup_{\mathcal O} \rho_\epsilon(0) + t \sup_{[0,T] \times \mathcal O} \partial_x \eta_\epsilon = : M_0, \quad \forall t \in (0,T).
\end{align*}
Consequently, when $M > M_0$, $\rho_\epsilon:=\rho_\epsilon^M$ solves \eqref{appCPDE} in classical sense. 

We note that, for each $t>0$ and $\epsilon >0$, we cannot rule out the possibility that 
$\rho_\epsilon(x,t)<0$. But we will show that this possibility disappears in the limit $\epsilon \to 0^+$, by asymptotic estimates we now establish.
  
There are three important regularity properties of the $\rho_\epsilon$ we will exploit. First, let
\begin{align}
  \label{Psidef}
  \Psi_\epsilon (s) := \int_0^s \Phi_\epsilon(r) dr.
\end{align}
Then we have the energy inequality~(5.20) in Theorem~5.7 of~\cite{Vaz07}:
\begin{align}
  \label{VazEnEst}
  &  \int_{\mathcal O} \Psi_\epsilon(\rho_\epsilon(x,T)) dx + \int_s^T \int_{\mathcal O}
  \Big( |\partial_x \Phi_\epsilon(\rho_\epsilon(r,x))|^2 
  + \eta_\epsilon(r,x) \partial_x \Phi_\epsilon\big(\rho_\epsilon(r,x)\big)  \Big)  dx dr  \\
  & \qquad \leq \int_{\mathcal O} \Psi_\epsilon(\rho_{\epsilon}(s,x)) dx ,
  \quad \forall 0 \leq s \leq T \nonumber 
\end{align} 
Note that $\rho_\epsilon(0) \in L^\infty(\mathcal O)$, hence $\int | \Psi_\epsilon(\rho_\epsilon(0,x)) | dx <\infty$. Also, by
Jensen's inequality,
\begin{align*}
 \int_s^T \int_{\mathcal O} |\eta_\epsilon(r,x)|^2 dx dr \leq \int_s^T \int_{\mathcal O} |\eta(r,x)|^2 dx dr<\infty.
\end{align*}

Second, we have inequalities of dissipation type: for all $\gamma_0 \in H_{-1}(\mathcal O)$ such that $\int_{\mathcal O}
\Psi_\epsilon(\gamma_0)dx<\infty$ and $0\leq s<t$
\begin{align}
  \label{epsdiss}
 \frac12 \Vert \rho_\epsilon(t) - \gamma_0 \Vert_{-1}^2 
 +&  \int_s^t \big\langle (\rho_\epsilon(t) - \gamma_0),  \Phi_\epsilon(\rho_{\epsilon}(r))\big\rangle dr  \\
 & \qquad     + \int_s^t \big\langle \eta_\epsilon(r) , \partial_x (-\partial_{xx}^2)^{-1}  (\rho_\epsilon(t) - \gamma_0) \big\rangle dr
   \leq \frac12 \Vert \rho_\epsilon(s) -\gamma_0\Vert_{-1}^2.  \nonumber 
\end{align}
This estimate can be verified through integration by parts.  Note that $\Psi_\epsilon^{\prime \prime} = \Phi_\epsilon^\prime
>0$. The convexity of the $\Psi_\epsilon$ implies that $(v-u) \Phi_\epsilon(u) \leq \Psi_\epsilon(v) - \Psi_\epsilon(u)$.
Therefore the last inequality also leads to
\begin{align}
  \label{epsdiss2}
 \frac12 \Vert \rho_\epsilon(t) - \gamma_0 \Vert_{-1}^2 
 +&  \int_s^t \int_{\mathcal O} \big( \Psi_\epsilon(\rho_\epsilon(r,x))- \Psi_\epsilon(\gamma_0(x)) \big) dx dr \\
 & \qquad     + \int_s^t \big\langle \eta_\epsilon(r) , \partial_x (-\partial_{xx}^2)^{-1}  (\rho_\epsilon(r) - \gamma_0) \big\rangle dr
   \leq \frac12 \Vert \rho_\epsilon(s) -\gamma_0\Vert_{-1}^2.  \nonumber
\end{align}
By direct computation, 
\begin{align}
  \label{PsiID}
  \Psi_\epsilon(r) = 
  \begin{cases}
    \frac12  \big( r \log r - r ) + C_\epsilon r  - \frac12 ( \epsilon \log \epsilon - \epsilon)  - \epsilon C_\epsilon 
    + \Theta_\epsilon(\epsilon), & \text{ if } r \geq \epsilon \\
    \Theta_\epsilon(r), & \text{ if } 0 \leq  r   < \epsilon   \\
    \frac{1}{2 \epsilon} r^2, & \text{ if } r < 0.
  \end{cases}
\end{align} 
From $(t \theta_\epsilon)^\prime = t \theta_\epsilon^\prime + \theta_\epsilon \geq \theta_\epsilon$ for $t\geq 0$, we obtain the
estimate $\Theta_\epsilon(t) \leq t \theta_\epsilon(t)$. This implies in particular that
\begin{align*}
  -  \frac12 ( \epsilon \log \epsilon - \epsilon ) - \epsilon C_\epsilon + \Theta_\epsilon(\epsilon) \to 0,
  \text{ as } \epsilon \to 0^+.
\end{align*}
Integrating the solution of~\eqref{appCPDE}, we also arrive at the conservation property $\langle 1, \rho_\epsilon(t)\rangle =
\langle 1, \rho_{\epsilon,0}\rangle =1$.  We decompose $\rho_\epsilon$ into positive and negative parts,
\begin{align*}
  \rho_\epsilon(t,x) = \rho_\epsilon^+(t,x) -\rho_\epsilon^-(t,x).
\end{align*}
Then, when the $\gamma_0 \in \sfX$ is a probability measure satisfying $S(\gamma_0)<\infty$, we have
\begin{align*}
  & \int_0^t\int_{\mathcal O} \Big( \Psi_\epsilon(\rho_\epsilon(r,x)) - \Psi_\epsilon(\gamma_0(x)) \Big) dx dr \\
  & =\int_0^t\int_{\mathcal O} \frac12 \Big( \rho_\epsilon^+(r,x) \log \rho_\epsilon^+(r,x) - \gamma_0(x) \log \gamma_0(x) \Big) dx dr \\
  & \qquad + \int_0^t\int_{\mathcal O}   \Big(\frac{| \rho_\epsilon^-(r,x)|^2}{2\epsilon} + (\frac12 - C_\epsilon) \rho_\epsilon^{-}(r,x)
  \Big) dx dr + o_\epsilon(1).
\end{align*}
We note that  
\begin{align*}
  \frac{r^2}{2\epsilon} - C_\epsilon r \geq  \frac{r^2}{4 \epsilon}   - \epsilon C_\epsilon^2 \quad
  \text{ and } \quad \sqrt{\epsilon} C_\epsilon \to 0.
\end{align*}
Therefore, the above estimates combined with~\eqref{epsdiss2} give a useful control on the amount of negative mass of $\rho_\epsilon$:
\begin{align}
  \label{negM}
  \sup_{\epsilon >0} \int_0^t\int_{\mathcal O}  \frac{1}{\epsilon} | \rho_\epsilon^-(r,x)|^2 dx dr <\infty.
\end{align}

 Third, we show the following property.
\begin{lemma} 
The $\{ \rho_\epsilon  \}_\epsilon$ is relatively compact in 
$C\big([0,\infty); H_{-1}(\mathcal O)\big)$. Hence, selecting subsequence if necessary, there exists a limiting curve $\rho(\cdot) \in C\big([0,\infty); H_{-1}(\mathcal O)\big)$ such that
\begin{align}
  \label{rhoepslim}
  \lim_{\epsilon \to 0^+} \sup_{0\leq t \leq T} \Vert \rho_\epsilon(t) - \rho(t)\Vert_{-1} =0, \quad \forall T>0.
\end{align}
\end{lemma}
\begin{proof}
We verify relative compactness of the $\{ \rho_\epsilon(\cdot) \}_{\epsilon>0}$ through  Arzel\'a-Ascoli lemma. The proof would be easier if $\rho_\epsilon(t) \in \sf\sfX$, since the $\sfX$ is a compact space. However, for each fixed $\epsilon>0$, our construction allow the possibility of negative mass in $\rho_\epsilon(t)$ even though $\rho_\epsilon(0) \in \sf\sfX$. The negative masses only vanish in the $\epsilon \to 0$ limit.

First, we verify the existence of a compact subset $K_1 \subset \subset H_{-1}(\mathcal O)$ such that 
\begin{align}\label{Kcontain}
 \rho_\epsilon(t) \in K_1, \quad \forall \epsilon >0, t \in [0,T].
\end{align}
We start with a compact set $K_0 := \{ \rho_0, J_\epsilon* \rho_0 : \epsilon >0\} \subset \subset  H_{-1}(\mathcal O)$.
Then for every $\delta >0$, there exists a finite positive integer $N:=N(\delta) \in \N$ and $\rho_{1,0}, \ldots, \rho_{N,0} \in C^\infty(\mathcal O) \cap \sfX$ such that $K_0 \subset \cup_{k=1}^N B(\rho_{k, 0}; \delta)$. Let $\rho_{\epsilon, k}(t)$ be the solution to 
\begin{align*}
 \partial_t \rho_{\epsilon, k} = \partial_{xx}^2 \Phi_\epsilon(\rho_{\epsilon,k}) + \partial_x \eta_\epsilon, \quad \rho_{\epsilon,k}(0) = \rho_{k,0}.
\end{align*}
By a contraction estimate in Chapter~6.7.2 of~\cite{Vaz07} (see also part~(iii) of Theorem~6.17 there),
\begin{align}\label{contrhok}
\sup_{t \in [0,T]}  \Vert \rho_{\epsilon, k}(t) -\rho_{\epsilon}(t) \Vert_{-1}   \leq \Vert \rho_{\epsilon,k}(0) - \rho_{\epsilon}(0) \Vert_{-1}, \quad \forall \epsilon >0.
\end{align}
By \eqref{VazEnEst}, noting $\rho_{k,0} \in \mathcal P(\mathcal O)$, for every $t \in [0,T]$, we have
\begin{align*}
&  \sup_{\epsilon>0} \Big(\int_{\mathcal O} \Psi_\epsilon(\rho_{\epsilon,k}(x,t)) dx - C_\epsilon \Big) \\
&\qquad  \leq \sup_{\epsilon>0}\sup_{k=1,\ldots,N} \Big( \int_{\mathcal O} \Psi_\epsilon(\rho_{k,0}) dx - C_\epsilon \Big) + \int_0^T\int_{\mathcal O} |\eta(r,x)|^2 dx dr \\
& \qquad =: L \big(\rho_{1,0},\ldots, \rho_{N(\delta),0}; \eta(\cdot)\big) =: L_\delta<\infty.
\end{align*}
In view of the explicit form of $\Psi_\epsilon$ in \eqref{PsiID}, the set
\begin{align*}
K_{1,\delta}(l):= \Big\{ \gamma \in H_{-1}(\mathcal O) : \int_{\mathcal O} \gamma(x) dx =1,  
 \sup_{\epsilon>0} \Big( \int_{\mathcal O} \Psi_\epsilon(\gamma) dx - C_\epsilon \Big)
   \leq l\Big\}
\end{align*}
is relatively compact in $H_{-1}(\mathcal O)$ for every finite $l \in \R_+$.
Denote $K_{1,\delta}:=K_{1,\delta}(L_\delta)$, then 
\begin{align*}
 \rho_{k,\epsilon}(t) \in K_{1,\delta}, \quad \forall \epsilon>0, k \in \{ 1,\ldots, N(\delta)\}, t \in [0,T].
\end{align*}
Let $K_{1,\delta}^\delta$ denote $\delta$-thickening set of the $K_{1,\delta}$.
Then by \eqref{contrhok},
\begin{align}\label{rhoeK1}
 \rho_\epsilon(t) \in K_{1,\delta}^\delta, \quad \forall \delta >0, \epsilon >0, t \in [0,T].
\end{align}
Taking $K_1:= \overline{\cap_{\delta>0} K_{1,\delta}^\delta}$ (which is complete and totally bounded), we arrive at \eqref{Kcontain}. 
 
Second,  through variational inequality \eqref{epsdiss2}, we obtain a local uniform modulus of continuity estimate $\sup_{\epsilon>0}\sup_{t, s \in [0,T], |t-s|\leq 1}\Vert \rho_\epsilon(t) - \rho_\epsilon(s) \Vert_{-1} \leq \omega(|t-s|)$
 for some modulus $\omega$. It is sufficient to verify that, for every $\delta \in (0,1)$, there exists a finite positive number $C_\delta>0$ and $\alpha \in (0,1)$ such that 
\begin{align*}
\Vert \rho_\epsilon(t) - \rho_\epsilon(s) \Vert_{-1} \leq \delta + C_\delta  |t-s|^\alpha, \quad \forall \epsilon>0, 0\leq s \leq t\leq T.
\end{align*}
Then we conclude by taking the $\omega(r):= \inf_{\delta \in (0,1)} \{ \delta + C_\delta  r^\alpha\}$. Let $\delta \in (0,1)$ be given. For every $s,t \in [0,T]$ with $|s-t|\leq1$ and $\epsilon>0$,  by \eqref{rhoeK1}, there exists $\gamma:= \gamma(\delta,\epsilon,s) \in K_{1,\delta}(L_\delta)$ such that $\Vert\rho_\epsilon(s) - \gamma\Vert_{-1} <\delta$. By \eqref{epsdiss2}, 
\begin{align*}
\frac12 \Vert \rho_\epsilon(t) - \gamma \Vert_{-1}^2 
& \leq \frac12 \delta^2 + C_{1,\delta} \sqrt{|t-s|}
 + \int_s^t  (\int_{\mathcal O}\Psi_\epsilon(\gamma)dx - C_\epsilon ) dr \\
 &   \leq \frac12 \delta^2 + C_{1,\delta}  \sqrt{|t-s|} + L_\delta |t-s|.
\end{align*}
Consequently  
 \begin{align*}
  \Vert \rho_\epsilon(t) - \rho_\epsilon(s) \Vert_{-1} 
  \leq   \Vert \rho_\epsilon(t) -\gamma \Vert_{-1} + \Vert \rho_\epsilon(s) -\gamma \Vert_{-1}  
  \leq 2\delta + C_{2,\delta} |t-s|^{\frac14}.
\end{align*}
Note that the $C_{2,\delta}$ only depends on $\delta$ and not on $\epsilon$, nor on $s,t$,  we conclude.
\end{proof}

\subsubsection{A priori regularities for the PDE with control~\eqref{CPDE}}

\begin{lemma}
  \label{apr0}
  Let $\rho_0 \in \sfX$ and $\rho(\cdot) \in C([0,\infty); H_{-1}(\mathcal O))$ be the limit as obtained
    from~\eqref{rhoepslim}. It then satisfies the following properties.
    \begin{enumerate}
    \item It holds that $\rho(r) \in \sfX$ for every $t \geq 0$. Indeed, $\rho(r,dx)=\rho(r, x) dx$ for $r>0$ a.e., and
      \begin{align}\label{posrho}
        \rho(r,x) \geq 0, \text{ a.e. } (r,x) \in (0,\infty) \times {\mathcal O}.
      \end{align}
    \item The variational inequality~\eqref{disineq} holds.
    \item   $\int_0^T \int_{\mathcal O} \rho(t,x) \log \rho(t,x) dx dt <\infty$.
    \item  $\rho(\cdot) \in C([0,\infty); \sfX)$.
    \end{enumerate}
\end{lemma}

\begin{proof}
Taking the limit $\epsilon \to 0$ in~\eqref{rhoepslim}, by the approximate variational inequality
estimates~\eqref{epsdiss2}, the negative mass estimate~\eqref{negM}, and lower semicontinuity arguments, we conclude that $\rho(r,dx) = \rho(r,x) dx$ for $r >0$ a.e.,
that~\eqref{posrho} holds (hence $\rho(r) \in \sfX$ for all $r\geq 0$), and that~\eqref{disineq} holds.

The estimate $\int_0^T S(\rho(t)) dt < \infty$ follows from~\eqref{disineq}.  
 \end{proof}
 
We remark that the variational inequalities~\eqref{disineq} alone (for a family of $\gamma_0 \in \sfX$ with $S(\gamma_0)<\infty$)
can be used as a definition of a solution for~\eqref{CPDE}.  This definition would suffice to establish a uniqueness result, as we
now show.

\begin{lemma}
  \label{CPDEsta}
  Let $(\rho_i, \eta_i)$, $i=1,2$ solve~\eqref{CPDE}--\eqref{CPDEcost} in the sense that both pairs satisfy the variational
  inequalities~\eqref{disineq}. In addition, we assume that $\rho_i(\cdot) \in C([0,T];\sfX)$ for every $T>0$ and $i=1,2$. Then
  \begin{align}
    \label{cntr12}
    \Vert \rho_1(t) -\rho_2(t) \Vert_{-1}   \leq \Vert \rho_1(0) - \rho_2(0) \Vert_{-1}
    +  \int_0^t \Vert \eta_1 - \eta_2 \Vert_{L^2(\mathcal O)} dr.
  \end{align}
  Hence, given a fixed initial condition and the same control $\eta =\eta_1=\eta_2$, it follows that $\rho_1=\rho_2$.
\end{lemma}

\begin{proof}
Let $0<\alpha <\beta <T$ and $0<s<t<T$. From~\eqref{disineq}, 
\begin{align*}
    & \int_\alpha^\beta \big( \Vert \rho_1(t) - \rho_2(\tau) \Vert_{-1}^2
  - \Vert \rho_1(s) - \rho_2(\tau) \Vert_{-1}^2 \big) d \tau \\
    & \qquad   \leq \int_\alpha^\beta \int_s^t \big( S(\rho_2(\tau)) -S(\rho_1(r))
  + 2 \langle \eta_1(r), \partial_x (-\partial_{xx}^2)^{-1} (\rho_1(r) - \rho_2(\tau)) \rangle \big) d r d\tau.
\end{align*}
Similarly,
\begin{align*}
   & \int_s^t \big( \Vert \rho_1(r) - \rho_2(\beta) \Vert_{-1}^2 - \Vert \rho_1(r)
  - \rho_2(\alpha) \Vert_{-1}^2 \big) d r \\
   & \qquad   \leq \int_s^t \int_\alpha^\beta \big( S(\rho_1(r)) - S(\rho_2(\tau))
  + 2 \langle \eta_2(\tau), \partial_x (-\partial_{xx}^2)^{-1} (\rho_2(\tau) -\rho_1(r)) \rangle \big) d\tau dr.
\end{align*}
Adding these two inequalities, we obtain
\begin{align}
  \label{coup12}
  &  \int_\alpha^\beta \big( \Vert \rho_1(t) - \rho_2(\tau) \Vert_{-1}^2 - \Vert \rho_1(s) - \rho_2(\tau) \Vert_{-1}^2 \big) d \tau \\
  &  \qquad \qquad \qquad 
  +  \int_s^t \big( \Vert \rho_1(r) - \rho_2(\beta) \Vert_{-1}^2 - \Vert \rho_1(r) - \rho_2(\alpha) \Vert_{-1}^2 \big) d r \nonumber \\
  &  \qquad \qquad  \leq \int_{r=s}^t\int_{\tau=\alpha}^\beta 2 \langle \eta_1(r)
  -\eta_2(\tau), \partial_x (-\partial_{xx}^2)^{-1} (\rho_1(r) - \rho_2(\tau)) \rangle d r d\tau. \nonumber
\end{align}
We define
\begin{align*}
  F(t,s;\beta, \alpha) &:= \int_s^t \int_\alpha^\beta \Vert \rho_1(r) - \rho_2(\tau) \Vert_{-1}^2 d\tau d r, \\
  M(t,s;\beta,\alpha) &:= \int_{r=s}^t\int_{\tau=\alpha}^\beta 2 \langle \eta_1(r)-\eta_2(\tau),
  \partial_x (-\partial_{xx}^2)^{-1} (\rho_1(r) - \rho_2(\tau)) \rangle d \tau  d r.
\end{align*}
Then~\eqref{coup12} becomes
\begin{align*}
  \partial_t F + \partial_s F + \partial_\beta F + \partial_\alpha F \leq M.
\end{align*}
If we write $G(h) := F(t+h, s+h; t+h, s+h) \in C^1(\R_+)$, then the last inequality becomes
\begin{align*}
  \partial_h G(h) \leq M(t+h,s+h;t+h,s+h).
\end{align*}
That is,
\begin{align*} 
  &  \int_s^t \int_s^t  \Vert \rho_1(r+h) - \rho_2(\tau+h) \Vert_{-1}^2 d \tau  d r
  -  \int_s^t \int_s^t \Vert \rho_1(r) - \rho_2(\tau) \Vert_{-1}^2   d\tau d r \\
  &\quad \quad \quad \quad  = G(h) - G(0)  \\
  & \quad\leq \int_0^h \int_{s}^t\int_s^t 2 \langle \eta_1(r+q)-\eta_2(\tau+q), \partial_x (-\partial_{xx}^2)^{-1} (\rho_1(r+q)
  - \rho_2(\tau+q)) \rangle d \tau  d r   dq.
\end{align*}
We multiply by $(t-s)^{-2}$ on both sides and then take the limit $t \to s^+$ to find
\begin{align*}
&  \Vert \rho_1(s+h) - \rho_2(s+h) \Vert_{-1}^2 -   \Vert \rho_1(s) - \rho_2(s) \Vert_{-1}^2 \\
& \leq 
2 \int_0^h  \Vert \eta_1(s+q)-\eta_2(s+q) \Vert_{L^2} \Vert \rho_1(s+q) - \rho_2(s+q) \Vert_{-1} dq \\
&= 2 \int_s^{s+h}  \Vert \eta_1(r)-\eta_2(r) \Vert_{L^2} \Vert \rho_1(r) - \rho_2(r) \Vert_{-1} dr. 
\end{align*}
Here we used the fact that $\rho_i(\cdot) \in C(\R_+;H_{-1}(\mathcal O))$ for each $i=1,2$.  By further
mollification-approximation estimates, we find the above inequality is equivalent to~\eqref{cntr12}.
\end{proof}

Definition~\ref{DefCPDE} gives a notion of weak solution for the partial differential equation~\eqref{CPDE}. It requires an
\emph{a priori} estimate that $\log \rho(t,x)$ is locally integrable, so that this quantity can be viewed as a distribution
(see~\eqref{wCPDE}). Next, we establish this local integrability estimate for the limit $\rho$ obtained from~\eqref{rhoepslim}.
We note that from the estimates in Lemma~\ref{apr0}, we already know that $\int_0^T \int_{\mathcal O} \rho(r,x) \log \rho(r,x) dx
dr <\infty$, which implies in particular that
\begin{align}
  \label{logposrho}
  \int_{r \in [0,T]} \int_{x : \rho(r,x) \geq 1}  \log \rho(r,x) dx dr
  =  \int_0^T \int_{\mathcal O}  \big( \log \rho(r,x) \big)^+ dx dr <\infty.
\end{align}
Therefore, we need to focus on the case where $\rho(r,x) <1$.

\begin{lemma}
  \label{apr1}
  Let $\rho(\cdot) \in C([0,\infty); H_{-1}(\mathcal O))$ be the limit as obtained from~\eqref{rhoepslim}.  For every $0\leq s\leq
    T<\infty$, allowing the possibility of $S(\rho(0))=+\infty$, we have that
    \begin{align}
      \label{EEprod}
      & \frac12 S(\rho(T))  + \frac18 \int_s^T \Big( - \log 2 + \int_{\mathcal O} (-\log) \rho(r,x) dx \Big) dr  \\
      & \qquad \leq \frac12 S(\rho(s) ) + \int_s^T \int_{\mathcal O} |\eta(r,x)|^2 dx dr.  \nonumber
    \end{align}
    Furthermore, in view of~\eqref{logposrho} and Lemma~\ref{apr0},
    \begin{align*}
      \int_s^T \int_{\mathcal O} (\log \rho\big)^- dx dr <\infty, \quad \forall s>0.
    \end{align*}
\end{lemma}

\begin{proof}
Noting $\int_{\mathcal O} \rho_\epsilon(t, x) dx =1$ for all $t > 0$, we have $\max_x \rho_\epsilon(t,x) >1/2$.  Since $(t,x)
\mapsto \rho_\epsilon(t,x)$ is continuous, we can select a family of points $\{ x_\epsilon(t) \in \mathcal O : t>0\}$ such that
$\rho_\epsilon(t, x_\epsilon(t)) \geq 1/2$.  We also observe that
\begin{align}
  \label{Pineq}
  \sup_{z \in \mathcal O} \Big(\Phi_\epsilon(\rho_\epsilon(t,x_\epsilon(t))) - \Phi_\epsilon(\rho_\epsilon(t, z)) \Big)
  \leq \Big(\int_{\mathcal O} |\partial_x \Phi_\epsilon(\rho_\epsilon(t,x))|^2 dx\Big)^{\frac12} \sup_{z \in \mathcal O} |z-x_\epsilon(t)|^{\frac12}.
\end{align}
Next, we estimate the left hand side of the above in three situations, namely
\begin{align*}
  \rho_\epsilon(t,z) \geq \epsilon, \quad 0 \leq \rho_\epsilon(t,z) \leq \epsilon, \quad 
  \rho_\epsilon(t,z) < 0.
\end{align*}
We note that $\Phi_\epsilon(\rho_\epsilon(t,x_\epsilon(t))) \geq -\frac12 \log 2 + C_\epsilon$. Therefore
 \begin{align*}
   \sup_{z :  \rho_\epsilon(t,z) \geq \epsilon} \Big(\Phi_\epsilon(\rho_\epsilon(t,x_\epsilon(t))) - \Phi_\epsilon(\rho_\epsilon(t, z)) \Big) 
   \geq  -\frac12 \log 2 - \frac12 \log \inf_{z : \rho_\epsilon(t,z) >\epsilon} \rho_\epsilon(t,z)  .
\end{align*}
In addition,
\begin{align*}
  C_\epsilon - \Phi_\epsilon(r) \geq C_\epsilon - \theta_\epsilon(\epsilon) = -\frac12 \log \epsilon, \quad \forall r \in [0,\epsilon],
 \end{align*}
 which implies
 \begin{align*}
   \sup_{z : 0\leq \rho_\epsilon(t,z) <\epsilon} \Big(\Phi_\epsilon(\rho_\epsilon(t,x_\epsilon(t))) - \Phi_\epsilon(\rho_\epsilon(t, z)) \Big) 
   \geq  -\frac12 \log 2 - \frac12 \log \epsilon.
\end{align*}
 Therefore, when $\epsilon >0$ is small enough,~\eqref{Pineq} gives (using the convention that $\sup \emptyset = -\infty$),
\begin{align*}
g_\epsilon(t)&:= \Big( -\frac12 \log 2 - \frac12 \log \big(\inf_{z \in \mathcal O} \rho_\epsilon(t,z) \vee \epsilon\big) \Big)
  \vee \Big( (-\frac12 \log 2 + C_\epsilon) +\frac{1}{\epsilon}  \sup_{z: \rho_\epsilon <0} \rho_\epsilon^{-}(t,z) \Big)     \\
&\qquad \qquad \qquad \leq \sqrt{2}   \Big( \int_{\mathcal O} |\partial_x \Phi_\epsilon(\rho_\epsilon(t,x))|^2 dx\Big)^{\frac12}. 
\end{align*}
%From second term on the left, we conclude that finiteness for the right hand side implies $\{z: \rho_\epsilon(t,z) < 0\} = \phi$ is empty.
Combined with~\eqref{VazEnEst}, this yields
\begin{align}
  \label{gbdd}
  \int_{\mathcal O} \Psi_\epsilon(\rho_\epsilon(x, T)) dx + \frac12 \int_s^T   g_\epsilon^2(r) dr
  \leq \int_{\mathcal O} \Psi_\epsilon(\rho_\epsilon(s,x)) dx + \int_s^T \int_{\mathcal O} |\eta_\epsilon(r,x)|^2 dx dr.
\end{align}
Using $ \int_{\mathcal O} \big( - \log \big)(\rho_\epsilon \vee \epsilon) dx \leq \big( - \log \big)(\inf_{\mathcal O}
\rho_\epsilon \vee \epsilon)$, we conclude
\begin{align}
  \label{AEEprod}
  & \int_{\mathcal O} \Psi_\epsilon(\rho_\epsilon(x, T)) dx + \frac18 \int_s^T   \Big( -\log 2
  + \int_{\mathcal O} (-\log) (\rho_\epsilon \vee \epsilon) dx \Big) dr \\
  &  \qquad  \leq \int_{\mathcal O} \Psi_\epsilon(\rho_\epsilon(s,x)) dx + \int_s^T \int_{\mathcal O} |\eta_\epsilon(r,x)|^2 dx dr. \nonumber
\end{align}

Now we pass $\epsilon \to 0$ in the above inequality to conclude~\eqref{EEprod}. The details are given in the following steps.
First, we note that
\begin{align*}
  \Vert \rho_\epsilon(t) \vee \epsilon -\rho_\epsilon(t) \Vert_{L^\infty} \leq 2 \epsilon + \sup_{z} \rho_\epsilon^{-}(t,z).
\end{align*}
Hence by the convergence in~\eqref{rhoepslim} and by the estimate~\eqref{gbdd}, 
\begin{align*}
  \lim_{\epsilon \to 0} \int_s^T \langle \rho_\epsilon \vee \epsilon,   \varphi \rangle dt 
  =  \int_s^T\langle \rho,   \varphi \rangle dt, \quad \forall \varphi \in C(\mathcal O).
\end{align*}
The observation
\begin{align*}
  -\log r = \sup_{s>0} \big( (- s) r + 1 + \log s\big), \quad r >0,
\end{align*}
leads to the variational formula
\begin{align}
  \label{logvar}
  \int_s^T \int_{\mathcal O} -\log \big(\rho_\epsilon \vee \epsilon\big) dz dt
  = \sup_{\varphi \in C(\mathcal O), \varphi >0} \Big( \int_s^T  \big( \langle \rho_\epsilon \vee \epsilon,  
  - \varphi \rangle  + 1 + \langle 1,  \log \varphi \rangle\big) \Big) dt.
\end{align}
Consequently
\begin{align}
  \label{logLSC}
  \int_s^T \int_{\mathcal O} \big(-\log \big) \rho(t,z) dz dt \leq \liminf_{\epsilon \to 0^+} \int_s^T \int_{\mathcal O} -\log \big(\rho_\epsilon \vee \epsilon\big) dz dt.
\end{align}
Second, by Jensen's inequality, $\Vert \eta_\epsilon(r) \Vert_{L^2} \leq \Vert \eta(r)\Vert_{L^2}$. Finally, to get~\eqref{EEprod}
from~\eqref{AEEprod} by taking $\epsilon \to 0^+$, noting the identification of $\Psi_\epsilon$ in~\eqref{PsiID}, all we need is
to justify the inequality
\begin{align}\label{PhiJen}
 \limsup_{\epsilon \to 0} \int_{\mathcal O}  \Psi_\epsilon(\rho_\epsilon(s,x)) dx +\frac12  - C_\epsilon  \leq \frac12 \int_{\mathcal O} \rho(s,x) \log \rho(s,x) dx, \quad \forall s \geq 0, \quad \text{a.e.}
\end{align}
If $s=0$, then this follows directly by convexity/Jensen inequality arguments,
\begin{align*}
  \int_{\mathcal O} \big(J_\epsilon * \rho_0\big) \log \big(J_\epsilon * \rho_0 \big)dx \leq \int_{\mathcal O} \rho_0(x) \log \rho_0(x) dx.
\end{align*}
The case of $0<s<T<\infty$ is more subtle. We divide the justifications into three steps. In step one, we construct the solutions
$\{\rho_\epsilon(r) : 0\leq r < s\}$ as before with $\rho_\epsilon(0) : =J_\epsilon * \rho_0$ and take its limit
$\{\rho(r): 0\leq r < s\}$. Then we construct $\{ \hat{\rho}_\epsilon(r) : s \leq T\}$ as solution to~\eqref{appCPDE} with initial
data $\hat{\rho}_\epsilon(s) := J_\epsilon * \rho(s)$ and concatenate $\rho_\epsilon$ with $\hat{\rho}$ to arrive at a new
$\sfX$-valued curve $\{ \tilde{\rho}_\epsilon(r) : 0 \leq r \leq T\}$.  This new curve is defined on $r \in [0,T]$, but may have a
discontinuity at time $s>0$. We proceed to the second step next. All arguments and estimates before~\eqref{PhiJen} in the proof of
this lemma still hold if we replace $\rho_\epsilon$ by $\tilde{\rho}_\epsilon$. Hence, for the concatenated curve,~\eqref{PhiJen}
still holds. In the last step, we note that $\{ \tilde{\rho}_\epsilon \colon \epsilon >0\}$ and $\{ \rho_\epsilon : \epsilon >0\}$
have the same limit $\rho$. This holds by the stability-uniqueness result in Lemma~\ref{CPDEsta}.  Therefore,~\eqref{PhiJen} is
verified for the curve $\tilde{\rho}_\epsilon$.

We conclude that~\eqref{EEprod} holds for the limit $\rho$.
\end{proof}

\begin{lemma}
  The energy estimate~\eqref{SIfluc} holds. 
\end{lemma}

\begin{proof}
Our strategy is to derive~\eqref{SIfluc} by passing to the limit $\epsilon \to 0^+$ in~\eqref{VazEnEst}.

Let $\rho$ be the limit of the the sequence of functions $\rho_\epsilon$ as in~\eqref{rhoepslim}. From~\eqref{VazEnEst}, the
following holds for every $0<s <T$:
\begin{align}
  \label{grest}
  & \int_s^T \Big(\liminf_{\epsilon \to 0} \int_{\mathcal O}|\partial_x \Phi_\epsilon(\rho_\epsilon(r,x))|^2 dx \Big)dr \\
  & \qquad  \leq \liminf_{\epsilon >0} \int_s^T  \int_{\mathcal O} |\partial_x \Phi_\epsilon(\rho_\epsilon(r,x))|^2 dx dr 
    \leq \int_s^T \int_{\mathcal O}|\eta(r,x)|^2 dx dr + S(\rho(s)). \nonumber
\end{align}
Suppose that we can find a set $\mathcal N \subset [s,T]$ of Lebesgue measure zero and functions
$[s,T] \ni r \mapsto k_\epsilon(r) \in \R$ such that for every $r \in [s,T]\setminus {\mathcal N}$, there exists a subsequence
(still labeled by $\epsilon:=\epsilon(r)$) with 
\begin{align}
  \label{DcnvPsi}
  \lim_{\epsilon \to 0^+} \langle \Psi_\epsilon (\rho_\epsilon(r,\cdot)) + k_\epsilon(r), \partial_x \varphi \rangle
  =\langle \frac12 \log \rho(r,\cdot), \partial_x \varphi \rangle, \quad \forall \varphi \in C^\infty(\mathcal O).
\end{align}
Then
\begin{align*}
  \int_{\mathcal O} |\partial_x \frac12 \log \rho(r, x) |^2dx  \leq  \liminf_{\epsilon \to 0} \int_{\mathcal O} |\partial_x \Phi_\epsilon(\rho_\epsilon(r,x)) |^2 dx, \quad \text{for a.e. } r \in [s,T],
\end{align*}
hence we conclude. 

We establish~\eqref{DcnvPsi} next. Let
\begin{align*}
  h(r):=  \liminf_{\epsilon \to 0} \Vert \partial_x \Phi_\epsilon(\rho_\epsilon(r,\cdot)) \Vert_{L^2(\mathcal O)}^2 .
\end{align*}
By~\eqref{grest}, we can find a set $\mathcal N \subset [s,T]$ of Lebesgue measure zero such that
\begin{align}
  \label{hrfinite}
  h(r) <\infty, \quad \forall r \in [s,T] \setminus {\mathcal N}.
\end{align}
Therefore for $r \in [s,T]\setminus {\mathcal N}$, there exists a subsequence $\epsilon:=\epsilon(r)$, and there exist constants
$k_\epsilon :=k_\epsilon(r)$ such that $\{ \Phi_\epsilon(\rho_\epsilon(r,\cdot)) + k_\epsilon \}_{\epsilon>0}$ is relatively
compact in $C(\mathcal O)$. Let
\begin{align*}
  u(r,x):= \lim_{\epsilon \to 0^+} \Phi_\epsilon(\rho_\epsilon(r,x)) + k_\epsilon , 
\end{align*}
where the convergence (along the selected subsequence) is uniform in $\mathcal O$.  We claim that the set
$\{x : \rho_\epsilon(r,x) <0\} = \emptyset$, when $\epsilon$ is small enough. Suppose this is not true. Then by continuity of
$x \mapsto \rho_\epsilon(r,x)$, we can find $\tilde{x}_\epsilon(r)$ such that $\rho_\epsilon(r, \tilde{x}_\epsilon(r)) =0$. We
also recall that in the proof of Lemma~\ref{apr1}, we can find $x_\epsilon(r)$ such that
$\rho_\epsilon(r, x_\epsilon(r)) \geq 1/2$. Hence
\begin{align*}
  \Phi_\epsilon(\rho_\epsilon(r,x)) - \Phi_\epsilon(\rho(r,\tilde{x}_\epsilon)) \geq \Phi_\epsilon(\epsilon) - 0 
  = - \log \epsilon \to +\infty, \quad \text{ as } \epsilon \to 0. 
\end{align*}
But on the other hand, the estimate~\eqref{hrfinite} implies that
\begin{align*}
  \liminf_{\epsilon \to 0} \sup_{x \in \mathcal O} |\Phi_\epsilon(\rho_\epsilon(r,x)) - \Phi_\epsilon(0)| <\infty.
\end{align*}
This contradiction allows us to conclude that
\begin{align*}
  \liminf_{\epsilon \to 0} \Big( \inf_{z \in \mathcal O} \rho_\epsilon(r,z)\Big) 
  = \liminf_{\epsilon \to 0} \Big(\inf_{z \in \mathcal O} \rho_\epsilon^+(r,z)\Big) 
  =\liminf_{\epsilon \to 0} \Big( \inf_{z \in \mathcal O} \rho_\epsilon^+(r,z) \vee \epsilon \Big) >0.
\end{align*}
Therefore $u(r,x) = \lim_{\epsilon \to 0^+}\big( \frac12 \log (\rho_\epsilon(r,x)) + (k_\epsilon + C_\epsilon)\big)$, where the
convergence is uniform in $x$.  That is, along this subsequence of $\epsilon =\epsilon(r)$,
\begin{align*}
  \lim_{\epsilon \to 0} \Vert e^{2(k_\epsilon +C_\epsilon)}\rho_\epsilon(r,\cdot)- e^{2u(\cdot)}\Vert_{L^\infty} =0.
\end{align*}
In view of the weak convergence in~\eqref{rhoepslim}, $u(r,x) = \frac12 \log \rho(r,x) + C_0$ for some constant $C_0:=C_0(r)$.
Hence we verified~\eqref{DcnvPsi}.
\end{proof}

\subsection{A posteriori estimates for the PDE with control}
\label{sec:post-estim-PDE}

We see in Lemma~\ref{CPDEsta} that the notion of solutions in terms of the variational inequalities~\eqref{disineq} implies
uniqueness and stability. Next, we prove that the weak solution in the sense of Definition~\ref{DefCPDE} implies that these
variational inequalities hold, and hence uniqueness and stability follow.

\begin{lemma}
  \label{apost}
  Every weak solution $(\rho,\eta)$ to~\eqref{CPDE}--\eqref{CPDEcost} also satisfies~\eqref{disineq}.
\end{lemma}

\begin{proof}
  From~\eqref{wCPDE}, simple approximations shows that the following holds for all $0<s<t$:
\begin{align}
\label{TwSol}
  \langle \varphi(t), \rho(t)   \rangle - \langle \varphi(s), \rho(s)   \rangle = 
  \int_s^t \big( \langle \partial_r \varphi(r), \rho(r)   \rangle + \langle \varphi(r), \frac12 \partial_{xx}^2 \log \rho (r) 
  + \partial_x \eta \rangle \big) dr
\end{align}
for all smooth $\varphi :=\varphi(r,x) $ which includes in particular the choice
\begin{align*} 
  \varphi_\epsilon(r):=  G_\epsilon *(-\Delta_{xx})^{-1} (\rho(r) - \gamma_0)  ;
\end{align*}
here, $G_\epsilon:=J_\epsilon* J_\epsilon$ where $J_\epsilon$ is a smooth mollifier and the $*$ denotes convolution in the spatial
variable only.
 
Therefore
\begin{align*}
  \Vert J_\epsilon*(\rho(t) - \gamma_0) \Vert_{-1}^2 & = \Vert J_\epsilon * (\rho(s) - \gamma_0) \Vert_{-1}^2 
                                                       - \int_s^t \int_{\mathcal O} (G_\epsilon) * \rho(r,x) \log \rho(r,x) dx dr \\
  & \qquad + \int_s^t \int_{\mathcal O} (G_\epsilon) * \gamma_0(r,x) \log \rho(r,x) dx dr \\
  & \qquad \qquad -2 \int_s^t\int_{\mathcal O} ( G_\epsilon)*\eta(r,x) \partial_x (-\partial_{xx}^2)^{-1} (\rho(r) - \gamma_0) (x) dx dr.
\end{align*}
We note that, by Jensen's inequality,
\begin{align*}
  \int_{\mathcal O} (G_\epsilon * \gamma_0) (x) \log \rho(r,x) dx
  &= \int_{\mathcal O} (G_\epsilon * \gamma_0) (x) \log \frac{\rho(r,x)}{G_\epsilon * \gamma_0(x)} dx
    + S (G_\epsilon * \gamma_0) \\
  & \leq  \log \big(\int_{\mathcal O} 1 dx \big) +  S(\gamma_0)  \leq S(\gamma_0).
 \end{align*}
 With the \emph{a priori} regularity estimates~\eqref{Sest}--\eqref{Dlog}, passing to the limit $\epsilon \to 0$, we
 obtain~\eqref{disineq} for $s>0$. Taking $s \to 0+$, the case $s=0$ follows.
\end{proof}

\begin{lemma}
  \label{smalltS}
  Suppose that $(\rho, \eta)$ is the weak solution to~\eqref{CPDE}-\eqref{CPDEcost} in the sense of Definition~\ref{DefCPDE}. Let
  \begin{align*}
    \phi(t):= t S(\rho(t)) + \Vert \rho(t) - \gamma_0 \Vert_{-1}^2,
  \end{align*}
  then
  \begin{align*}
    \phi(t) \leq \phi(0) + \int_0^t \big( \frac12 r + 2 \Vert \rho(r) - \gamma_0 \Vert_{-1} \big) \Vert \eta(r)\Vert_{L^2} dr.
  \end{align*}
\end{lemma}

\begin{proof}
  Following the ideas exposed in Theorem~24.16 of Villani~\cite{V09} in a similar setting, we combine~\eqref{disineq}
  and~\eqref{SIfluc} to obtain the desired estimate.
\end{proof}

We define a set of regular points in the state space $\sfX$,
\begin{align}
  \label{R}
  \text{Reg}&:= \Big\{ \rho \in \sfX : \rho(dx) = \rho(x) dx, \text{ some measurable } \rho(x),\\
            &  \qquad \quad  \log \rho \in L^1(\mathcal O), \int_{\mathcal O}|\partial_x \log \rho|^2 dx <\infty \Big\} \nonumber \\
            &= \Big\{ \rho \in \sfX : \rho(dx) = \rho(x) dx, \rho \in C(\mathcal O), \inf \rho >0, 
              \int_{\mathcal O}|\partial_x \log \rho|^2 dx <\infty \Big\}, \nonumber
\end{align}
where the last equality follows from one-dimensional Sobolev inequalities.  The estimates in Lemma~\ref{apriori} implies that
under finite control cost~\eqref{CPDEcost}, the weak solution of~\eqref{CPDE} has the property that it spends zero Lebesgue time
outside the set $\text{Reg}$. That is,
\begin{align}
  \label{pathreg}
  \rho(t) \in \text{Reg}, \quad \forall t>0, \quad  \text{a.e.}
\end{align}

\subsection{A Nisio semigroup}
\label{sec:Nisio-semigroup}

We recall that 
\begin{align*}
  L(\rho,\eta) := \frac12 \int_{\mathcal O} |\eta(x)|^2 dx.
\end{align*}
For every $(\rho, \eta)$ satisfying~\eqref{CPDE}--\eqref{CPDEcost} in the sense of Definition~\ref{DefCPDE}, by the regularity
results established in Lemma~\ref{apriori}, $\log \rho(t,x) \in \mathcal D^\prime ((0,T) \times \mathcal O)$ exists as a
distribution, and
\begin{align*}
  \int_0^t \int_{\mathcal O} |\eta(t,x)|^2 dx ds = \int_0^t \Vert \partial_s \rho - \frac12 \partial_{xx}^2 \log \rho \Vert_{-1}^2 ds.
\end{align*}

Let $f\colon \sfX \mapsto \R \cup \{+\infty\}$ with $\sup_\sfX f <\infty$, we define
\begin{align}
  V(t)f(\rho_0) &:= \sup \Big\{ f(\rho(t)) - \int_0^t L \big(\rho(r),\eta(r) \big) dr :
                  (\rho(\cdot), \eta(\cdot)) \text{ satisfies }~\eqref{CPDE}  \label{Vdef} \\
                & \qquad \qquad  \text{ in the sense of Definition~\ref{DefCPDE} with } \rho(0) = \rho_0 \Big\}, 
                  \quad  \forall \rho_0 \in \sfX. \nonumber
\end{align}
It follows that $\sup_\sfX V(t) f <\infty$. Moreover, $V(t) C = C$ for any $C \in \R$.

We define an \emph{action functional} on curves $\rho(\cdot) \in C([0,\infty); \sfX)$, thus giving a precise meaning
to~\eqref{ACT}:
\begin{align}
  \label{Adef}
  A_T[\rho(\cdot)]:=  
  \begin{cases} 
    \int_0^T\frac12  \Vert \partial_t \rho - \frac12 \partial_{xx}^2 \log \rho \Vert_{-1}^2 dt, 
    & \text{ if } \int_{\mathcal O}| \log \rho(t,x)| dx <\infty,  \text{ a.e. } t>0,\\
    +\infty, & \text{ otherwise.}
  \end{cases}
\end{align}
and
\begin{align}
  \label{Qpotdef}
  A[\rho_0, \rho_1;T]:=\inf \{ A_T [\rho(\cdot)] : \rho(\cdot)  \in C([0,\infty);\sfX), \rho(0)=\rho_0, \rho(T)=\rho_1  \}.
\end{align}

We recall that $R_\alpha$ is defined in~\eqref{Resolve} before giving regularity results for $V(t)$ and $R_\alpha$.
\begin{lemma}
  \label{CVR}
  For $h \in C_b(\sfX)$, $V(t) h\in C_b(\sfX)$ for all $t \geq 0$ and $R_\alpha h \in C_b(\sfX)$ for every $\alpha >0$.
\end{lemma}

\begin{proof}
These claims are consequences of the stability Lemma~\ref{CPDEsta}. The proofs resemble those in standard-finite dimensional
control problem.  Hence we only prove the claim that $V(t)h \in C_b(\sfX)$.

Let $\rho_0, \gamma_0 \in \sfX$. For any $\epsilon >0$, there exists
$(\rho(\cdot), \eta(\cdot)):=(\rho_\epsilon(\cdot), \eta_\epsilon(\cdot))$ and
$(\gamma(\cdot), \eta(\cdot)):= (\gamma_\epsilon(\cdot), \eta_\epsilon(\cdot))$ satisfying~\eqref{CPDE}--\eqref{CPDEcost} with
$\rho(0) =\rho_0$ and $\gamma(0)=\gamma_0$, and the contraction estimate~\eqref{cntr12} holds. Consequently,
\begin{align*}
V(t) h(\rho_0) - V(t) h(\gamma_0) &\leq   \epsilon + h(\rho(t)) - h(\gamma(t)) 
\leq \epsilon + \omega_h\big(\sfd(\rho(t), \gamma(t))\big) \\
&\leq  \epsilon + \omega_h(\sfd(\rho_0, \gamma_0)).
\end{align*}
Since $\epsilon>0$ is arbitrary, it follows that $V(t) h \in C_b(\sfX)$.
\end{proof}

\begin{lemma}[Nisio semigroup]
  \label{Nisio}
  The family of operators $\{  V(t): t \geq 0\}$ has the following properties:
  \begin{enumerate}
  \item It forms a nonlinear semigroup on $C_b(\sfX)$,
    \begin{align*}
      V(t)V(s) f = V(t+s) f, \quad \forall t, s \geq 0, f \in C_b(\sfX).
    \end{align*}
  \item The semigroup is a contraction on $C_b(\sfX)$:  for every $f,g \in C_b(\sfX)$, we have
    \begin{align*}
      \Vert V(t) f - V(t) g \Vert_{L^\infty(\sfX)} \leq \Vert f - g \Vert_{L^\infty(\sfX)}.
    \end{align*}
    In fact,
    \begin{align*}
      \sup_\sfX \big( V(t) f - V(t) g \big) \leq \sup_\sfX \big(f - g \big).
    \end{align*}
  \item The resolvent is a contraction on $C_b(\sfX)$:  for every $h_1, h_2 \in C_b(\sfX)$, we have
    \begin{align*}
      \Vert R_\alpha h_1 - R_\alpha h_2 \Vert_{L^\infty(\sfX)} \leq \Vert h_1 - h_2\Vert_{L^\infty(\sfX)}.
    \end{align*}
    Moreover, if $h_1$ is bounded from above and $h_2$ satisfies~\eqref{hclass} and is bounded from below, then
    \begin{align}
      \label{Rcntra}
      \sup_\sfX ( R_\alpha h_1 - R_\alpha h_2 ) \leq \sup_\sfX (h_1 - h_2).
    \end{align}
    If $h_1$ is bounded and $h_2$ is bounded from above, then
    \begin{align*}
      \inf_\sfX (R_\alpha h_1 - R_\alpha h_2) \geq \inf_\sfX (h_1 - h_2).
    \end{align*}
 \end{enumerate}
\end{lemma}

\begin{proof}
  The semigroup property follows from standard reasoning in dynamical programming. The contraction properties follow from an
  $\epsilon$-optimal control argument applied to the definition of $V(t)$ in~\eqref{Vdef}, similar to the proof of the last lemma.
\end{proof}

 \begin{lemma}
   Let $\rho_0 \in \sfX$. We have
   \begin{align*}
     V(t) f(\rho_0) = \sup_{\rho_1 \in \sfX} \Big\{ f(\rho_1) - A[\rho_0, \rho_1; t] \Big\}, \quad \forall \sup_\sfX f<\infty.
   \end{align*}
   If in addition $f \in C_b(\sfX)$, then there exists $\rho^t(\cdot) \in C([0,t]; \sfX)$ or equivalently, there exists
   $(\rho^t(\cdot), \eta(\cdot))$ satisfying \eqref{CPDE}-\eqref{CPDEcost} in the sense of Definition~\ref{DefCPDE}, with
   $\rho^t(0) =\rho_0$ and satisfying the variational inequality~\eqref{disineq}, such that
   \begin{align}
     \label{Vattain}
     V(t) f(\rho_0) = f(\rho^t(t)) - A_t[\rho^t(\cdot)] = f(\rho^t(t)) - \int_0^t \int_{\mathcal O} \frac12 |\eta(r,x)|^2 dx dr.
   \end{align}
\end{lemma}

\begin{proof}
By the definition of $V(t) f$ in~\eqref{Vdef}, we can find a sequence of $(\rho_n(\cdot), \eta_n(\cdot))$
satisfying~\eqref{CPDE} in the weak sense of Definition~\ref{DefCPDE} with $\rho_n(0) =\rho_0$ and such that
\begin{align*}
  V(t) f(\rho_0) = \lim_{n \to \infty} \Big( f(\rho_n(t)) - \int_0^t \int_{\mathcal O}\frac12 |\eta_n(r,x)|^2 dxdr \Big).
\end{align*}
We note that $V(t) 0 = 0$. The contraction property in Lemma~\ref{Nisio} gives
$\Vert V(t) f \Vert_{L^\infty(\sfX)} \leq \Vert f \Vert_{L^\infty(\sfX)}$, which in turn implies that
\begin{align}
  \label{C}
  C:=\sup_n \int_0^t \int_{\mathcal O} |\eta_n(r,x)|^2 dxdr <\infty.
\end{align}
Therefore there exists an $\eta$ such that $\eta_n \rightharpoonup \eta$ weakly in $L^2((0,T) \times \mathcal O)$. This implies in
particular
\begin{align}
  \label{etanFatou}
  \int_0^t \int_{\mathcal O} |\eta(r,x)|^2 dxdr \leq \liminf_{n \to \infty} \int_0^t \int_{\mathcal O} |\eta_n(r,x)|^2 dxdr.
\end{align}
Therefore, it suffices to show that $\{ \rho_n(\cdot) : n =1,2,\ldots \}$ is relatively compact in $C([0,\infty); \sfX)$, and that
any limit point $\rho_\infty(\cdot)$ of a convergent subsequence satisfies the regularity estimates~\eqref{Sest}--\eqref{Dlog}.
 
Since $(\sfX, \sfd)$ is a compact metric space, the relative compactness of the curves $\{ \rho_n(\cdot) : n=1,2,\ldots\}$ follows
from a uniform modulus of continuity which we show to hold. First, for each $n$,~\eqref{disineq} implies that
\begin{align*}
  \sup_n \sup_{0\leq t \leq T} \Vert \rho_n(t)\Vert_{-1} <\infty.
\end{align*}
Second, by~\eqref{smalltS}, this implies that
\begin{align*}
  \sup_n S(\rho_n(t)) < \frac{M_1}{t}<\infty , 
\end{align*}
where the $M_1>0$ is a finite constant. Using this, we obtain from~\eqref{disineq} the estimate
\begin{align*}
  \Vert \rho_n(t) - \rho_n(s)\Vert_{-1}^2 \leq \frac{M}{s} (t-s) + M_2 \int_s^t \Vert \eta_n(r) \Vert_{L^2} dr, \quad \forall 0<s<t.
\end{align*}
Hence there exists a modulus $\omega_1 : [0, \infty) \mapsto [0,\infty)$ with $\omega_1 (0+)=0$ such that
\begin{align*}
  \Vert \rho_n(t) - \rho_n(s)\Vert_{-1} \leq s^{-1} \omega_{1} (|t-s|), \quad \forall 0<s <t.
\end{align*}
Third, we derive a short-time estimate of a uniform modulus of continuity.  From the weak solution property coupled with the
\emph{a priori estimates} in the definition of a solution in Definition~\ref{DefCPDE}, we can derive a variant of~\eqref{disineq}
by approximation arguments:
\begin{align*}
  &  \frac12 \Vert \rho_n(t) - \gamma_0 \Vert_{-1}^2 - \frac12 \Vert \rho_n(s) - \gamma_0 \Vert_{-1}^2 \\
  &  =   \int_s^t \big( \frac12 \langle \rho_n(r)-\gamma_0, - \log \rho_n(r) \rangle - \langle \partial_x 
    (-\Delta_{xx}^2)^{-1} (\rho_n(r) - \gamma_0), \eta\rangle \big) dr \\
  & \leq  \int_s^t \big( \frac12 \langle \rho_n(r)-\gamma_0, - \log \gamma_0 \rangle - \Vert \rho_n(r) 
    - \gamma_0\Vert_{-1} \Vert \eta\Vert_{L^2} \big) dr \\
  & \leq  \int_s^t  \Vert \rho_n(r)-\gamma_0\Vert_{-1}  \big( \frac12 \Vert \partial_x \log \gamma_0 \Vert_{L^2} 
    + \Vert \eta\Vert_{L^2} \big) dr.
\end{align*}
In the last two lines above, the first inequality follows from the fact that $\log$ is a monotonically increasing function. The
previous estimate implies that
\begin{align*}
   \Vert \rho_n(t) - \gamma_0 \Vert_{-1} -   \Vert \rho_n(s) - \gamma_0 \Vert_{-1} 
 \leq  \int_s^t \big( \frac12 \Vert \partial_x \log \gamma_0 \Vert_{L^2} + \Vert \eta\Vert_{L^2} \big) dr.
\end{align*}
Consequently, for any $\epsilon >0$, by a density argument we can find a $\gamma_{0,\epsilon} \in \sfX$ such that
$S(\gamma_{0,\epsilon})<\infty$ and $\Vert \gamma_{0,\epsilon} - \rho_0 \Vert_{-1}<\epsilon$.  Taking
$\gamma_0:= \gamma_{0,\epsilon}$ and $s=0$, we have
\begin{align*}
\sup_{0\leq  r \leq t} \Vert \rho_n(r) -\rho_0\Vert_{-1} \leq \inf_{\epsilon \in (0,1)} \Big( 2 \epsilon +   \frac12 \Vert \partial_x  \log \gamma_{0,\epsilon} \Vert_{L^2} t 
+   \sqrt{C} \sqrt{t}\Big) =: \omega_2(t).
\end{align*}
The constant $C$ is the one from~\eqref{C}. Fourth, by the existence of $\omega_i, i=1,2$, we conclude the existence of a uniform
modulus $\omega$ such that
\begin{align*}
 \sup_{0\leq  r \leq t} \Vert \rho_n(r) -\rho_0\Vert_{-1} \leq \omega (t).
\end{align*}

The entropy estimate~\eqref{Sest} for the $\rho(\cdot)$ follows from the fact that each $\rho_n(\cdot)$ satisfies~\eqref{disineq},
and that $\rho \mapsto S(\rho)$ is lower semicontinuous in $(\sfX, \sfd)$. Similarly,~\eqref{logrho} holds for $\rho(\cdot)$
since~\eqref{EEprod} holds for each $\rho_n$ and $\rho \mapsto \int_{\mathcal O} (-\log) \rho) dx$ is lower semicontinuous in
$(\sfX, \sfd)$ (see the proof of~\eqref{logvar}). Finally,~\eqref{Dlog} holds for $\rho(\cdot)$ because~\eqref{SIfluc} holds for
each $\rho_n(\cdot)$.
 \end{proof}

\begin{lemma}
  \label{resoID}
  For every $\alpha, \beta >0$, we have
  \begin{align*}
    R_\alpha h = R_\beta \Big( R_\alpha h - \beta \frac{R_\alpha h - h }{\alpha}\Big), \quad \forall h \in C_b(\sfX).
  \end{align*}
\end{lemma}

\begin{proof}
  The idea of proof in Lemma~8.20 of~\cite{FK06} applies in this context. Because of the special structure of the problem here,
  the use of a relaxed control there is not necessary.
\end{proof}

Let $f_0$ and $H_0f_0$ be as in~\eqref{DH0def}--\eqref{H0def}. Note that the estimate in~\eqref{disineq} holds, that
$f_0 - \alpha H_0 f_0$ is lower semicontinuous and bounded from below, and moreover that it satisfies for $\alpha >0$
\begin{align*}
  \int_0^t e^{-\alpha^{-1}r} \Big(  \frac{(f_0 - H_0 f_0)(\rho(r))}{\alpha} - \frac12 \int_{\mathcal O} |\eta(r,x)|^2 dx  \Big) dr 
  < + \infty
\end{align*}
for every $(\rho(\cdot), \eta(\cdot))$ solving the controlled PDE~\eqref{CPDE} with~\eqref{CPDEcost}. Therefore
\begin{align*}
  R_\alpha (f_0 - \alpha H_0 f_0) \colon \sfX \mapsto  \R\cup\{+ \infty\}
\end{align*}
is a well defined function, it is bounded from below.

Similarly, let $f_1$ and $H_1 f_1$ be defined as in~\eqref{DH1def}--\eqref{H1def}.  Since $\sfX$ is compact,
$f_1 - \alpha H_1 f_1$ is bounded from above. It is also upper semicontinuous and bounded from above. Hence
\begin{align*}
  R_\alpha (f_1 - \alpha H_1 f_1) \colon \sfX \mapsto \R \cup\{-\infty\}
\end{align*}
is a well defined function and it is bounded from above.

With these estimates, we prove a variant of Lemma~8.19 in Feng and Kurtz~\cite{FK06}.
\begin{lemma}
  \label{Rcmp}
  For every $\alpha >0$,
  \begin{align}
    R_\alpha (f_0 - \alpha H_0 f_0) &\leq f_0, \label{Rceq1}\\
    R_\alpha (f_1 - \alpha H_1 f_1) & \geq  f_1. \label{Rceq2}
  \end{align}
\end{lemma}

\begin{proof}
We note that, for every $\eta \in L^2(\mathcal O)$, 
\begin{align*}
  &  H_0 f_0(\rho) + L(\rho, \eta) \\
  & = \frac{k}{2} \big(S(\gamma) -S(\rho)\big) +  \sup_{\hat{\eta} \in L^2(\mathcal O)} \Big( \langle -k \partial_x (-\partial_{xx}^2)^{-1}(\rho -\gamma),   \hat{\eta} \rangle  
        -\frac12 \int_{\mathcal O} |\hat{\eta}|^2 dx \Big) +\frac12 \int_{\mathcal O} |\eta|^2 dx   \\
  & \geq  \frac{k}{2} \big(S(\gamma) -S(\rho)\big) 
    +  k \langle  -\partial_x (-\partial_{xx}^2)^{-1}(\rho -\gamma), \partial_x \eta \rangle.
\end{align*}
By the \emph{a priori} estimate~\eqref{disineq}, for every $(\rho(\cdot), \eta(\cdot))$ solving the PDE~\eqref{CPDE} in the sense
of Definition~\ref{DefCPDE} with~\eqref{CPDEcost} and the initial condition $\rho(0) =\rho_0$, we have
\begin{align*}
  \int_0^t\Big( H_0 f_0(\rho(r)) + L\big(\rho(r), \eta(r)\big)\Big) dr \geq f_0(\rho(t)) - f_0(\rho_0), \quad \forall t>0.
\end{align*}
Moreover, 
\begin{align*}
  &  \int_0^t e^{-\alpha^{-1}s} \Big( \frac{\big(f_0 - \alpha H_0 f_0\big)(\rho(s))}{\alpha} -L(\rho(s),\eta(s)) \Big) ds \\
  & =  \int_0^t \alpha^{-1} e^{-\alpha^{-1}s} f_0(\rho(s)) ds -\int_{s=0}^\infty \alpha^{-1} e^{-\alpha^{-1} s} \int_{r=0}^{s \wedge t} 
        \big( H_0 f_0(\rho(r)) + L(\rho(r),\eta(r)) \big) dr ds \\
  & \leq  \int_0^t \alpha^{-1} e^{-\alpha^{-1}s} f_0(\rho(s)) ds -\int_0^\infty \alpha^{-1} e^{-\alpha^{-1} s} 
           \big( f_0(\rho(s \wedge t)) - f_0(\rho_0) \big) ds \\
  &= f_0(\rho_0) - e^{-\alpha^{-1} t} f_0(\rho(t)).
\end{align*}
In view of~\eqref{Resolve}, we conclude~\eqref{Rceq1}.

Next, we show~\eqref{Rceq2}. For any $\gamma \in \sfX$ and $\rho \in \sfX$ in the definition of $f_1:=f_1(\gamma)$ as
in~\eqref{DH1def}, we define $\eta := - k \partial_x (-\partial_{xx}^2)^{-1} (\rho - \gamma)$. Then
\begin{align*}
  &  H_1 f_1(\gamma) + L(\gamma, \eta) \\
  &= \frac{k}{2} \big( S(\gamma) - S(\rho) \big) + \langle -k \partial_x (-\partial_{xx}^2)^{-1}(\rho -\gamma), \eta\rangle 
      - \frac12 \int_{\mathcal O} |\eta(x)|^2 dx
      + \frac12 \int_{\mathcal O} |\eta(x)|^2 dx \\
  & = \frac{k}{2} \big( S(\gamma) - S(\rho) \big) + \langle -k \partial_x (-\partial_{xx}^2)^{-1}(\rho -\gamma), \eta\rangle.
\end{align*}
We consider the unique solution $\gamma:= \gamma(t)$ to 
\begin{align*}
  \partial_t \gamma = \frac12 \partial_{xx}^2 \log \gamma + \partial_x \eta = \frac12 \partial_{xx}^2 \log \gamma 
  + k (\rho - \gamma), \quad \gamma(0)=\gamma_0,
\end{align*}
where the $\rho$ is such that $S(\rho)<\infty$. Then, in view of the estimate~\eqref{disineq} (with the roles of $\rho$ and
$\gamma$ swapped)
\begin{align*}
  &  \int_{r=0}^{s} \big( H_1 f_1(\gamma(r)) + L(\gamma(r),\eta(r)) \big) dr  \\
  & = \int_0^s \Big( \frac{k}{2} \big( S(\gamma(r)) - S(\rho) \big) 
       + \langle -k \partial_x (-\partial_{xx}^2)^{-1}(\rho -\gamma(r)), \eta(r)\rangle \Big) dr\\
  & \leq   f_1(\gamma(s)) - f_1(\gamma(0)).
\end{align*}
Consequently
\begin{align*}
  &  \sup \Big\{ \int_0^t e^{-\alpha^{-1}s} \Big( \frac{\big(f_1 - \alpha H_1 f_1\big)(\hat{\rho}(s))}{\alpha} 
      -L(\hat{\rho}(s),\hat{\eta}(s)) \Big) ds : \\
  &  \qquad \qquad \qquad \qquad  (\hat{\rho},\hat{\eta}) \text{ solves }~\eqref{CPDE}-\eqref{CPDEcost}, \hat{\rho}(0)=\rho_0 \Big\} \\
  & \geq   \int_0^t e^{-\alpha^{-1}s} \Big( \frac{\big(f_1 - \alpha H_1 f_1\big)(\rho(s))}{\alpha} -L(\rho(s),\eta(s)) \Big) ds \\
  & =   \int_0^t \alpha^{-1} e^{-\alpha^{-1}s} f_1(\rho(s)) ds -\int_{s=0}^\infty \alpha^{-1} e^{-\alpha^{-1} s} \int_{r=0}^{s \wedge t} 
        \big( H_1 f_1(\rho(r)) + L(\rho(r),\eta(r)) \big) dr ds \\
  & =   \int_0^t \alpha^{-1} e^{-\alpha^{-1}s} f_1(\rho(s)) ds -\int_{s=0}^\infty \alpha^{-1} e^{-\alpha^{-1} s} 
        \big(f_1(\gamma(s\wedge t)) - f_1(\gamma_0)\big) ds\\
  &= f_1(\gamma_0) - e^{-\alpha^{-1}t} f_1(\gamma(t)).
\end{align*}
Sending $t \to \infty$, we conclude the proof of~\eqref{Rceq2} and have thus established the lemma.
\end{proof}

\begin{lemma}
  \label{Sec:Exist}
  Let $\alpha>0$ and $h \in C_b(\sfX)$. We denote $f:=R_\alpha h \in C_b(\sfX)$ (Lemma~\ref{CVR}). Then $f$ is a viscosity
  sub-solution to~\eqref{H0eqn} with the $h_0$ replaced by $h$, it is also a viscosity super-solution to~\eqref{H1eqn} with the
  $h_1$ replaced by $h$.
\end{lemma}

\begin{proof}
The proof follows from the proof of part~(a) of Theorem~8.27 in Feng and Kurtz~\cite{FK06}, using Lemmas~\ref{resoID}.  We only
give details for the sub-solution case. The conditions on $H_0 f_0$ is different than those imposed on ${\bf H}_{\dagger}$
in~\cite{FK06}. However, in view of the improved contraction estimate~\eqref{Rcntra} and because that $f_0 - \beta H_0 f_0$
satisfies~\eqref{hclass} (see the \emph{a priori} estimate~\eqref{disineq}), the proof can be repeated almost verbatim.
 
Let $f_0, H_0 f_0$ be defined as in~\eqref{DH0def},~\eqref{H0def}. Then $f_0$ is bounded from below and for every $\beta >0$
\begin{align*}
  \sup_\sfX (f  - f_0) &= \sup_\sfX (R_\alpha h - f_0) \\
                       &\leq \sup_\sfX \Big(R_\beta \big(R_\alpha - \beta \alpha^{-1}(R_\alpha h - h )\big)  
                              - R_\beta \big( f_0 - \beta H_0 f_0\big)\Big) \\
                       &\leq  \sup_\sfX \Big( R_\alpha h - \beta \alpha^{-1}(R_\alpha h - h )\big)  
                               - \big( f_0 - \beta H_0 f_0\big)\Big) \\
                       &= \sup_\sfX \Big( f - f_0 - \beta \big( \frac{f- h}{\alpha} - H_0 f_0\big)\Big)
\end{align*}
In this estimate, the first inequality follows from Lemmas~\ref{resoID} and~\ref{Rcmp}; the second inequality follows
from~\eqref{Rcntra}.  By the arbitrariness of the $\beta >0$, the sub-solution property follows from Lemma~7.8
of~\cite{FK06}. Note that $f \in C_b(\sfX)$ and $f_0 \in LSC(\sfX; \R \cup\{+\infty\})$ and
$H_0f_0 \in USC(\sfX; \R \cup\{-\infty\})$.
\end{proof}

In view of the comparison principle we established in Theorem~\ref{CMP}, the above result allows us to conclude that Theorem~\ref{wellposed} holds.
  
Finally, we link the operator $R_\alpha$ with the semigroup $V$ by a product formula.
\begin{lemma}
  \label{genthm}
  Let $h \in C_b(\sfX)$, then  
  \begin{align*}
    V(t) h(\rho_0) = \lim_{n \to \infty} R_{n^{-1}}^{[nt]} h(\rho_0), \quad \forall \rho_0 \in \sfX.
  \end{align*}
\end{lemma}

\begin{proof}
  The proof of Lemma~8.18 in~\cite{FK06} applies here. The use of a relaxed control argument is not required in the current
  context.
\end{proof}

\section{An informal derivation of the Hamiltonian $H$ from stochastic particles}
\label{DerH}

In this section, we outline a non-rigorous derivation of the Hamiltonian~\eqref{H}.  A rigorous version of the theory requires
significant additional work and will be presented elsewhere. Here, we restrict ourselves to establishing a detailed but heuristic
picture. 

To explain our program in a nutshell, we start with a system of interacting stochastic particles which has been used as a simplified toy model for gas dynamics. This model leads to the Carleman equation as kinetic limit. Kurtz obtained the hydrodynamic
limit of this model, a nonlinear diffusion equation,~\cite{Kurtz73} in 1973, followed by work by McKean~\cite{McKean75}. We will
also use ideas of Lions and Toscani~\cite{LT97} in their work on this model. While these references are concerned with the
hydrodynamic limit, we are interested more broadly in fluctuations around this limit (which includes the hydrodynamic limit as
minimiser of the rate functional).

The system is a high-dimensional Markov process. In the hydrodynamic limit, the macroscopic particle density is described by a
probability measures on $\mathcal O$ satisfying a nonlinear diffusion equation. We aim to characterize \emph{both} the limit as
well as the fluctuations around it through an effective action minimization theory formulated as a path-integral. The
probabilistic large deviation theory gives us a mathematical framework for explaining this rigorously.

Following a method developed by Feng and Kurtz~\cite{FK06}, we establish the large deviations by studying the convergence of a
sequence of Hamiltonians derived from the underlying Markov processes. A critical step in the program is to prove comparison
principles for the limiting Hamiltonian. This is the motivation for the results presented in earlier sections of this paper.
Another critical step is the derivation of the limit Hamiltonian, which we present now informally. The main technique involved is
a singular perturbation method generalized to a setting of nonlinear PDEs in the space of probability measures.
 
\subsection{Carleman equations, mean-field version}
\label{sec:Carl-equat-mean}

We now describe the particle model studied by Kurtz, McKean and Lions and Toscani~\cites{Kurtz73,McKean75,LT97}.  On the unit
circle $\mathcal O$, we are given a fictitious gas consisting of particles with two velocities. The first particle type moves into
the positive $x$-direction and the second particle type in the negative direction, both with the same (modulus of) speed $c>0$.
Let $w_1(t,x)$ be density of the first particles type at time $t$ and at location $x$, and $w_2(t,x)$ be density of the second
type.  When particles collide, reactions occur if the types are the same; otherwise, particles move freely as if nothing happened.
The reaction happens at a rate $k>0$ and the reaction mechanism is simple --- they both switch to becoming the opposite type.  At
a mean-field level, we can express the above description in terms of a system of PDEs known as the Carleman equation
\begin{align}
  \label{meanfield}
  \begin{cases}
    \partial_t w_1 + c \partial_x w_1 = k(w_2^2-w_1^2), \\
    \partial_t w_2 - c \partial_x w_2 = k (w_1^2-w_2^2). 
  \end{cases}
\end{align}
Following Lions and Toscani~\cite{LT97}, we introduce the total mass density variable $\rho$ and the flux variable $j$:
\begin{align}
  \label{rhojuw}
  \rho:=w_1+w_2, \quad j:=c(w_1-w_2).
\end{align}
Then
\begin{align*}
  \begin{cases}
    \partial_t \rho + \partial_x j =0, \\
    \partial_t j + c^2 \partial_x \rho = - 2 k \rho j.
  \end{cases}
\end{align*}
We consider a hydrodynamic rescaling of the system by setting
\begin{align*}
  c=\epsilon^{-1}, \quad k=\epsilon^{-2}, \quad \text{ where } \epsilon \to 0.
\end{align*}
Then
\begin{align*}
  \partial_t \rho+  \partial_x j &= 0, \\
  \epsilon^2 \partial_t j + \partial_x \rho& = - 2 \rho j.
\end{align*}
The flux variable $j$ quickly equilibrates as $\epsilon \to 0$ to an invariant set indexed by the slow variable $\rho$:
\begin{align*}
  \partial_x \rho + 2 \rho j = 0.
\end{align*}
This very explicit density-flux relation enables us to close the description using the $\rho$-variable only, giving a nonlinear
diffusion equation
\begin{align}
  \label{Diffeqn}
  \partial_t \rho = \frac12 \partial_x \Big( \frac{\partial_x \rho}{\rho}  \Big).
\end{align}
The first rigorous derivation of~\eqref{Diffeqn} as limit from~\eqref{meanfield} was given by Kurtz~\cite{Kurtz73} in 1973 under
suitable assumptions on the initial data.  McKean~\cite{McKean75} improved the result by giving a different and more elementary
proof.  The change of coordinate to the pair $(\rho, j)$ in Lions and Toscani~\cite{LT97} appeared later but makes a two-scale
nature of the problem much more transparent.

\subsection{A microscopically defined stochastic Carleman particle model}
\label{sec:micr-defin-stoch}

The Carleman equation~\eqref{meanfield} is a mean field model without any fluctuation. We go beyond mean field model by adding
more details. One way of doing so would be to introduce explicitly a Lagrangian action, so that the Carleman
dynamic~\eqref{meanfield} appears as a critical point or minimizer in the space of curves. We will, however, pursue a different
implicit approach by introducing the action probabilistically using underlying stochastic particle dynamics.  There are more than
one possible choice for such a model. However, they all should have the following properties: one, such a model should give the
Carleman equation in the large particle number limit; two, the action should appear implicitly in the sense of the likelihood of
seeing a curve in the space of curves. That is, the higher the action, the less likely to see the curve.  This action can be
defined through a limit theorem as several parameters get rescaled (particle number, transport speed and reaction speed), as
in~\eqref{LDPform}.  The precise language to be used here are large deviations.  Caprino, De Masi, Presutti and
Pulvirenti~\cite{CDPP89} has considered such a stochastic particle model and studied its law of large number limit. We now study
the large deviation using a slight variation of their model.
 
We denote the phase space variable of an $N$-particle system
\begin{align*}
  ({\bf x,v }):=\Big( (x_1,v_1), \ldots, (x_N, v_N)\Big), \qquad (x_i, v_i) \in {\mathcal O} \times \R
\end{align*}
and define an operator $\Phi_{ij}$ in the phase space
\begin{align}
  \label{Phiij}
  \Phi_{ij} ({\bf x,v}) :=\Big( ( x_1, v_1), \ldots, (x_i, - v_i), \ldots, (x_j, -v_j), \ldots, (x_N,v_N)\Big), \quad i\neq j.
\end{align}
For $f:=f({\bf x,v})$ and $i \neq j$, with a slight abuse of notation, we denote
\begin{align}
  \label{Phiijf}
  (\Phi_{ij}f)({\bf x,v})& := & f\big( \Phi_{ij} ({\bf x,v}) \big):= f(x_1, v_1; \ldots; x_i, - v_i; \ldots, x_j, -v_j; 
                                \ldots, x_N,v_N) .
\end{align}
To model nearest neighbor interaction, we introduce a standard non-negative symmetric mollifier $\hat{J} \in C^\infty((-1,1);\R_+)$ with $\int_{x \in (-1,1)} \hat{J}(x) dx =1$, $\hat{J}(0)>0$. 
We denote
\begin{align*}
  \hat{J}_\theta(x)   := \theta^{-1}\hat{J}(\frac{x}{\theta}),   \quad 
  J_\theta(x)  := \sum_{k \in \Z} \hat{J}_\theta(x+k),
  \end{align*}
  and
  \begin{align*}
  \chi &:=\chi_N(x,y; v,u) := J_{\theta_N}\big(x-y\big) \delta_{v}(u).
\end{align*}
Let $\theta:= \theta_N \to 0$ slowly with $N \theta_N \to \infty$, and let $\tau := \tau_N \to 0$. 
We now describe the modification of the model studied by Caprino, De Masi, Presutti and Pulvirenti~\cite{CDPP89}. We consider a Markov process in
state space $\big( {\mathcal O} \times \{-1, +1\}\big)^N$ given by generator
\begin{align*}
  B_Nf({\bf x,v}) &:=
   c \sum_{i=1}^N v_i \partial_{x_i} f + \tau \sum_{i=1}^N \partial_{x_i}^2 f 
    + \frac{k}{2N} \sum_{\substack{i,j=1\\i \neq j}}^N \chi_N(x_i, x_j; v_i,v_j)
     \Big( \Phi_{ij} f({\bf x, v}) - f({\bf x,v}) \Big) \\
                  &= c \sum_{i=1}^N v_i \partial_{x_i} f + \tau \sum_{i=1}^N \partial_{x_i}^2 f 
                    + \frac{k}{2N} \sum_{\substack{i,j=1\\i \neq j}}^N J_{\theta_N}\big(x_i -x_j\big) 
  \delta_{v_i}(v_j)\Big( \Phi_{ij} f({\bf x, v}) - f({\bf x,v}) \Big).
\end{align*}
From a formal point of view, the parameter $\tau$ is unnecessary. However, it is essential for us to obtain useful \emph{a priori} estimates which allow an analysis of the limit passage. It was introduced in~\cite{CDPP89} to avoid a paradoxical feature
observed by Uchiyama~\cite{Uchi88} in the case of Broadwell equations.  This feature shows that particles at the same location
cannot be separated by the dynamics. Hence the kinetic limit $N \to \infty$ of the stochastic model without the term with $\tau$
does not converge to the Carleman equation as expected by formal computations.  We refer the reader to page~628 and Section~4
of~\cite{CDPP89} for more information on this point.

Let $({\bf X, V}):= \big( (X_1, V_1), \ldots, (X_N, V_N)\big)$ be the Markov process defined by the generator $B_N$. Moreover, we
denote the one-particle-marginal density
\begin{align*}
  \mu_N(dx,v;t):=P(X_i(t) \in dx, V_i(t) =v).
\end{align*}
Exploring propagation of chaos, through the BBGKY hierarchy, the authors of~\cite{CDPP89} proved that, as $N \to \infty$, $\mu_N$
has a (kinetic) limit $\mu:= \mu(dx,v; t):= \mu(x,v;t) dx$ satisfying
\begin{align*}
  \partial_t \mu + c v \cdot \partial_x \mu =k\Big( \mu^2(x, -v;t) - \mu^2(x,v;t)\Big), \quad \mu(0) = \mu_0.
\end{align*}
This is the Carleman system~\eqref{meanfield} if we take $w_1(t,x)=\mu(x,+1;t)$ and $w_2(t,x):=\mu(x,-1;t)$.
  
In order to understand the large deviation behavior, following~\cite{FK06}, we compute the following nonlinear operator
\begin{align*}
  H_B f({\bf x,v}) 
  &  :=    e^{- f} B_N e^{f}({\bf x,v}) \\
  &  =  c \sum_{i=1}^N v_i \partial_{x_i} f+ \tau \Big( \sum_{i=1}^N \partial_{x_i}^2 f+ \sum_{i=1}^N |\partial_{x_i}f|^2 \Big) \\
  &  \qquad    \qquad  + \frac{k}{2N} \sum_{\substack{i,j=1\\i \neq j}}^N
  J_{\theta_N}\big(x_i-x_j\big) \delta_{v_i}\big(v_j\big) \Big( e^{\Phi_{ij}f - f} -1 \Big).
\end{align*}
We define the empirical probability measure
\begin{align}
  \label{empirical}
 \mu(dx,dv):=  \mu_N\big(dx,dv\big):= \frac1N \sum_{i=1}^N \delta_{(x_i, v_i)}(dx, dv),
\end{align}
and choose a class of test functions which are symmetric under particle permutations,
\begin{align*}
  f({\bf x,v}) &:= f(\mu ):= \psi( \langle \varphi_1,\mu_N\rangle, \ldots,  \langle \varphi_M,\mu_N\rangle) \\
               & = \psi\Big( \frac1N \sum_{k=1}^N \varphi_1(x_k,v_k), \ldots,  \frac1N \sum_{k=1}^N \varphi_M(x_k,v_k) \Big).
\end{align*}
The test function $f$ can be abstractly thought of as a function in the space of probability measures with a typical element
denoted as $\mu$, hence the notation $f(\mu)$.  In the following, we use the traditional notation of a functional derivative,
\begin{align*}
  \frac{\delta f}{\delta \mu}(x,v) 
  := \sum_{l=1}^M \partial_l \psi ( \langle \varphi_1,\mu\rangle, \ldots,  \langle \varphi_M,\mu\rangle) \varphi_l(x,v), \quad 
  \forall (x,v) \in {\mathcal O} \times \{ \pm 1\}.
\end{align*}
For any test function $\varphi=\varphi(x,v)$, we define a collision operator which maps a function $\varphi$ of two variables
$(x,v)$ into a function $C \varphi$ of four variables $(x,v, x_*, v_*)$, as follows:
\begin{align*}
  \big( C \varphi\big) \big( (x,v);(x_*,v_*) \big):= \varphi(x,-v) - \varphi(x,v) + \varphi(x_*, -v_*) - \varphi(x_*, v_*).
\end{align*}
For a measure $\nu$ on $\mathcal O$ and $\theta>0$, we define its mollification
\begin{align*}
(J_\theta* \nu) (y) := \int_{y \in \mathcal O} J_\theta(y-z) \nu(d z) 
 = \sum_{k \in \Z} \int_{y \in \mathcal O} \hat{J}_\theta(y - z+k) \nu(dz).
\end{align*}
Then, direct computation leads to the estimate
\begin{align*}
  H_Nf(\mu)&:=\frac1N e^{-Nf} B_N e^{Nf}({\bf x,v}) = N^{-1} H_B (Nf)({\bf x,v}) \\
           &=c \langle -\big( v \partial_x \mu \big), \frac{\delta f}{\delta \mu} \rangle + 
             \frac{k}{2} \frac{1}{N^2} \sum_{\substack{i,j=1\\i \neq j}}^N J_{\theta_N}\big(x_i-x_j\big) \delta_{v_j}(v_i)  
  \Big(e^{C \frac{\delta f}{\delta \mu}(x_i,v_i;x_j,v_j)}-1 \Big)    + o_N(1) \\
  & = c \langle -\big( v \partial_x \mu \big), \frac{\delta f}{\delta \mu} \rangle + 
             \frac{k}{2} \sum_{v=+1,-1} \int_{x \in \mathcal O} (e^{C \frac{\delta f}{\delta \mu}(x,v;x,v)}-1 ) (J_{\theta_N}*\mu) (x,v) \mu(dx,v) + o_N(1). 
\end{align*}
In the last line above, we invoked the condition $N \theta_N \to +\infty$ to ensure that the diagonal terms $\sum_{i=j=1}^\infty$
have a negligible effect on the overall convergence. Assuming $\mu_N \to \mu $ in narrow topology where the $\mu(dx;v) = \mu(x,v) dx$, we then have
\begin{align*}
  H_Nf(\mu_N)       \to  c \langle -\big( v \partial_x \mu \big), \frac{\delta f}{\delta \mu} \rangle + 
             \frac{k}{2} \sum_{v=+1,-1} \int_{x \in \mathcal O} (e^{C \frac{\delta f}{\delta \mu}(x,v;x,v)}-1 ) \mu^2(x,v) dx.
\end{align*}

\subsection{Large deviation from the hydrodynamic limit}
\label{sec:Large-deviation-from}

We now consider the hydrodynamic scaling by taking $c:=\epsilon^{-1}$ and $k:= \epsilon^{-2}$, together with
$N:=N(\epsilon) \to \infty$.

To emphasize the two-scale nature of the problem, we switch to the density-flux coordinates:
\begin{align}
  \label{rhoj}
  \rho(dx) := \sum_{v =+1,-1} \mu(dx,v), \quad j(dx) :=\epsilon^{-1} \sum_{v=\pm1} v \mu(dx,v). 
\end{align}
The calculations in the coming paragraphs will heavily rely upon the simple relations:
\begin{align}
  \sum_{v=\pm1} v \mu^2(x,v) &= \mu^2(x,1) - \mu^2(x,-1) = \epsilon \rho(x) j(x) \label{rhojmu} 
                                 \intertext{and}
  \sum_{v=\pm1} \mu^2(x,v) &=\frac12\big( \rho^2(x) + \epsilon^2 j^2(x)\big). \nonumber
\end{align}
Let 
\begin{align*}
  \tilde{\varphi}_1(x,v):= \varphi_1(x), \quad \tilde{\varphi}_2(x,v) := c v  \varphi_2(x), 
  \quad \forall   \varphi_i \in C^1({\mathcal O}), i=1,2.
\end{align*}
Then
\begin{align*}
  \langle \varphi_1, \rho \rangle = \langle \tilde{\varphi}_1, \mu \rangle, \quad \langle \varphi_2, j \rangle 
  = \langle \tilde{\varphi}_2,\mu\rangle.
\end{align*}
We consider 
\begin{align*}
  f(\mu):=  f(\rho, j) :=\psi\big(\langle\varphi_1, \rho\rangle, \langle  \varphi_2, j \rangle \big) 
  =\psi\big( \langle \tilde{\varphi}_1, \mu \rangle, \langle \tilde{\varphi}_2, \mu \rangle\big),
\end{align*}
and then
\begin{align}
  \label{fmurhoj}
  \frac{\delta f}{\delta \mu}(x,v) = (\partial_1 \psi)  \varphi_1(x) + (\partial_2 \psi) cv \varphi_2(x)  
  = \frac{\delta f}{\delta \rho}(x) + \frac{1}{\epsilon} v \frac{\delta f}{\delta j}(x).
\end{align}
More generally, we can consider
\begin{align*}
  f(\mu)&:= f(\rho, j ) \\
        &:=\psi\big(\langle \varphi_{1,1}, \rho\rangle, \ldots, \langle \varphi_{1,M}, \rho \rangle;
             \langle  \varphi_{2,1}, j \rangle, \ldots,  \langle  \varphi_{2,M}, j \rangle \big) \\
        & =  \psi\big( \langle \tilde{\varphi}_{1,1}, \mu \rangle, \ldots, \langle \tilde{\varphi}_{1,M}, \mu \rangle; 
              \langle \tilde{\varphi}_{2,1}, \mu \rangle, \ldots, \langle \tilde{\varphi}_{2,M}, \mu \rangle  \big).
\end{align*}
The identity~\eqref{fmurhoj} relating derivatives of $\mu$ to those of $\rho$ and $j$ still holds:
\begin{align} \label{murhoj}
  \frac{\delta f}{\delta \mu}(x,v)=  \frac{\delta f}{\delta \rho}(x) + \frac{1}{\epsilon} v \frac{\delta f}{\delta j}(x) .
\end{align}
From this point on, we write $H_\epsilon:= H_{N(\epsilon)}$ to emphasize the dependence of $\epsilon$.  Then
\begin{align*}
  \big( H_\epsilon f \big)(\rho, j) = 
  &\langle \frac{\delta f}{\delta \rho}, -  \partial_x j\rangle + \frac{1}{\epsilon^2} 
                                        \langle \frac{\delta f}{\delta j},  - \partial_x \rho    \rangle  \\
                                      &  \qquad   + \frac{1}{2\epsilon^2} \sum_{v=\pm1} \int_{x \in \R} 
                                        \big( e^{\frac{1}{\epsilon}(-4v) \frac{\delta f}{\delta j}(x) }-1  \big) \mu^2(x,v)dx + o_\epsilon(1).
\end{align*}

Following the abstract theorems in Feng and Kurtz~\cite{FK06}, if we can derive a limit of the operator $H_\epsilon$ (which we
claim to be the $H$ in~\eqref{H}) and if we can prove the associated comparison principle, then
 \begin{align*}
   \lim_{\delta \to 0^+} \lim_{\epsilon \to 0^+} - \epsilon \log P \big( \rho_\epsilon(\cdot) \in B(\rho(\cdot); \delta)  \big) 
   =  A_T[\rho(\cdot)], \quad \forall \rho \in C([0,T]; \sfX),
\end{align*}
where the action functional $A_T$ is defined as in~\eqref{ACT}; here $B(\rho(\cdot);\delta)$ is a ball of size $\delta$ around
$\rho(\cdot)$ in $C([0,T]; \sfX)$ and $\sfX=\mathcal P(\mathcal O)$ is the metric space specified in the introduction. This will
then rigorously justify the formally statement~\eqref{LDPform}. We reiterate that the aim of this paper is to establish the one
challenging part, the comparison principle, rigorously (Section~~\ref{Sec:CMP}), which we now give heuristic arguments for the
other part, the convergence.

\subsection{Convergence of Hamiltonian operators, in a singularly perturbed sense}
\label{sec:Conv-Hamilt-oper}
We show that the Hamiltonian $H$ in~\eqref{H} is a formal limit for the sequence of operators given by $H_\epsilon$.  The
identification of $H$ is related to an infinite-dimensional version of ground state energy problem in~\eqref{cell}. We now
describe three different possible approaches.

We consider a class of perturbed test functions
\begin{align}
  \label{fepsilon}
  f_\epsilon(\rho, j) := f_0(\rho) + \epsilon^2 f_1(\rho, j).
\end{align}
It follows then
\begin{align}
  \label{Hexpan}
  \big( H_\epsilon f_\epsilon\big) (\rho, j) 
  & = {\mathbb H} (\rho, j; \frac{\delta f_0}{\delta \rho}, \frac{\delta f_1}{\delta j}) + o_\epsilon(1),
\end{align}
where
\begin{align}
  \label{defH}
  {\mathbb H}(\rho,j; \varphi, \xi) &:=   \langle \varphi, -  \partial_x j\rangle
  + \langle \xi,  - \partial_x \rho  -2 \rho j   \rangle   + 2 \int_x |\xi|^2   \rho^2    dx\\
                                    & =  \sfH(j,\xi; \rho) + \sfV (j; \varphi), \nonumber
\end{align}
with
\begin{align}
  \label{sfHdef}
  \sfH (j, \xi):= \sfH(j,\xi; \rho):= \langle \xi,  - \partial_x \rho  -2 \rho j   \rangle 
  +2 \int_x |\xi(x)|^2   \rho^2(x)  dx,  \quad  \forall \xi \in C_c^\infty(\mathcal O),
\end{align}
and
\begin{align*}
\sfV(j):= \sfV(j;  \varphi) := \langle \varphi, -\partial_x j \rangle = \langle \partial_x \varphi, j\rangle.
\end{align*}
We would like to have the limit in~\eqref{Hexpan} independent of the $j$-variable and thus want to make the $j$-variable to
disappear asymptotically.  We can choose the test function $f_1$ suitably to achieve this.

We introduce perturbed Hamiltonians in the $j$-variable
\begin{align*}
  \sfH_\sfV (j, \xi) := \sfH_\sfV (j, \xi; \rho, \varphi):=  \sfH (j, \xi) + \sfV(j).
\end{align*}
Then we seek solution to a stationary Hamilton-Jacobi equation in the $j$-variable
\begin{align}
  \label{cell}
  \sfH_\sfV(j,\frac{\delta f_1}{\delta j}) = H,
\end{align}
where $H$ is a constant in $j$ but may depend on $\rho$ through ${\sfH}(\cdot, \cdot;\rho)$ and on $\varphi$ through
$\sfV(\cdot;\varphi)$. We denote this dependence as
\begin{align*}
  H:= H(\rho, \varphi).
\end{align*}
Suppose that we can solve~\eqref{cell}, then  
\begin{align*}
  H_\epsilon f_\epsilon (\rho, j) = H \big(\rho; \frac{\delta f_0}{\delta \rho}\big) + o_\epsilon(1)
\end{align*}
and
\begin{align*}
  \lim_{\epsilon \to 0^+} H_\epsilon f_\epsilon = H f_0.
\end{align*}
Hence we can conclude our program. Next, we identify the Hamiltonian $H$ as the one defined in~\eqref{H} and show that the
associated Hamilton-Jacobi equation~\eqref{HJBmacro} (in the interpretation of Sections~\ref{Sec:CMP} and~\ref{Sec:Nisio}) is
solvable.

We now present three different approaches to identify $H$. We comment that, although we work with a specific model (the Carleman
particles) in this paper, our goal has been more ambitious. We would like to explore the scope of applicability of the Hamiltonian
operator convergence method in the context of hydrodynamic limits.  As this ambition is very general, we aim to present as many
ways of verifying the required conditions as possible.

\subsection{First approach to identify $H$ --- formal weak {KAM} method in infinite dimensions}
\label{sec:First-approach-to}

In \emph{finite} dimensions, equations of the type~\eqref{cell} have been studied in the weak KAM (Kolmogorov-Arnold-Moser) theory
for Hamiltonian dynamical systems. See Fathi~\cites{Fa97,Fa97b,Fa98}, E~\cite{E99}, Evans~\cites{Evans01,Evans04,Evans04b}, Evans
and Gomes~\cites{EG01,EG02}, Fathi and Siconolfi~\cites{FS04,FS05} and others; there is an unpublished book of
Fathi~\cite{FaBook}.  The existing literature focuses on finite-dimensional systems, mostly with compactness assumptions on the
physical space. Our setting is necessarily very different, as we have an infinite-dimensional non-locally compact state space. In
the following, we (formally) apply conclusions of the existing weak KAM theory to arrive at the representation
\begin{align}
  \label{barHE}
  H(\rho, \varphi)  =  \inf_{f_1}\sup_{j} \Big( \sfH(j,\frac{\delta f_1}{\delta j};\rho)  + \sfV(j;\rho,\varphi)\Big).
\end{align}
 
The representation~\eqref{barHE} can be made more explicit due to a hidden controlled gradient flow structure in $\sfH$
(see~\eqref{gradfloH}).  To present the ideas as clearly as possible, we introduce yet another set of coordinates by considering
\begin{align*}
  u(x):= \frac{d j}{d \rho}(x), \quad x \text{ a.e. } \rho.
\end{align*}
Then
\begin{align*}
  j(dx) = u(x) \rho(dx), \quad \frac{\delta f}{\delta u}(x) = \rho (dx) \frac{\delta f}{\delta j}(x).
\end{align*}
This motivates us to introduce new test functions $\phi:= \rho \xi$. 
Under the new coordinates, we have
\begin{align*}
  \sfV(u; \varphi, \rho) &=   \langle \varphi, - \partial_x (\rho u)\rangle , \\
  \sfH(u,\phi;\rho) & = \langle \phi, - \partial_x \log \rho - 2 \rho u \rangle + 2 \int_{\mathcal O} \big| \phi \big|^2 dx.
 \end{align*}
We define a free energy function
\begin{align*}
  \sfF(u):=\sfF(u;\rho) := \frac14 \int_{\mathcal O}\big(  \rho u^2 +  u(x) \partial_x \log \rho(x)\big) dx,
\end{align*}
so
\begin{align*}
  \frac{\delta \sfF}{\delta u} = \frac14 ( 2 \rho u + \partial_x \log \rho),
\end{align*}
and
\begin{align}
  \label{gradfloH}
  \sfH(u, \frac{\delta f}{\delta u}; \rho) 
  &= - 4 \langle \frac{\delta f}{\delta u}, \frac{\delta \sfF}{\delta u} \rangle + 2 \int_{\mathcal O} |\frac{\delta f}{\delta u}|^2 dx \\
  & = 2 \Big( \int_{\mathcal O} \big( \frac{\delta f}{\delta u}-   \frac{\delta \sfF}{\delta u}\big)^2 dx
    -  \int_{\mathcal O} |\frac{\delta \sfF}{\delta u}|^2 dx\Big).   \nonumber
\end{align}
In particular, for any $\theta \in [0,2]$,
\begin{align*}
  \sfH (\theta \sfF)(u;\rho) = \sfH(u;\theta \frac{\delta \sfF}{\delta u}; \rho) = -2  \theta  (2-\theta) \int_{\mathcal O} |\frac{\delta \sfF}{\delta u}|^2 dx \leq 0 .
\end{align*}
This inequality will play important role in the rigorous justification of the derivation of $H$ in~\eqref{H}.

Putting everything together,~\eqref{barHE} gives 
\begin{align*}
  H(\rho, \varphi) &= 
  \inf_{f_1}\sup_{u} \Big( \sfH(u,\frac{\delta f_1}{\delta u};\rho)  + \sfV(u; \varphi, \rho)\Big) \\
                   &= \inf_{f_1}\sup_{u}  2 \int_{\mathcal O}    \Big( \big| \frac{\delta f_1}{\delta u} 
                       -\frac{\delta \sfF}{\delta u} \big|^2 
                       - \big| \frac{\delta \sfF}{\delta u} \big|^2  \Big) dx  + \sfV(u;\rho, \varphi)  \\
                   &=   \sup_{u}  \int_{\mathcal O} \Big(  \rho u \partial_x \varphi  - 2\big| \frac{\delta \sfF}{\delta u}\big|^2\Big)  dx  \\
                   &=  \sup_{u}  \int_{\mathcal O} \Big(  \rho u \partial_x \varphi  - \frac18 |2  \rho u 
                       + \partial_x \log \rho |^2  \Big)  dx \\
                   &=  \sup_{\eta}  \int_{\mathcal O} \Big(  (\eta -\frac12\partial_x \log \rho)
                        \partial_x \varphi  - \frac12 | \eta |^2  \Big)  dx \\
                   &= -\frac12 \langle \partial_x \log \rho, \partial_x \varphi \rangle 
                       + \frac12 \int_{\mathcal O}  |\partial_x \varphi|^2  dx.
\end{align*}
This is the Hamiltonian we gave in~\eqref{H}.

 \subsection{A decomposition of the $\sfH$ into a family of microscopic ones
   $\big\{ {\mathfrak h}(\cdot; \alpha,\beta) : \alpha, \beta \in \R \big\}$}
\label{sec:decomp-sfH-into}

The second and third approaches to identify $H$ involve a subtle argument we are going to explain first.
 
For the kind of problem we consider, we intuitively expect that \emph{propagation of chaos} to hold. We expect this even at the
large deviation/ hydrodynamic limit scale.  Therefore, the infinite-dimensional Hamiltonian $\sfH$ is expected to be representable
as summation of a family of one-particle level Hamiltonians indexed by some hydrodynamic parameters in statistical local
equilibrium. This intuition leads to the following arguments.

We define a family of Hamiltonians indexed by $(\alpha,\beta)$ at the one-particle level,
\begin{align*}
  \mathfrak{h}(\upsilon, p; \alpha, \beta)&:=  -(2 \alpha \upsilon + \beta) p   
                                            + 2   p^2 , \quad (\upsilon,p) \in \R \times \R, \quad \forall \alpha, \beta \in \R, \\
  \mathfrak{h}^P(\upsilon, p;\alpha,\beta)&:=\mathfrak{h}(\upsilon, p; \alpha, \beta) + \alpha P \upsilon , \quad \forall P \in \R.
\end{align*}
We observe that
\begin{align}
  \label{Heta}
  \sfH(u, \phi; \rho) = \int_{\mathcal O}  \mathfrak{h} \big( u(x), \phi(x); \rho(x), \partial_x \log \rho(x) \big) dx,
\end{align}
and that
\begin{align*}
  \sfH_\sfV(u,\phi;\rho, \varphi)
  =  \int_{\mathcal O}  \mathfrak{h}^{\partial_x \varphi(x)} \big( u(x), \xi(x); \rho(x), \partial_x \log \rho(x) \big) dx.
\end{align*}
At least formally, if we take 
\begin{align}
  \label{f1j}
  f_1(u):= \int_{y \in \mathcal O} \psi\big(u(y);y\big) dy
\end{align}
and denote $\partial_1 \psi(\upsilon;y):= \partial_\upsilon \psi(\upsilon;y)$, then
\begin{align*}
  \frac{\delta f_1}{\delta u}(x) = \partial_1 \psi\big(u(x);x\big),
\end{align*}
and 
\begin{align*}
  \sfH \Big(u, \frac{\delta f_1 }{\delta u}; \rho\Big) 
  & =  \int_{\mathcal O}  {\mathfrak h} \big(u(x), \partial_1 \psi\big(u(x);x\big); \rho(x), \partial_x \log \rho(x)  \big)  dx, \\
  \sfH_\sfV \Big(u, \frac{\delta f_1 }{\delta u}; \rho, \varphi\Big) 
  & =  \int_{\mathcal O}   {\mathfrak h}^{\partial_x \varphi(x)} \big(u(x), \partial_1 \psi\big(u(x);x\big); 
    \rho(x), \partial_x \log \rho(x) \big) dx.
\end{align*}
Therefore, in order to solve~\eqref{cell}, it suffices to solve a family (indexed by $\alpha$ and $\beta$) of finite-dimensional
``small cell'' problems
\begin{align}
  \label{microcell}
  {\mathfrak h}\big(\upsilon, \partial_\upsilon \psi; \alpha, \beta\big) + \alpha P \upsilon = E[P;\alpha,\beta], 
  \quad \forall \upsilon \in \R.
\end{align}
Here, $E$ is a constant of the variable $\upsilon$. Moreover, if we can solve this finite-dimensional PDE problem, then the $H$
term for the infinite-dimensional problem~\eqref{cell} has a solution
\begin{align*}
  H(\rho, \varphi) = \int_{\mathcal O} E[\partial_x \varphi(x); \rho(x), \partial_x \log \rho(x)] dx.
\end{align*}

These consideration leads to two more ways of identifying the effective Hamiltonian $H=H(\rho,\varphi)$. (We remark that we could
present at least one further approach, which exploits the special one-dimensional nature of~\eqref{microcell} by invoking the
Maupertuis' principle. We choose not to present this approach, since we are interested in general methodologies that work even
when the velocity field $u(x)$ take values in several dimensions and $\upsilon$ in~\eqref{microcell} lives in several dimensions.)

 \subsection{Second approach to identify $H$ --- finite-dimensional weak KAM and the method of equilibrium points}
\label{sec:Second-approach-to}

We introduce a microscopic (one-particle level) free energy function
\begin{align*}
  \mathfrak{f}(\upsilon):=  \mathfrak{f}(\upsilon;\alpha,\beta):= \frac{1}{4}   \big( \alpha \upsilon^2 + \beta \upsilon \big).
\end{align*}
The connection with the free energy introduced earlier is that 
\begin{align*}
  \sfF(u;\rho) = \int_{\mathcal O} \mathfrak{f}\big(u(x); \rho(x), \partial_x \log \rho(x)\big) dx.
\end{align*} 
It is not surprising that the microscopic Hamiltonians ${\mathfrak h}$ also have controlled gradient flow structures:
\begin{align}\label{hgradF}
\mathfrak{h}(\upsilon, p; \alpha, \beta)& =  4  \Big( \frac12 p^2 -  p \partial_\upsilon \mathfrak{f}  \Big)
=  2   \big( |p - \partial_\upsilon \mathfrak{f}|^2 -  |\partial_\upsilon \mathfrak{f}|^2 \big) \\
&= {\mathfrak h}_{\rm iso} \big(\upsilon, p -\partial_\upsilon \mathfrak{f} \big), \nonumber
\end{align}
if we introduce a family of isotropic Hamiltonians 
\begin{align*}
{\mathfrak h}_{\rm iso}(\upsilon, p) :=  2   \big( |p|^2 - |\partial_\upsilon \mathfrak{f}|^2 \big).
\end{align*}

Solving~\eqref{microcell} is equivalent to solving
\begin{align*}
  {\mathfrak h}_{\rm iso}\big(\upsilon, \partial_\upsilon \Psi \big) + \alpha \upsilon P = E,
\end{align*}
with $\Psi = (\psi -{\mathfrak f})$.  We note $\mathfrak h_{\rm iso}$ is isotropic in the sense that the dependence on generalized
momentum variable $p$ is only through its length $|p|$, i.e.
${\mathfrak h}_{\rm iso}(\upsilon, p) = {\mathfrak h}_{\rm iso} \big(\upsilon, |p|\big)$.  It also holds that
$\R_+ \ni r \mapsto \eta_{\rm iso}(\upsilon, r)$ is convex, monotonically nondecreasing and super-linear. In particular,
\begin{align*}
  \inf_{r \in \R_+} \eta_{\rm iso}(\upsilon, r)  = \eta_{\rm iso}(\upsilon, 0).
\end{align*}
For this kind of Hamiltonian, it is known that (e.g. Fathi~\cite{FaBook})
\begin{align*}
  E&=E[P;\alpha,\beta]=  \sup_{\upsilon \in \R} \big( {\mathfrak h}_{\rm iso}(\upsilon, 0) + \alpha \upsilon P\big) \\
   &= \sup_{\upsilon \in \R} \big(\alpha \upsilon P -  2 |\partial_\upsilon {\mathfrak f}(\upsilon)|^2 \big) 
     = \sup_{\upsilon \in \R}\Big( \alpha \upsilon P -\frac12  \big| \alpha \upsilon + \frac12 \beta \big|^2\Big) \\
   &= - \frac12 \beta P + \frac{1}{2}P^2.
\end{align*}
Consequently,
\begin{align*}
  H(\rho,\varphi) = 
  \int_{\mathcal O} E[\partial_x \varphi(x); \rho(x), \partial_x \log \rho(x)]  dx
  = -\frac12 \langle \partial_x \log \rho, \partial_x \varphi \rangle + \frac12 \int_{\R}  |\partial_x \varphi|^2  dx.
\end{align*}

\subsection{ Third approach to identify $H$ ---  semiclassical approximations}
\label{sec:Third-approach-to}

Finally, we abandon methods based on weak KAM. Instead, we introduce a method for identifying $E[P;\alpha,\beta]$ directly using
probability theory and ideas from semi-classical limits.

Our point of departure is to approximate equation~\eqref{microcell} by introducing an extra viscosity parameter $\kappa >0$.  For
readers familiar with the Hamiltonian convergence approach to large deviation as described in Feng and Kurtz~\cite{FK06},
$\mathfrak h$ is (see~\eqref{eq:h-k-lim} below) the limiting Hamiltonian for small noise large deviations ($\kappa \to 0^+$) for
the stochastic differential equations
\begin{align}
  \label{SDE}
  d \upsilon(t) + \big(2 \alpha \upsilon(t) + \beta \big) dt= 2 \sqrt{\kappa}  d W(t).
\end{align}
The solution $\upsilon(t)$ is an $\R$-valued Markov process with infinitesimal generator
\begin{align*}
  L_\kappa \psi(\upsilon) := - (2 \alpha \upsilon +\beta) \partial_\upsilon \psi (\upsilon) 
   + 2 \kappa \partial^2_{\upsilon \upsilon} \psi(\upsilon).
\end{align*}
Following \cite{FK06}, we define a sequence of nonlinear second order differential operators
\begin{align*}
  \big( {\mathfrak h}_\kappa  \psi\big) (\upsilon):=
  e^{-\kappa^{-1} \psi} \Big( \kappa L_\kappa \Big) e^{\kappa^{-1} \psi}(\upsilon),
\end{align*}
then
\begin{align}
  \label{eq:h-k-lim}
   \lim_{\kappa \to 0^+}  \big({\mathfrak h}_\kappa \psi\big)(\upsilon)
   = {\mathfrak h} \big(\upsilon, \partial_\upsilon \psi \big).
\end{align}
We also consider a second-order stationary Hamilton-Jacobi equation with constant $E_\kappa$:
\begin{align} \label{mcell2}
  \big( {\mathfrak h}_\kappa \psi\big) (\upsilon)+ \alpha \upsilon P  = E_\kappa.
\end{align}
This can be viewed as a regularized approximation to the first-order equation~\eqref{microcell}.

A simple transformation turns the nonlinear PDE~\eqref{mcell2} into a linear eigenfunction, eigenvalue equation
\begin{align} \label{schrodinger}
  ( \kappa L_\kappa + \alpha \upsilon P )  \Psi_\kappa = E_\kappa \Psi_\kappa,
\end{align}
where
\begin{align*}
\Psi_\kappa := e^{ \kappa^{-1} \psi} >0.
\end{align*}
This is the equation defining ground state $\Psi_\kappa$ with ground state energy $E_\kappa$ of the rescaled Schr\"odinger operator $\kappa L_\kappa$. There is a theory giving uniqueness for the constant $E$ in~\eqref{microcell}. By well-known stability results for viscosity solution of Hamilton-Jacobi equations~\eqref{mcell2}, we can prove that
\begin{align*}
  E= \lim_{\kappa \to 0} E_\kappa.
\end{align*}

The ground state energy $E_\kappa$ is given by the Rayleigh-Ritz formula, which has been extensively studied in probability theory
in the context of large deviations for occupation measures by Donsker and Varadhan.  We denote by
\begin{align*}
  m_\kappa(d \upsilon) := Z_\kappa^{-1} e^{-\frac{\alpha \upsilon^2 +  \beta \upsilon}{2 \kappa}} d \upsilon
\end{align*}
the invariant probability measure for the Markov process $\upsilon(t)$, and introduce a family of related probability measures indexed by $\Phi \in C_b(\R)$:
\begin{align*}
  m_\kappa^\Phi(d\upsilon) := \frac{e^{2\Phi(\upsilon)}}{ \int e^{2 \Phi} d m_\kappa} m_\kappa(d \upsilon).
\end{align*}
We identify the (pre-)Dirichlet form associated with $\kappa L_\kappa$ by 
\begin{align*}
  \mathcal E_\kappa(\phi_1,\phi_2):= - \int_{\upsilon \in \R} \phi_1(\upsilon) (\kappa L_\kappa \phi_2)(\upsilon) m_\kappa(d\upsilon) 
  = 2  \kappa^2 \int_{\R} (\partial_{\upsilon} \phi_1)( \partial_{\upsilon} \phi_2 )m_\kappa(d \upsilon).
\end{align*}
Then, by the arguments on pages~112 and~113 of Stroock~\cite{Str84} (alternatively, one can also follow Example~B.14 in Feng and
Kurtz~\cite{FK06}), we have
\begin{align*}
  E_\kappa[P] = \sup_{\Phi} \Big\{ \alpha P \int_{\R}  \upsilon  m_\kappa^\Phi(d\upsilon)
  - \mathcal E_\kappa\Big(\sqrt{\frac{d m_\kappa^\Phi}{d m_\kappa}}\Big)\Big\} 
  = \sup_\Phi  \int_{\R}\big( \alpha \upsilon P - 2  \kappa^2 |\partial_\upsilon \Phi |^2\big) m_\kappa^\Phi(d\upsilon).
\end{align*}
A change of variable $\hat{\Phi} \mapsto \kappa \Phi$ gives
\begin{align*}
  E_\kappa[P] =  \sup_{\Phi} \int_{\R} \big( \alpha \upsilon P 
   - 2  |\partial_\upsilon \hat{\Phi}(\upsilon)|^2 \big)
  m_{\kappa,\hat{\Phi}}(d\upsilon) ,
\end{align*}     
where 
\begin{align*}  
m_{\kappa,\hat{\Phi}}(d\upsilon) = \frac{e^{\frac{1}{\kappa} (2 \hat{\Phi}
      - \frac{\alpha \upsilon^2 +\beta \upsilon}{2} )}d\upsilon}{ Z_{\kappa,\hat{\Phi}}}. 
\end{align*}
We can further lift the $m_{\kappa, \Phi}$ probability measure to 
\begin{align*}
  \mathfrak{m}_{\kappa,\Phi}(d\upsilon, d \xi)
  := \delta_{\partial_\upsilon \Phi }(d \xi) m_{\kappa,\Phi}(d\upsilon),
  \quad (\upsilon, \xi) \in \R \times \R,
\end{align*}
giving
\begin{align*}
  E_\kappa[P] = \sup_{\Phi} \int_{\R} \big( \alpha \upsilon P - 2 |\xi|^2 \big) \mathfrak{m}_{\kappa,\Phi}(d\upsilon, d \xi).
\end{align*}
We see that as $\kappa \to 0$, by the Laplace principle, the limit points of $\{ \mathfrak{m}_{\kappa, \Phi}: \kappa >0 \}$ form a
family of probability measures as follows:
\begin{align*}
  \mathfrak{m}_\Phi(d\upsilon, d \xi):=\sum_k p_k \delta_{\{\upsilon_k, \partial_\upsilon \Phi(\upsilon_k)\}}(d\upsilon, d \xi), \quad \sum_k p_k=1, p_k>0,
\end{align*}
with $\upsilon_k$ solves the algebraic equation
\begin{align*}
  4  \partial_\upsilon \Phi(\upsilon) - (2 \alpha \upsilon + \beta ) =0.
\end{align*}
That is,  
\begin{align*}
  \alpha \upsilon =2   \xi -\frac{\beta}{2 }, \quad \forall (\upsilon, \xi) \in \text{supp}[\mathfrak{m}_\Phi].
\end{align*}
Then it follows that
\begin{align*}
  E[P;\alpha,\beta] &= \lim_{\kappa \to 0} E_\kappa[P;\alpha,\beta] 
  = \lim_{\kappa \to 0} \sup_\Phi \int_{\R} \big( \alpha \upsilon P - 2 |\xi|^2 \big) \mathfrak{m}_{\kappa,\Phi}(d\upsilon, d \xi)  \\
  &=  \sup_\Phi \int_{\R} \big( \alpha \upsilon P - 2 |\xi|^2 \big) \mathfrak{m}_{\Phi}(d\upsilon, d \xi)  \\
  & = \sup_{\mathfrak{m}}  \int \Big( \big( 2  \xi - \frac{\beta}{2}\big)P - 2  |\xi|^2 \Big) {\mathfrak m}(d\upsilon,d\xi) \\
  & = -\frac{\beta}{2} P +\frac{P^2}{2}.
\end{align*}
Hence again we are lead to the 
\begin{align*}
  H(\rho,\varphi) = 
  \int_{\mathcal O} E[\partial_x \varphi(x); \rho(x), \partial_x \log \rho(x)]  dx
  = \frac12 \langle \partial^2_{xx} \log \rho, \varphi \rangle + \frac12 \int_{\R}  |\partial_x \varphi|^2  dx
\end{align*}
and again recover the Hamiltonian~\eqref{H}.

\appendix
\section{Some properties of $\mathcal P(\mathcal O)$}
\label{sec:Append-Some-prop} 

\subsection{Quotients and Coverings}
\label{sec:Quotients-Coverings}
 
\subsubsection{Projections and lifts}
\label{sec:Projections-lifts}

As defined in Section~\ref{Intro2}, we view $\mathcal O:= \R/\Z$ as a quotient space with corresponding quotient metric $r$.  We
define a projection $\sfp \colon \R \mapsto \mathcal O$ by
\begin{align}
  \label{eq:app-proj}
  x := \sfp(\hat{x}) = \hat{x} \pmod{1}, \quad \forall \hat{x} \in \R.
\end{align}

Let $\mathcal P_2(\R)$ be the Wasserstein order-2 metric space~\cite{AGS08}. For every $\hat{\mu} \in \mathcal P_2(\R)$, we
define push forward of $\hat{\mu}$ by $\mu:= \sfp_\# \hat{\mu} \in \mathcal P(\mathcal O)$. That is, we project $\hat{\mu}$ to
$\mu$ in the following way
\begin{align}
  \label{summu}
  \mu(A) := \hat{\mu} \big(\sfp^{-1}(A)\big):=  \sum_{k \in \Z} \hat{\mu}(A + k), \quad \forall A \in {\mathcal B}({\mathcal O}).
\end{align}
Here, we use the set $A+k:= \{ x +k : x \in A\}$, and we write ${\mathcal B}(\mathcal O)$ to denote the collection of Borel sets
in $\mathcal O$.

There are many ways to lift a probability measure $\mu \in \mathcal P(\mathcal O)$ to a probability measure $\hat{\mu} \in
\mathcal P_2(\R)$ such that $\sfp_\# \hat{\mu} = \mu$. Following Galaz-Garc\'{\i}a, Kell, Mondino and Sosa~\cite{GKMS}, we now
describe one class of such lifts.  Given a family of weights
\begin{align}
  \label{defalpha}
  \alpha:= \{ \alpha_k \in [0,1] :  \sum_{k \in \Z} \alpha_k =1 , \sum_{k \in \Z} k^2 \alpha_k <\infty \}_{k \in \Z},
\end{align}
we introduce a probability measure on $\R$ by
\begin{align*}
\nu_x(d\hat{x}) :=\nu_x^\alpha(d \hat{x}):= \sum_{k \in \Z} \alpha_k \delta_{x+k}(d\hat{x}), \quad \forall x \in {\mathcal O}.
\end{align*}
Second, using the family of measures $\{\nu_x \}_{x \in \mathcal O}$, we define a lift operator 
$\Lambda: = \Lambda^\alpha: \mathcal P(\mathcal O) \mapsto \mathcal P_2(\R)$ as follows
\begin{align}
  \label{liftm}
  \mu \mapsto \hat{\mu} := \Lambda (\mu) := \int_{x \in \mathcal O} \nu_x  \mu(dx).
\end{align}
We note that $\sfp_\# \hat{\mu} = \mu$.

Let $C_{\per}(\R)$ be the collection of continuous functions which are $1$-periodic on $\R$. Similarly, we define $C_{\per}^p(\R)$
for $p =1,2,\ldots, \infty$. For each $\hat{\varphi} \in C_{\per}^p(\R)$, we have translation invariance
\begin{align*}
  \hat{\varphi}(\hat{x}) = \hat{\varphi}(\hat{x}+k), \quad \forall k \in Z.
\end{align*}
Hence such a function $\hat{\varphi}$ projects to an element $\varphi \in C^p(\mathcal O)$ as follows
\begin{align}
  \label{hat2phi}
  \varphi (x):= \hat{\varphi} (\hat{x}), \quad \forall \hat{x} \in \sfp^{-1}(x).
\end{align}
On the other hand, each $\varphi \in C^p(\mathcal O)$ has a lift to $\hat{\varphi} \in C^p_{\per}(\R)$ defined by
\begin{align}
  \label{phi2hat}
  \hat{\varphi}(\hat{x}):= \varphi \big(\sfp(\hat{x})\big).
\end{align}
Such a lift is translation invariant in $\Z$ and its projection (as defined in~\eqref{eq:app-proj}) gives $\varphi$.

Given $\rho, \gamma \in \mathcal P(\mathcal O)$ and $\varphi \in C^p(\mathcal O)$, for a fixed $\alpha$, let $\hat{\rho},
\hat{\gamma} \in \mathcal P_2(\R)$ and $\hat{\varphi} \in C^p_{\per}(\R)$ be lifts as just defined. Then
\begin{align*}
  \langle \rho - \gamma, \varphi \rangle = \langle \hat{\rho} - \hat{\gamma}, \hat{\varphi} \rangle.
\end{align*}
In particular, this implies that 
\begin{align}
  \Vert \rho - \gamma \Vert_{-1} = \sup \Big( \langle \hat{\rho} - \hat{\gamma}, \hat{\varphi} \rangle
  : \varphi \in C^\infty(\mathcal O), \int_{\mathcal O} |\partial_x \varphi|^2 dx \leq 1 \Big) . 
\end{align}

\subsubsection{A random variable description}
\label{sec:rand-vari-descr}

The constructions of projection and lifts of the previous subsection~\ref{sec:Projections-lifts} can be described using the
language of random variables. In certain situations, this can be more intuitive.

Let $(\Omega, \mathcal F, \mathbb{P})$ be a probability space and let $(X,K)\colon \Omega \mapsto \mathcal O \times \Z$ be a pair
of random variables. We define the $\R$-valued random variable
\begin{align*}
  \hat{X}:=X+K.
\end{align*}
Then
\begin{align*}
  X= \hat{X}    \pmod{1}, \quad K:=\lfloor{X}\rfloor.
\end{align*}

If $\hat{X}$ has the probability law $\hat{\rho}$, then
\begin{align*}
  \rho(dx):=\mathbb{P}(X \in dx) = \sum_{k \in \Z} \mathbb{P}(\hat{X} \in dx+k; K=k) = \sum_{k\in \Z} \hat{\rho}(dx+k).
\end{align*}
On the other hand, if $X$ has the probability law $\rho$, depending on the conditional probability law of the $K$,
\begin{align*}
  \alpha_k:=\alpha_k(x):={\mathbb P}(K =k | X=x),
\end{align*}
or equivalently for the conditional law of $\hat{X}$
\begin{align*}
  \nu_x(d\hat{x}) := {\mathbb P}(\hat{X} \in d \hat{x}|X=x) = {\mathbb P}(X+K \in d\hat{x} | X=x)
  =\sum_{k\in\Z} \alpha_k\delta_{x+k}(d\hat{x}),
\end{align*}
 the lift defined in~\eqref{liftm} becomes
\begin{align*}
  \hat{\rho}(d\hat{x}) &= {\mathbb P}(\hat{X} \in d\hat{x}) 
  = \int_{x \in \mathcal O} {\mathbb P}( X+K \in d \hat{x} |X=x) {\mathbb P}(dx) 
  =\int_{x \in \mathcal O} \nu_x(d\hat{x}) \rho(dx).
\end{align*}
 
\subsection{Equivalence of metric topologies}
\label{sec:Equiv-metr-topol}

We recall inequality~\eqref{W1Hn1}
\begin{align*}
  W_1(\rho,\gamma) \leq \Vert \rho -\gamma\Vert_{-1}.
\end{align*}
Next, we establish a converse of sorts. The proof below is an adaptation of Lemma 4.1 in Mischler-Mouhot~\cite{MM13}.
\begin{lemma}
  \label{equiDis}
  For every $\rho, \gamma \in \mathcal P(\mathcal O)$, we have
  \begin{align*}
    \Vert \rho - \gamma\Vert_{-1} \leq \frac{2}{\sqrt{\pi}} \sqrt{W_1(\rho,\gamma)}.
  \end{align*}
\end{lemma}

\begin{proof}
We can construct a probability space $(\Omega, \mathcal F, {\mathbb P})$ with two pairs of $\mathcal O \times \Z$-valued 
random variables $(X,K_1), (Y,K_2)$ such that
\begin{align*}
  \rho(dx)= {\mathbb P}(X \in dx), \quad \gamma(dy)={\mathbb P}(Y \in dy).
\end{align*}
We introduce
\begin{align*}
\hat{X}:=X+K_1 \text{ and } \hat{Y}:=Y+K_2
\end{align*}
and denote $\hat{\rho}(d\hat{x}) := {\mathbb P}(\hat{X} \in d \hat{x})$ and $\hat{\gamma}(d\hat{y}):={\mathbb P} (\hat{Y} \in d
\hat{y})$.  We use the Fourier transform
\begin{align*}
  {\mathcal F}[\hat{\rho}](\xi)& :=\frac{1}{\sqrt{2\pi}} \int_{\R} e^{-i z \cdot \xi} \hat{\rho}(dz )
  = \frac{1}{\sqrt{2\pi}}  {\mathbb E}[ e^{-i \hat{X} \cdot \xi} ], \intertext{and}
  {\mathcal F}[\hat{\gamma}](\xi) &:= \frac{1}{\sqrt{2\pi}}  \int_{\R} e^{-i z \cdot \xi} \hat{\gamma}(dz )
  = \frac{1}{\sqrt{2\pi}}  {\mathbb E}[ e^{-i \hat{Y} \cdot \xi}].
\end{align*}
Then
\begin{align*}
|{\mathcal F}[\hat{\rho}](\xi) -{\mathcal F}[\hat{\gamma}](\xi)| 
 \leq  \frac{1}{\sqrt{2\pi}}  \big| {\mathbb E} [e^{-i \hat{X} \cdot \xi} - e^{-i \hat{Y} \cdot \xi})  ] \big| 
\leq \frac{|\xi|}{\sqrt{2\pi}}    {\mathbb E} [|X-Y-K|], 
\end{align*}
 where $K:=K_1-K_2$.
 
On the other hand,  it holds that
\begin{align*}
  \Vert \rho - \gamma \Vert_{-1}^2 &= \sup \Big( \langle \hat{\rho} - \hat{\gamma}, \hat{\varphi}
  \rangle : \varphi \in C^\infty(\mathcal O),
  \int_{\mathcal O} |\partial_x \varphi|^2 dx \leq 1, \\
  & \qquad \qquad \qquad \hat{\varphi} \text{ is defined from $\varphi$ as in }~\eqref{phi2hat},
  \hat{\rho}, \hat{\gamma} \text{ are lifts of  $\rho,\gamma$ as in }~\eqref{liftm}  \Big)^2 \\
  & \leq  \sup \Big( \langle \hat{\rho} - \hat{\gamma}, \hat{\varphi} \rangle : \hat{\varphi} \in C_c^\infty(\R),
  \int_{\R} |\partial_x \hat{\varphi}|^2 dx \leq 1 \Big)^2 \\
  &=   \int_{\R} \frac{|{\mathcal F}[\hat{\rho}] (\xi)- {\mathcal F}[\hat{\gamma}](\xi)|^2}{|\xi|^2} d \xi   \\
  & \leq  \inf_{R>0} \Big( \sup_{\xi \in \R}\frac{|{\mathcal F}[\hat{\rho}] (\xi)- {\mathcal F}[\hat{\gamma}](\xi)|^2}{|\xi|^2}
  \int_{|\xi|\leq R} d \xi +\frac{ 4}{2 \pi} \int_{|\xi|>R} \frac{1}{|\xi|^2} d \xi \Big).
\end{align*}
Therefore
\begin{align*}
  \Vert \rho - \gamma \Vert_{-1}^2
  & \leq \frac{1}{2 \pi} \inf_{R>0} \Big(  (2R) {\mathbb E}^2[| X-Y-K|]  +  \frac{8}{R} \Big) = \frac{4}{\pi} {\mathbb E}[|X-Y-K|].
\end{align*}

Next, we note that
\begin{align*}
  & \big\{ \varphi(X,Y) \quad \big| \quad \varphi : \mathcal O \times \mathcal O \mapsto \Z \text{ is measurable}  \big \}  \\
  & \qquad  \subset \big\{ K:=K_1-K_2 \quad \big| \quad K_1, K_2 \text{ are $\Z$-valued random variables} \big\}.
\end{align*}
Therefore, with the quotient metric $r$ defined in~\eqref{defr}, we have
\begin{align*}
  & \inf_K  {\mathbb E}[|X-Y-K|]  \\
  &\leq \inf \Big\{  {\mathbb E}[ |X-Y-\varphi(X,Y)| ]  : \text {where } \varphi \text{ is measurable function from }
  \mathcal O \times \mathcal O \text{ to } \Z \Big \} \\
  &= {\mathbb E}[ \inf_{k \in \Z}|X-Y-k|] =   {\mathbb E}[ r(X,Y)] = \int_{\mathcal O \times \mathcal O} r(x,y) {\bm \nu}(dx, dy),
  \quad \forall {\bm \nu} \in \Pi(\rho, \gamma)
\end{align*}
(see~\eqref{defPi} for the definition of $\Pi(\rho,\gamma)$).  This leads to
\begin{align*}
  \Vert \rho - \gamma \Vert_{-1}^2 \leq  \frac{4}{\pi} W_1(\rho,\gamma).
\end{align*}
\end{proof}


\begin{bibdiv}
   \begin{biblist} 

\bib{AF14}{article}{
   author={Ambrosio, Luigi},
   author={Feng, Jin},
   title={On a class of first order Hamilton-Jacobi equations in metric
   spaces},
   journal={J. Differential Equations},
   volume={256},
   date={2014},
   number={7},
   pages={2194--2245},
   issn={0022-0396},
   review={\MR{3160441}},
   doi={10.1016/j.jde.2013.12.018},
}

\bib{AGS08}{book}{
   author={Ambrosio, Luigi},
   author={Gigli, Nicola},
   author={Savar\'e, Giuseppe},
   title={Gradient flows in metric spaces and in the space of probability
   measures},
   series={Lectures in Mathematics ETH Z\"urich},
   edition={2},
   publisher={Birkh\"auser Verlag, Basel},
   date={2008},
   pages={x+334},
   isbn={978-3-7643-8721-1},
   review={\MR{2401600}},
}

\bib{CDPP89}{article}{
   author={Caprino, S.},
   author={De Masi, A.},
   author={Presutti, E.},
   author={Pulvirenti, M.},
   title={A stochastic particle system modeling the Carleman equation},
   journal={J. Statist. Phys.},
   volume={55},
   date={1989},
   number={3-4},
   pages={625--638},
   issn={0022-4715},
   review={\MR{1003531}},
}

\bib{CL85}{article}{
   author={Crandall, Michael G.},
   author={Lions, Pierre-Louis},
   title={Hamilton-Jacobi equations in infinite dimensions. I. Uniqueness of
   viscosity solutions},
   journal={J. Funct. Anal.},
   volume={62},
   date={1985},
   number={3},
   pages={379--396},
   issn={0022-1236},
   review={\MR{794776}},
   doi={10.1016/0022-1236(85)90011-4},
}

\bib{CL86}{article}{
   author={Crandall, Michael G.},
   author={Lions, Pierre-Louis},
   title={Hamilton-Jacobi equations in infinite dimensions. II. Existence of
   viscosity solutions},
   journal={J. Funct. Anal.},
   volume={65},
   date={1986},
   number={3},
   pages={368--405},
   issn={0022-1236},
   review={\MR{826434}},
   doi={10.1016/0022-1236(86)90026-1},
}

\bib{CL86b}{article}{
   author={Crandall, Michael G.},
   author={Lions, Pierre-Louis},
   title={Hamilton-Jacobi equations in infinite dimensions. III},
   journal={J. Funct. Anal.},
   volume={68},
   date={1986},
   number={2},
   pages={214--247},
   issn={0022-1236},
   review={\MR{852660}},
   doi={10.1016/0022-1236(86)90005-4},
}

\bib{CL90}{article}{
   author={Crandall, Michael G.},
   author={Lions, Pierre-Louis},
   title={Viscosity solutions of Hamilton-Jacobi equations in infinite
   dimensions. IV. Hamiltonians with unbounded linear terms},
   journal={J. Funct. Anal.},
   volume={90},
   date={1990},
   number={2},
   pages={237--283},
   issn={0022-1236},
   review={\MR{1052335}},
   doi={10.1016/0022-1236(90)90084-X},
}

\bib{CL91}{article}{
   author={Crandall, Michael G.},
   author={Lions, Pierre-Louis},
   title={Viscosity solutions of Hamilton-Jacobi equations in infinite
   dimensions. V. Unbounded linear terms and $B$-continuous solutions},
   journal={J. Funct. Anal.},
   volume={97},
   date={1991},
   number={2},
   pages={417--465},
   issn={0022-1236},
   review={\MR{1111190}},
   doi={10.1016/0022-1236(91)90010-3},
}

\bib{CL94}{article}{
   author={Crandall, Michael G.},
   author={Lions, Pierre-Louis},
   title={Hamilton-Jacobi equations in infinite dimensions. VI. Nonlinear
   $A$ and Tataru's method refined},
   conference={
      title={Evolution equations, control theory, and biomathematics},
      address={Han sur Lesse},
      date={1991},
   },
   book={
      series={Lecture Notes in Pure and Appl. Math.},
      volume={155},
      publisher={Dekker, New York},
   },
   date={1994},
   pages={51--89},
   review={\MR{1254890}},
}

\bib{CL94b}{article}{
   author={Crandall, Michael G.},
   author={Lions, Pierre-Louis},
   title={Viscosity solutions of Hamilton-Jacobi equations in infinite
   dimensions. VII. The HJB equation is not always satisfied},
   journal={J. Funct. Anal.},
   volume={125},
   date={1994},
   number={1},
   pages={111--148},
   issn={0022-1236},
   review={\MR{1297016}},
   doi={10.1006/jfan.1994.1119},
}


\bib{Dolgopyat2011a}{article}{
   author={Dolgopyat, Dmitry},
   author={Liverani, Carlangelo},
   title={Energy transfer in a fast-slow Hamiltonian system},
   journal={Comm. Math. Phys.},
   volume={308},
   date={2011},
   number={1},
   pages={201--225},
   issn={0010-3616},
   review={\MR{2842975}},
   doi={10.1007/s00220-011-1317-7},
}


\bib{E99}{article}{
   author={E, Weinan},
   title={Aubry-Mather theory and periodic solutions of the forced Burgers
   equation},
   journal={Comm. Pure Appl. Math.},
   volume={52},
   date={1999},
   number={7},
   pages={811--828},
   issn={0010-3640},
   review={\MR{1682812}},
}

\bib{Evans01}{article}{
   author={Evans, Lawrence C.},
   title={Effective Hamiltonians and quantum states},
   conference={
      title={S\'eminaire: \'Equations aux D\'eriv\'ees Partielles, 2000--2001},
   },
   book={
      series={S\'emin. \'Equ. D\'eriv. Partielles},
      publisher={\'Ecole Polytech., Palaiseau},
   },
   date={2001},
   pages={Exp. No. XXII, 13},
   review={\MR{1860693}},
}

\bib{Evans04}{article}{
   author={Evans, Lawrence C.},
   title={Towards a quantum analog of weak KAM theory},
   journal={Comm. Math. Phys.},
   volume={244},
   date={2004},
   number={2},
   pages={311--334},
   issn={0010-3616},
   review={\MR{2031033}},
}
	
\bib{Evans04b}{article}{
   author={Evans, Lawrence C.},
   title={A survey of partial differential equations methods in weak KAM
   theory},
   journal={Comm. Pure Appl. Math.},
   volume={57},
   date={2004},
   number={4},
   pages={445--480},
   issn={0010-3640},
   review={\MR{2026176}},
}

\bib{EG01}{article}{
   author={Evans, L. C.},
   author={Gomes, D.},
   title={Effective Hamiltonians and averaging for Hamiltonian dynamics. I},
   journal={Arch. Ration. Mech. Anal.},
   volume={157},
   date={2001},
   number={1},
   pages={1--33},
   issn={0003-9527},
   review={\MR{1822413}},
}

\bib{EG02}{article}{
   author={Evans, L. C.},
   author={Gomes, D.},
   title={Effective Hamiltonians and averaging for Hamiltonian dynamics. II},
   journal={Arch. Ration. Mech. Anal.},
   volume={161},
   date={2002},
   number={4},
   pages={271--305},
   issn={0003-9527},
   review={\MR{1891169}},
}

\bib{Fa97}{article}{
   author={Fathi, Albert},
   title={Th\'eor\`eme KAM faible et th\'eorie de Mather sur les syst\`emes
   lagrangiens},
   language={French, with English and French summaries},
   journal={C. R. Acad. Sci. Paris S\'er. I Math.},
   volume={324},
   date={1997},
   number={9},
   pages={1043--1046},
   issn={0764-4442},
   review={\MR{1451248}},
}

\bib{Fa97b}{article}{
   author={Fathi, Albert},
   title={Solutions KAM faibles conjugu\'ees et barri\`eres de Peierls},
   language={French, with English and French summaries},
   journal={C. R. Acad. Sci. Paris S\'er. I Math.},
   volume={325},
   date={1997},
   number={6},
   pages={649--652},
   issn={0764-4442},
   review={\MR{1473840}},
}

\bib{Fa98}{article}{
   author={Fathi, Albert},
   title={Orbites h\'et\'eroclines et ensemble de Peierls},
   language={French, with English and French summaries},
   journal={C. R. Acad. Sci. Paris S\'er. I Math.},
   volume={326},
   date={1998},
   number={10},
   pages={1213--1216},
   issn={0764-4442},
   review={\MR{1650195}},
}

\bib{FS04}{article}{
   author={Fathi, Albert},
   author={Siconolfi, Antonio},
   title={Existence of $C^1$ critical subsolutions of the Hamilton-Jacobi
   equation},
   journal={Invent. Math.},
   volume={155},
   date={2004},
   number={2},
   pages={363--388},
   issn={0020-9910},
   review={\MR{2031431}},
}

\bib{FS05}{article}{
   author={Fathi, Albert},
   author={Siconolfi, Antonio},
   title={PDE aspects of Aubry-Mather theory for quasiconvex Hamiltonians},
   journal={Calc. Var. Partial Differential Equations},
   volume={22},
   date={2005},
   number={2},
   pages={185--228},
   issn={0944-2669},
   review={\MR{2106767}},
}

\bib{FaBook}{book}{
   author={Fathi, Albert},
    title={Weak KAM theorem in Lagrangian dynamics },
   series={Cambridge Studies in Advanced Mathematics},
volume={88},
   publisher={Cambridge University Press },
   date={Not Published},
   pages={300},
   isbn={978-0521822282},
   isbn={0521822289},
}

 	
 
\bib{FK06}{book}{
   author={Feng, Jin},
   author={Kurtz, Thomas G.},
   title={Large deviations for stochastic processes},
   series={Mathematical Surveys and Monographs},
   volume={131},
   publisher={American Mathematical Society, Providence, RI},
   date={2006},
   pages={xii+410},
   isbn={978-0-8218-4145-7},
   isbn={0-8218-4145-9},
   review={\MR{2260560}},
}
	

\bib{FK09}{article}{
   author={Feng, Jin},
   author={Katsoulakis, Markos},
   title={A comparison principle for Hamilton-Jacobi equations related to
   controlled gradient flows in infinite dimensions},
   journal={Arch. Ration. Mech. Anal.},
   volume={192},
   date={2009},
   number={2},
   pages={275--310},
   issn={0003-9527},
   review={\MR{2486597}},
}

\bib{FS13}{article}{
   author={Feng, Jin},
   author={\'{S}wi\polhk ech, Andrzej},
   title={Optimal control for a mixed flow of Hamiltonian and gradient type
   in space of probability measures},
   note={With an appendix by Atanas Stefanov},
   journal={Trans. Amer. Math. Soc.},
   volume={365},
   date={2013},
   number={8},
   pages={3987--4039},
   issn={0002-9947},
   review={\MR{3055687}},
   doi={10.1090/S0002-9947-2013-05634-6},
}




\bib{GKMS}{article}{
   author={Galaz-Garc\'{\i}a, Fernando},
   author={Kell, Martin},
   author={Mondino, Andrea},
   author={Sosa, Gerardo},
   title={On quotients of spaces with Ricci curvature bounded below},
   journal={J. Funct. Anal.},
   volume={275},
   date={2018},
   number={6},
   pages={1368--1446},
   issn={0022-1236},
   review={\MR{3820328}},
   doi={10.1016/j.jfa.2018.06.002},
}

\bib{GNT08}{article}{
   author={Gangbo, Wilfrid},
   author={Nguyen, Truyen},
   author={Tudorascu, Adrian},
   title={Hamilton-Jacobi equations in the Wasserstein space},
   journal={Methods Appl. Anal.},
   volume={15},
   date={2008},
   number={2},
   pages={155--183},
   issn={1073-2772},
   review={\MR{2481677}},
   doi={10.4310/MAA.2008.v15.n2.a4},
}

\bib{GT10}{article}{
   author={Gangbo, W.},
   author={Tudorascu, A.},
   title={Lagrangian dynamics on an infinite-dimensional torus; a weak KAM
   theorem},
   journal={Adv. Math.},
   volume={224},
   date={2010},
   number={1},
   pages={260--292},
   issn={0001-8708},
   review={\MR{2600997}},
   doi={10.1016/j.aim.2009.11.005},
}

\bib{GS15}{article}{
   author={Gangbo, Wilfrid},
   author={\'{S}wi\polhk ech, Andrzej},
   title={Metric viscosity solutions of Hamilton-Jacobi equations depending
   on local slopes},
   journal={Calc. Var. Partial Differential Equations},
   volume={54},
   date={2015},
   number={1},
   pages={1183--1218},
   issn={0944-2669},
   review={\MR{3385197}},
   doi={10.1007/s00526-015-0822-5},
}

\bib{GHN15}{article}{
   author={Giga, Yoshikazu},
   author={Hamamuki, Nao},
   author={Nakayasu, Atsushi},
   title={Eikonal equations in metric spaces},
   journal={Trans. Amer. Math. Soc.},
   volume={367},
   date={2015},
   number={1},
   pages={49--66},
   issn={0002-9947},
   review={\MR{3271253}},
   doi={10.1090/S0002-9947-2014-05893-5},
}

\bib{Gorban2018a}{article}{
   author={Gorban, A. N.},
   title={Hilbert's sixth problem: the endless road to rigour},
   journal={Philos. Trans. Roy. Soc. A},
   volume={376},
   date={2018},
   number={2118},
   pages={20170238, 10},
   issn={1364-503X},
   review={\MR{3797486}},
   doi={10.1098/rsta.2017.0238},
}


\bib{GPV88}{article}{
   author={Guo, M. Z.},
   author={Papanicolaou, G. C.},
   author={Varadhan, S. R. S.},
   title={Nonlinear diffusion limit for a system with nearest neighbor
   interactions},
   journal={Comm. Math. Phys.},
   volume={118},
   date={1988},
   number={1},
   pages={31--59},
   issn={0010-3616},
   review={\MR{954674}},
}
	

\bib{Kipnis1999a}{book}{
   author={Kipnis, Claude},
   author={Landim, Claudio},
   title={Scaling limits of interacting particle systems},
   series={Grundlehren der Mathematischen Wissenschaften [Fundamental
   Principles of Mathematical Sciences]},
   volume={320},
   publisher={Springer-Verlag, Berlin},
   date={1999},
   pages={xvi+442},
   isbn={3-540-64913-1},
   review={\MR{1707314}},
   doi={10.1007/978-3-662-03752-2},
}

 
 
\bib{Kurtz73}{article}{
   author={Kurtz, Thomas G.},
   title={Convergence of sequences of semigroups of nonlinear operators with
   an application to gas kinetics},
   journal={Trans. Amer. Math. Soc.},
   volume={186},
   date={1973},
   pages={259--272 (1974)},
   issn={0002-9947},
   review={\MR{0336482}},
}


\bib{LSU68}{book}{
   author={Lady\v zenskaja, O. A.},
   author={Solonnikov, V. A.},
   author={Ural\cprime ceva, N. N.},
   title={Linear and quasilinear equations of parabolic type},
   language={Russian},
   series={Translated from the Russian by S. Smith. Translations of
   Mathematical Monographs, Vol. 23},
   publisher={American Mathematical Society, Providence, R.I.},
   date={1968},
   pages={xi+648},
   review={\MR{0241822}},
}



\bib{LT97}{article}{
   author={Lions, Pierre Louis},
   author={Toscani, Giuseppe},
   title={Diffusive limit for finite velocity Boltzmann kinetic models},
   journal={Rev. Mat. Iberoamericana},
   volume={13},
   date={1997},
   number={3},
   pages={473--513},
   issn={0213-2230},
   review={\MR{1617393}},
}

\bib{McKean75}{article}{
   author={McKean, H. P.},
   title={The central limit theorem for Carleman's equation},
   journal={Israel J. Math.},
   volume={21},
   date={1975},
   number={1},
   pages={54--92},
   issn={0021-2172},
   review={\MR{0423553}},
   doi={10.1007/BF02757134},
}

\bib{MM13}{article}{
   author={Mischler, St\'{e}phane},
   author={Mouhot, Cl\'{e}ment},
   title={Kac's program in kinetic theory},
   journal={Invent. Math.},
   volume={193},
   date={2013},
   number={1},
   pages={1--147},
   issn={0020-9910},
   review={\MR{3069113}},
   doi={10.1007/s00222-012-0422-3},
}

\bib{Spohn1991a}{book}{
	Author = {Spohn, H.},
	Publisher = {Springer-Verlag Berlin},
	Title = {Large scale dynamics of interacting particles},
	Volume = {174},
	Year = {1991}
}


\bib{Str84}{book}{
   author={Stroock, D. W.},
   title={An introduction to the theory of large deviations},
   series={Universitext},
   publisher={Springer-Verlag, New York},
   date={1984},
   pages={vii+196},
   isbn={0-387-96021-X},
   review={\MR{755154}},
   doi={10.1007/978-1-4613-8514-1},
}
	
 
 \bib{Uchi88}{article}{
   author={Uchiyama, K\=ohei},
   title={On the Boltzmann-Grad limit for the Broadwell model of the
   Boltzmann equation},
   journal={J. Statist. Phys.},
   volume={52},
   date={1988},
   number={1-2},
   pages={331--355},
   issn={0022-4715},
   review={\MR{968589}},
   doi={10.1007/BF01016418},
}
 
 \bib{Vaz07}{book}{
   author={V\'azquez, Juan Luis},
   title={The porous medium equation},
   series={Oxford Mathematical Monographs},
   note={Mathematical theory},
   publisher={The Clarendon Press, Oxford University Press, Oxford},
   date={2007},
   pages={xxii+624},
   isbn={978-0-19-856903-9},
   isbn={0-19-856903-3},
   review={\MR{2286292}},
}
 
 
 \bib{V03}{book}{
   author={Villani, C\'{e}dric},
   title={Topics in optimal transportation},
   series={Graduate Studies in Mathematics},
   volume={58},
   publisher={American Mathematical Society, Providence, RI},
   date={2003},
   pages={xvi+370},
   isbn={0-8218-3312-X},
   review={\MR{1964483}},
   doi={10.1007/b12016},
}

 
 
 
 \bib{V09}{book}{
   author={Villani, C\'{e}dric},
   title={Optimal transport},
   series={Grundlehren der Mathematischen Wissenschaften [Fundamental
   Principles of Mathematical Sciences]},
   volume={338},
   note={Old and new},
   publisher={Springer-Verlag, Berlin},
   date={2009},
   pages={xxii+973},
   isbn={978-3-540-71049-3},
   review={\MR{2459454}},
   doi={10.1007/978-3-540-71050-9},
}

\end{biblist}  
\end{bibdiv}
\end{document}